\documentclass[reqno, a4paper, 11pt]{amsart}
\usepackage{xpatch}

\xpatchcmd{\proof}
  {\itshape}
  {\bfseries}
  {}
  {}
  
\usepackage[margin=0.95in]{geometry}
\usepackage{amssymb, amsfonts,amsmath,amsthm,amsopn,amstext,amscd,mathtools}
\usepackage[colorlinks=true, citecolor=cyan, linkcolor=blue, urlcolor=cyan, filecolor=magenta]{hyperref}
\usepackage[color,matrix,arrow]{xy}

\xyoption{all}
\usepackage{tabularx}
\usepackage{latexsym}

\usepackage[only, llbracket,rrbracket]{stmaryrd}
\usepackage{todonotes}

\usepackage[shortlabels]{enumitem} 
\usepackage{bbm}
\usepackage{pdfpages}
\usepackage{titletoc}
\usepackage{booktabs}
\usepackage{mathtools}
\usepackage{mathrsfs}
\usepackage{thmtools}
\usepackage{hhline}
\usepackage{caption}
\usepackage{multirow, url}
\usepackage{textcomp}
\DeclareMathAlphabet{\cmcal}{OMS}{cmsy}{m}{n}
\usepackage{tcolorbox}
\usepackage{tikz-cd}

\usepackage{arydshln}
\usepackage{mlmodern}
\usepackage[utf8]{inputenc}
\usepackage[OT2, T1]{fontenc}

\DeclareMathOperator*{\smallcoprod}{\text{\raisebox{0.4ex}{\scalebox{0.6}{$\coprod$}}}}
\DeclareMathOperator*{\mcoprod}{\text{\raisebox{0.4ex}{\scalebox{0.9}{$\coprod$}}}}

\newtheorem{theorem}{Theorem}[section]
\newtheorem{corollary}[theorem]{Corollary}
\newtheorem{lemma}[theorem]{Lemma}

\newtheorem{conjecture}[theorem]{Conjecture}

\theoremstyle{definition}
\newtheorem{definition}[theorem]{Definition}

\theoremstyle{remark}
\newtheorem{remark}[theorem]{\bf{Remark}}

\newtheorem{notation}[theorem]{\bf{Notation}}

\numberwithin{equation}{section} \numberwithin{table}{subsection}
	
\newtheorem*{theorem*}{\bf{Theorem}}
\newtheorem*{claim*}{\bf{Claim}}
\newtheorem*{remark*}{\bf{Remark}}
\newtheorem*{remarks*}{\bf{Remarks}}
\newtheorem*{example*}{\bf{Example}}
\newtheorem*{examples*}{\bf{Examples}}

\newcommand{\C}{{\mathbf{C}}}

\newcommand{\Q}{{\mathbf{Q}}}

\newcommand{\Z}{{\mathbf{Z}}}

\newcommand{\fa}{{\mathfrak{a}}}
\newcommand{\fb}{{\mathfrak{b}}}

\newcommand{\fd}{{\mathfrak{d}}}

\newcommand{\fh}{{\mathfrak{h}}}

\newcommand{\fl}{{\mathfrak{l}}}
\newcommand{\fm}{{\mathfrak{m}}}

\newcommand{\fp}{{\mathfrak{p}}}

\newcommand{\cA}{{\cmcal{A}}}

\newcommand{\cD}{{\cmcal{D}}}

\newcommand{\cS}{{\cmcal{S}}}

\newcommand{\scA}{{\mathscr{A}}}
\newcommand{\scB}{{\mathscr{B}}}
\newcommand{\scC}{{\mathscr{C}}}

\newcommand{\scF}{{\mathscr{F}}}

\newcommand{\bbO}{\mathbb{O}}

\newcommand{\tn}{\textnormal}

\def\a{\alpha}
\def\b{\beta}

\def\d{\delta}
\def\e{\epsilon}
\def\ve{\varepsilon}
\def\g{\gamma}

\def\s{\sigma}
\def\z{\zeta}
\def\p{\varphi}

\def\G{\Gamma}

\def\U{\Upsilon}

\def\<{\left\langle}
\def\>{\right\rangle}
\def\gcd{\tn{gcd}}
\def\max{\tn{max}}

\newcommand{\vv}{\vspace{4mm}}
\newcommand{\svv}{\vspace{2mm}}

\newcommand{\zmod}[1]{{\Z/{#1}\Z}}

\newcommand{\arinj}{\ar@{^(->}}
\newcommand{\arsurj}{\ar@{->>}}
\newcommand{\arsub}{\ar@{}[r]|-*[@]{\subset}}
\newcommand{\arsup}{\ar@{}[r]|-*[@]{\supset}}
\newcommand{\arcap}{\ar@{}[d]|-*[@]{\subset}}
\newcommand{\arcup}{\ar@{}[u]|-*[@]{\subset}}
\newcommand{\arin}{\ar@{}[u]|-*[@]{\in}}

\renewcommand{\pmod}[1]{{\,(\tn{mod}\hspace{1mm} {#1})}}

\newcommand{\Gal}{{\tn{Gal}}}

\newcommand{\Pic}{{\tn{Pic}}}

\newcommand{\SL}{{\tn{SL}}}

\newcommand{\sHom}{{\mathscr{H}\kern-.5pt om}}
\newcommand{\sExt}{{\mathscr{E}\kern-.5pt xt}}

\newcommand{\new}{{\tn{new}}}
\newcommand{\tor}{{\tn{tors}}}

\newcommand{\rd}{\tn{red}}

\renewcommand{~}{\hspace*{0.5mm}}

\newcommand{\ov}{\overline}

\newcommand{\sm}{\smallsetminus}

\mathchardef\hyp="2D

\newcommand{\gauss}[1]{\left\lfloor #1 \right\rfloor}

\newcommand{\mm}[1]{(\zmod {#1})^\times}

\newcommand{\qa}{{\quad \text{and} \quad}}
\newcommand{\qqa}{{~ \text{ and } ~}}

\newcommand{\mat}[4]{
 \left(  \begin{smallmatrix} #1 & #2 \\ #3 & #4 \end{smallmatrix} \right)}

\newcommand{\vect}[2]{
 \left(  \begin{smallmatrix} #1 \\ #2 \end{smallmatrix} \right)}

\newcommand{\br}[1]{\langle #1 \rangle}

\newcommand{\leg}[2]{\genfrac(){}{}{#1}{#2}}

\newcommand{\hide}[1]{}

\newcommand{\imply}{\Longrightarrow}

\newcommand{\dd}{~|~}

\renewcommand{\emptyset}{\varnothing}

\begin{document}

\title{The rational cuspidal subgroup of $J_0(N)$}

\author{Hwajong Yoo}
\address{Hwajong Yoo, College of Liberal Studies and Research Institute of Mathematics, Seoul National University, Seoul 08826, South Korea}
\email{hwajong@snu.ac.kr}                                                                  

\author{Myungjun Yu}
\address{Myungjun Yu, Department of Mathematics, Yonsei University, Seoul 03722, South Korea}
\email{mjyu@yonsei.ac.kr}  
                                                           
\maketitle

\begin{abstract}
For a positive integer $N$, let $J_0(N)$ be the Jacobian of the modular curve $X_0(N)$.
In this paper we completely determine the structure of the rational cuspidal subgroup of $J_0(N)$ when the largest perfect square dividing $N$ is either an odd prime power or a product of two odd prime powers. Indeed, we prove that the rational cuspidal divisor class group of $X_0(N)$ is the whole rational cuspidal subgroup of $J_0(N)$ for such an $N$, and the structure of the former group is already determined by the first author in \cite{Y1}.
\end{abstract}
\setcounter{tocdepth}{1}
\tableofcontents


\section{Introduction}\label{section1}
Let $N$ be a positive integer and let 
\begin{equation*}
\G_0(N)=\left\{ \mat a b c d \in \SL_2(\Z) : c \equiv 0 \pmod N\right\}
\end{equation*}
be the congruence subgroup of $\SL_2(\Z)$. We denote by $X_0(N)$ Shimura's canonical model over $\Q$ of the modular curve corresponding to $\G_0(N)$ and let $J_0(N)=\Pic^0(X_0(N))$ be its Jacobian variety, which is an abelian variety over $\Q$. By the Mordell--Weil theorem, the group of all rational points on $J_0(N)$ is finitely generated and hence we have
\begin{equation*}
J_0(N)(\Q) \simeq \Z^{\oplus r} \oplus J_0(N)(\Q)_\tor,
\end{equation*}
where $r$ is a nonnegative integer and $J_0(N)(\Q)_\tor$ is a finite abelian group called the \textit{rational torsion subgroup}, which is important in the arithmetic theory of modular curves.

Suppose that $N$ is a prime number, in which case $X_0(N)$ has exactly two cusps $0$ and $\infty$. Ogg computed the order of $[0-\infty]$ (the equivalence class of the divisor $0-\infty$) in $J_0(N)$ \cite{Og73} and conjectured that the subgroup generated by this class is the full rational torsion subgroup of $J_0(N)$ \cite[Conj. 2]{Og75}.  Mazur then proved this conjecture (among others) in his seminal paper \cite{M77}.

Now, suppose that $N$ is not necessarily a prime.
Let $\scC_N$ be the \textit{cuspidal subgroup} of $J_0(N)$, which is defined as a subgroup of $J_0(N)(\ov{\Q})$ generated by the equivalence classes of the differences of any two cusps of $X_0(N)$.  By the theorem of Manin \cite[Cor. 3.6]{Ma72} and Drinfeld \cite{Dr73}, $\scC_N$ is a finite abelian group. As a naive generalization of Ogg's conjecture, one may assert that $\scC_N$ is the whole rational torsion subgroup. Unfortunately, for a nonsquarefree integer $N$ some cusps of $X_0(N)$ are not defined over $\Q$ and so $\scC_N$ may have a non-rational point. Hence the following may be regarded as a natural generalization of Ogg's conjecture (cf. Conjecture 2* on \cite[p. 243]{BM25}).

\begin{conjecture}[Generalized Ogg's conjecture] \label{conj: GOC}
For any positive integer $N$ we have
\begin{equation*}
J_0(N)(\Q)_\tor= \scC_N(\Q),
\end{equation*}
where 
\begin{equation*}
\scC_N(\Q)=\scC_N \cap J_0(N)(\Q)
\end{equation*}
is the \textnormal{rational cuspidal subgroup} of $J_0(N)$.
\end{conjecture}

 There are several partial results on Conjecture \ref{conj: GOC}. First of all, when $N$ is squarefree, Ohta proved that for a prime $p$, the $p$-primary parts of two groups are equal as long as $p$ does not divide $2\cdot\gcd(3, N)$ \cite{Oh14}. (Ribet and Wake gave another proof under an assumption that $p$ does not divide $6N$ using ``pure thought'' \cite{RW22}.) Next, when $N$ is not necessarily squarefree, Ren proved that when the $p$-primary part of $\scC_N(\Q)$ is trivial, so is that of $J_0(N)(\Q)_\tor$ under an extra assumption \cite{Ren18}. Lastly, the first author proved that for any positive integer $N$, the $p$-primary parts of two groups are equal as long as $p^2$ does not divide $12N$ (and some more cases) \cite{Y2}. 

Thus, it is worth computing the rational cuspidal subgroup $\scC_N(\Q)$ of $J_0(N)$. Although the cuspidal subgroup $\scC_N$ of $J_0(N)$ is a very explicit object, its size is not known in general. Also, the group $\scC_N(\Q)$ is still mysterious in most cases.
Motivated by a question posed by Ken Ribet, the first author proposed the following.
\begin{conjecture}\label{conj: Yoo}
For any positive integer $N$ we have
\begin{equation*}
\scC_N(\Q)=\scC(N),
\end{equation*}
where $\scC(N)$ is the \textnormal{rational cuspidal divisor class group} of $X_0(N)$.
\end{conjecture}
Here the \textit{rational cuspidal divisor class group} of $X_0(N)$ is defined as a subgroup of $J_0(N)(\ov{\Q})$ generated by the equivalence classes of \textit{rational cuspidal divisors} of degree $0$ on $X_0(N)$. Note that a divisor on $X_0(N)$ is called \textit{cuspidal} (resp. \textit{rational}) if its support lies only on the cusps of $X_0(N)$ (resp. it is fixed under the action of the absolute Galois group of $\Q$). 
The first author already determined the structure of $\scC(N)$ for any positive integer $N$ \cite{Y1}. So the computation of $\scC_N(\Q)$ reduces to proving Conjecture \ref{conj: Yoo}. 
 To the best of our knowledge, Conjecture \ref{conj: Yoo} is known in the following cases.
\begin{enumerate}
\item
If $N$ is small enough, it can be directly verified by computers.
\item
If $N=2^r M$ for $M$ odd squarefree and $0\leq r \leq 3$, then all cusps of $X_0(N)$ are defined over $\Q$ and so it trivially holds true.
\item
If $N=n^2 M$ for $n\dd 24$ and $M$ squarefree, then it is proved by Wang and Yang \cite[Th. 3]{WY20}.
\item
If $N=p^2 M$ for $M$ squarefree and $p$ prime, then it is proved in \cite[Th. 1.4]{GYYY}.
\end{enumerate}

This paper is a sequel to our previous joint work with J. Guo and Y. Yang \cite{GYYY}. The main theorem is the following.

\begin{theorem}\label{main theorem}
Let $N=ML^2$ for some squarefree integer $M$. If $L$ is either an odd prime power or a product of two odd prime powers, then we have
\begin{equation*}
\scC_N(\Q)=\scC(N).
\end{equation*}
\end{theorem}

\svv    
In the remaining part of the introduction, we sketch the proof of our main theorem.
To compute the rational cuspidal subgroup of $J_0(N)$, it is necessary to understand \textit{modular units} on $X_0(N)$, which are defined as modular functions on $X_0(N)$
that have zeros and poles only on cusps. A full description of modular units is given in \cite{GYYY} in terms of certain functions $F_{m, h}$, which are again constructed via generalized Dedekind eta functions $E_{g, h}$. For convenience of the readers, we recall their definition.
\begin{definition}
For integers $g$ and $h$ not both congruent to $0$ modulo $N$, let
\begin{equation*}
E_{g, h}(\tau):=q^{B_2(g/N)/2} \prod_{n=1}^\infty \left(1-\z_N^h q^{n-1+g/N}\right)\left(1-\z_N^{-h}q^{n-g/N}\right),
\end{equation*}
where $q=e^{2\pi i \tau}$ and $B_2(x)=x^2-x+1/6$ is the second Bernoulli polynomial.
\end{definition}

To introduce our functions $F_{m, h}$ we need some notation.
\begin{notation}\label{notation}
Let $N$ denote  a positive integer and let $p$ denote a prime.
\begin{itemize}[--]
\item
For a positive divisor $m$ of $N$, let $m':=N/m$ and 
$\ell(m)$ be the largest integer whose square divides $m'$. Also, let
\begin{equation*}
m'':=\frac{m'}{\ell(m)}=\frac{N}{m\ell(m)} \qa N':=\frac{N}{\ell(m)}.
\end{equation*}

\item
Let $L=\ell(1)$, i.e., $L$ is the largest integer whose square divides $N$.

\item
Let $\cD_N$ be the set of all positive divisors of $N$ different from $N$. 
\item
Let $\cD_N^*$ be a subset of $\cD_N$ consisting of divisors $m$ such that $\ell(m)>1$.

\item
For any positive integer $n$, let $\z_n:=e^{2 \pi i/n}$. 
\item
For any integer $s$ prime to $L$, let $\s_s$ be an element of $\Gal(\Q(\z_L)/\Q)$ such that 
\begin{equation*}
\s_s(\z_L)=\z_L^s.
\end{equation*}

\item
For any $m \in \cD_N$, let
\begin{equation*}
\begin{split}
I_m&:=\{ h\in \Z : 1\leq h \leq \ell(m)\} \qa \\
I_m^{\rd}&:=\{h \in I_m : 1\leq h \leq \ell(m)-\p(\ell(m))\},
\end{split}
\end{equation*}
where $\p$ is Euler's totient function. (Here, $\rd$ means ``redundant'', see Theorem \ref{thm: basis variant}.)

\item
Let 
\begin{equation*}
\begin{split}
\cS&:=\{(m, h) \in \cD_N \times \Z : 1 \leq h \leq \ell(m)\},\\
\cS^*&:=\{(m, h) \in \cD_N \times \Z : 0\leq h <\ell(m)\},\\
\cS_\fm^{\rd}&:=\{(\fm, h) \in \cS : h \in I_\fm^{\rd} \}, \\
\cS^{\rd}&:=\cup_{\fm \in \cD_N} \cS_\fm^{\rd} \qa \cS^{\new}:=\cS \sm \cS^{\rd}.
\end{split}
\end{equation*}

\item
Let $\scF$ denote the set of all set-theoretic maps from $\cS$ to $\Q$. We often regard elements of $\scF$ as the maps to $\Q/{\Z}$ after composing the natural surjection $\Q \to \Q/{\Z}$.

\item
For two complex-valued functions $f$ and $g$, we denote by  
\begin{equation*}
f \approx g
\end{equation*}
if they are equal up to constant, i.e., there exists a constant $c \in \C^*$ such that $f=cg$.

\item
For $x \in \Q$, let $v_p(x)$ be the $p$-adic valuation of $x$, i.e., $v_p(p^k a/b)=k$ if $(ab, p)=1$.
\end{itemize} 
\end{notation}

\svv    
We are ready to define our functions.
\begin{definition}\label{definition: Fmh}
For each $m\in \cD_N$, we fix a set $S_{m''} \subset \{1, 2, \dots, m''-1\}$ of representatives of 
$\mm {m''}/{\{\pm 1\}}$. For each $\a \in S_{m''}$, let $\d \in \{1, 2, \dots, m''-1\}$ be an integer such that $\a\d \equiv 1\pmod {m''}$. If $m'' \neq 2$, then we set
\begin{equation*}
F_{m, h}(\tau):=\prod_{\a \in S_{m''}} E_{\a m\ell(m), \d h N'}(N' \tau).
\end{equation*}
Also, if $m''=2$, then we define
\begin{equation*}
F_{m, h}(\tau):=E_{m\ell(m), hN'}(N'\tau)^{1/2}.
\end{equation*}
\end{definition}

As the function $F_{m, h}$ is dependent only on the residue class of $h$ modulo $\ell(m)$ \cite[Rem. 2.2]{GYYY}, any product made by $F_{m, h}$ is simply written as (up to constant)
\begin{equation*}
\prod_{m \in \cD_N} \prod_{h=1}^{\ell(m)} F_{m, h}^{a(m, h)}=\prod_{(m, h) \in \cS} F_{m, h}^{a(m, h)}
\end{equation*}
or 
\begin{equation*}
\prod_{m \in \cD_N} \prod_{h=0}^{\ell(m)-1} F_{m, h}^{b(m, h)}=\prod_{(m, h) \in \cS^*} F_{m, h}^{b(m, h)}.
\end{equation*}
In fact, some functions $F_{m, h}$ are redundant as the dimension of the space of the modular units on $X_0(N)$ is usually smaller than the number of possible pairs $(m, h)$, and we have the following, which is Theorem 1.8 of \cite{GYYY}.

\begin{theorem}\label{thm: basis}
Every modular unit $F$ on $X_0(N)$ can be uniquely expressed as
\begin{equation*}
F \approx \prod_{m \in \cD_N} \prod_{h=0}^{\p(\ell(m))-1} F_{m, h}^{a(m, h)} ~~~ \text{ for some } ~a(m, h) \in \Z.
\end{equation*}
\end{theorem}
By the same argument, we can easily prove the following.
\begin{theorem}\label{thm: basis variant}
Every modular unit $G$ on $X_0(N)$ can be uniquely expressed as
\begin{equation*}
G \approx \prod_{(m, h) \in \cS^\new} F_{m, h}^{b(m, h)} ~~~ \text{ for some } ~b(m, h) \in \Z.
\end{equation*}
\end{theorem}
\begin{proof}
For any integer $n$, 
\begin{equation*}
\{\z_{\ell(m)}^n, \z_{\ell(m)}^{n+1}, \dots, \z_{\ell(m)}^{n+\p(\ell(m))-1}\}
\end{equation*}
form an integral basis for $\Q(\z_{\ell(m)})$ over $\Q$. Hence the argument of the proof of Theorem \ref{thm: basis} applies verbatim.
\end{proof}

\begin{remark}\label{remark: rational becomes integral}
By Theorem \ref{thm: basis}, if
\begin{equation*}
F\approx \prod_{m \in \cD_N} \prod_{h=0}^{\p(\ell(m))-1} F_{m, h}^{a(m, h)}   \text{ for some } ~a(m, h) \in \Q
\end{equation*}
is a modular unit on $X_0(N)$, then we have 
$a(m, h) \in \Z$ for all $(m, h)$. Analogously, if 
\begin{equation*}
G\approx \prod_{(m, h) \in \cS^\new} F_{m, h}^{b(m, h)}   \text{ for some } ~b(m, h) \in \Q
\end{equation*}
is a modular unit on $X_0(N)$, then we also have $b(m, h) \in \Z$ for all $(m, h) \in \cS^{\new}$.
\end{remark}

Now, we are ready to illustrate our strategy for Conjecture \ref{conj: Yoo}, which already appeared in our previous joint work with Guo and Yang.
Let $D$ be a cuspidal divisor of degree $0$ on $X_0(N)$ such that $[D] \in \scC_N(\Q)$. Since all cusps of $X_0(N)$ are defined over $\Q(\z_L)$, so is $D$. 
Hence by definition, we have $[\s_s(D)]=[D]$ for any integer $s$ prime to $L$. 
Let $n$ be the order of $[D]$, which is finite by the theorem of Manin and Drinfeld. Then by definition, there is a modular unit $F$ such that $\tn{div}(F)=nD$. 
By Theorem \ref{thm: basis}, we can write 
\begin{equation*}
F \approx \prod_{m \in \cD_N} \prod_{h=0}^{\p(\ell(m))-1} F_{m, h}^{e(m, h)} \quad \text{ for some }~e(m, h) \in \Z.
\end{equation*}
To prove Conjecture \ref{conj: Yoo}, it suffices to show that there is another modular unit $g$ such that $D-\tn{div}(g)$ is a rational divisor. 
By Lemma \ref{lemma: Galois action}, for any integer $s$ prime to $L$ we have
\begin{equation*}
\s_s(\tn{div}(F_{m, 0}))=\tn{div}(F_{m, 0}),
\end{equation*}
i.e., $\tn{div}(F_{m, 0})$ is a rational divisor for any $m \in \cD_N$. Thus, if we formally set
\begin{equation}\label{eqn: intro1}
g=\prod_{m \in \cD_N} F_{m, 0}^{a(m)}\prod_{h=1}^{\p(\ell(m))-1} F_{m, h}^{e(m, h)/n} \quad \text{ for some }~ a(m) \in \Z,
\end{equation}
then it is straightforward that 
\[
D-\tn{div}(g)=\sum\nolimits_{m \in \cD_N} (e(m, 0)/n-a(m))\cdot \tn{div}(F_{m, 0})
\]
is rational. For such a formal product to be a modular unit, it is necessary that $e(m, h)$ is divisible by $n$ for all $m \in \cD_N^*$ and $h\neq 0$ (Remark \ref{remark: rational becomes integral}).
From this observation, we believe that our assumption $[\s_s(D)]=[D]$, which is equivalent to asserting that $(\s_s(F)/F)^{1/n}$ is a modular unit on $X_0(N)$, implies this divisibility.
Hence we propose the following.
\begin{conjecture}\label{conj: A}
Let $n\geq 2$ be an integer, and suppose that
\begin{equation*}
F=\prod_{m \in \cD_N} \prod_{h=0}^{\p(\ell(m))-1} F_{m, h}^{e(m, h)}~~~ \text{ with } ~e(m, h) \in \Z
\end{equation*}
is a modular unit on $X_0(N)$. Suppose all the following hold.
\begin{enumerate}
\item
The order of $F$ at any cusp is divisible by $n$.
\item
$(\s_s(F)/F)^{1/n}$ is a modular unit on $X_0(N)$ for any integer $s$ prime to $L$.
\end{enumerate}
Then $e(m, h)$ is divisible by $n$ for any $m \in \cD_N^*$ and $h\neq 0$.
\end{conjecture}

Indeed, we prove the following.
\begin{theorem}[Theorem \ref{theorem: rational modular units}]
Conjecture \ref{conj: Yoo} implies Conjecture \ref{conj: A}.
\end{theorem}

One of the main results of this paper is the following.
\begin{theorem}[Theorem \ref{theorem: B}]\label{theorem B}
If $L$ is odd, Conjecture \ref{conj: A} implies Conjecture \ref{conj: Yoo}.
\end{theorem}

To prove Theorem \ref{theorem B}, it is necessary to show that $g$ is indeed a modular unit. Thus, we need a criterion for a product in \eqref{eqn: intro1} to be a modular unit on $X_0(N)$, which is given as follows.

\begin{theorem}[Theorem \ref{theorem: criterion}]\label{theorem criterion}
Suppose that $L$ is odd and let
\begin{equation*}
g=\prod_{(m, h) \in \cS^*} F_{m, h}^{a(m, h)} \quad\text{ for some }~~ a(m, h) \in \Z.
\end{equation*}
Then $g$ is a modular unit on $X_0(N)$ if and only if all the following hold.
\begin{enumerate}
\item
The order of $g$ at a cusp $\infty$ is an integer.
\item
The order of $g$ at a cusp $0$ is an integer.
\item
The order of $g$ at a cusp $\vect 1 {N_0}$ is an integer, where $N_0$ is the odd part of $N$.
\item
\tn{(mod $L$ condition)} For any integer $1\leq i \leq N$ prime to $N$, we have
\begin{equation*}
\sum_{m \in \cD_N^*} \psi_i(m) \sum_{h=1}^{\ell(m)-1} h a(m, h) \equiv 0 \pmod L,
\end{equation*}
where 
\begin{equation*}
\psi_{i}(m) := \sum_{1\leq \a \leq m'', ~ (\a, m'')=1} \delta \gauss{\frac{\a i }{m''}}  \frac{L}{\ell(m)}.
\end{equation*}

\item
\tn{(mod $2$ condition)} For every odd prime divisor $p$ of $N$, we have
\begin{equation*}
\sum_{m: m''=p^r} \sum_{h=0}^{\ell(m)-1} a(m, h) \equiv 0 \pmod 2,
\end{equation*}
where the first sum runs over all $m \in \cD_N$ such that $m''$ is a power of $p$.
\end{enumerate}
\end{theorem}
Once this theorem is proved, the remaining part for Theorem \ref{theorem B} is taking suitable integers $a(m)$ in \eqref{eqn: intro1} so that $g$ becomes a modular unit. For more detail, see Section \ref{section2}.
 
\svv   
The most difficult part of this paper is proving Conjecture \ref{conj: A} when the number of the prime divisors of $L$ is at most two, which implies Theorem \ref{main theorem}. 
\begin{theorem}\label{theorem A}
If the number of the prime divisors of $L$ is at most two, then Conjecture \ref{conj: A} holds.
\end{theorem}

\svv    
In the rest of this section, we sketch our proof of Theorem \ref{theorem A}. 
From now on, we use the notation introduced in Notation \ref{notation}. To any $f \in \scF$, we associate a function of the form
\begin{equation*}
\prod_{(m, h) \in \cS} F_{m, h}^{f(m, h)}.
\end{equation*}
By this obvious association modular units on $X_0(N)$ can be regarded as elements of $\scF$ satisfying certain extra properties described in Theorem \ref{theorem criterion}. For any integer $h$, let
\begin{equation*}
E(m, h):=\begin{cases}
\frac{1}{n}e(m, h') & \text{ if } ~~1\leq h' < \p(\ell(m)),\\
\quad 0 & \text{ otherwise},
\end{cases}
\end{equation*}
where $h'$ is the remainder of $h$ modulo $\ell(m)$. 
For any integer $s$ prime to $L$, let $s^*$ be an integer such that $ss^*\equiv 1 \pmod L$.
We henceforth assume that $s$ and $s^*$ are such integers unless otherwise stated.
By Lemma \ref{lemma: Galois action}, we easily have
\begin{equation}\label{eqn: intro2}
(\s_{s^*}(F)/F)^{1/n} \approx \prod_{(m, h) \in \cS} F_{m, h}^{E_s(m, h)},
\end{equation}
where 
\begin{equation*}
E_s(m, h):=E(m, sh)-E(m, h).
\end{equation*}

\svv
To prove Conjecture \ref{conj: A}, it suffices to show that $E$ is an integer-valued function.
Ultimately, we will deduce this from certain properties of $E_s$. For instance, if $E_s$ is an integer-valued function for all $s$, then so is $E$.  As we assume that $(\s_{s^*}(F)/F)^{1/n}$ is a modular unit on $X_0(N)$, if $E_s(m, h)=0$ for all $(m, h) \in \cS^{\rd}$ then we can easily prove that $E_s$ is an integer-valued function (Remark \ref{remark: rational becomes integral}). Unfortunately, we cannot directly obtain that $E_s(m, h)=0$ for all $(m, h) \in \cS^{\rd}$. Instead, we first express \eqref{eqn: intro2} as the form in Theorem \ref{thm: basis variant}. 
As already mentioned above, there exist various relations among $F_{m, h}$ and we prove the following.
\begin{theorem}[Theorem \ref{theorem: relation among Fmh}]\label{theorem relation among Fmh}
Let $m\in \cD_N^*$ such that $\ell(m)$ is divisible by a prime $p$. Then we have
\begin{equation*}
\prod_{j=0}^{p-1} F_{m, h+xj} \approx F_{mp, ~p^{\ve}h},
\end{equation*}
where $x=\ell(m)/p$ and $\ve=1+v_p(\ell(mp))-v_p(\ell(m))$.
\end{theorem}

It seems that the relations in Theorem \ref{theorem relation among Fmh} generate all relations among $F_{m, h}$. So we may replace $F_{m, h}$ for all $(m, h) \in \cS^{\rd}$ by others and then obtain a function of the form in Theorem \ref{thm: basis variant}. In other words, it seems possible to construct a  map $\Psi : \scF \to \scF$ such that
$\Psi(a)(m, h)=0$ for all $(m, h) \in \cS^{\rd}$ and 
\begin{equation*}
\prod_{(m, h) \in \cS} F_{m, h}^{a(m, h)} \approx \prod_{(m, h) \in \cS^{\new}} F_{m, h}^{\Psi(a)(m, h)}.
\end{equation*}
The construction of such a map $\Psi$ is very complicated in general, so we restrict ourselves to the case where the number of the prime divisors of $L$ is at most two.
By \eqref{eqn: intro2} and Remark \ref{remark: rational becomes integral}, we then obtain the following property of $E_s$ (Corollary \ref{corollary: Psi(gs)=0}). 
\begin{equation}\label{star}
\text{For any $s$ prime to $L$, } ~\Psi(E_s)(m, h) \in \Z \quad \text{ for all } ~(m, h) \in \cS.
\end{equation}

\begin{remark}\label{remark: Q to Q/Z}
In fact, we prove Theorem \ref{theorem A} only using \eqref{star}. Since we are only interested in whether $E(m, h) \in \Z$ or not, it is useful to compose the natural surjective morphism $\Q \to \Q/\Z$ ``everywhere'' and try to prove that $E\equiv 0$, i.e., $E$ is identical to the zero function. Thus, we will prove the following.
\begin{equation}\label{eqn: final goal}
\text{For any $s$ prime to $L$, } ~\Psi(E_s) \equiv 0 \quad  \imply \quad E \equiv 0.
\end{equation}
\end{remark}

\svv  
Finally, we sketch our proof of \eqref{eqn: final goal} when $N=p^r$ is an odd prime power.  By definition, we have $L=p^t$, where $t=\gauss{r/2}$. If $t=0$, there is nothing to prove and so we assume that $t\geq 1$.
If we take $s=1+p^{t-1}k$ for some integer $k$ satisfying $(k, p)=(s, p)=1$, then $E_s(m, h)=0$ unless $\ell(m)=p^t$ and $(h, p)=1$. (Note that $(s-1)h \equiv 0 \pmod {\ell(m)}$ otherwise.) 
This motivates the following definition.
\begin{definition}
For a positive integer $a$, let
\begin{equation*}
\cA(a):=\{(m, h) \in \cS : v_p(h)=v_p(\ell(m))-a\} \qa \cA^+(a):=\cup_{n>a} \cA(n).
\end{equation*}
\end{definition}
The following is a special case of Theorem \ref{theorem: vanishing 1}.
\begin{theorem}
Let $a\in \Z_{\geq 1}$. Suppose that all the following hold.
\begin{enumerate}
\item
$\Psi(f)(m, h)=0$ for all $(m, h) \in \cS$.
\item
$f(m, h)=0$ for all $(m, h) \in \cS \sm  \cA (a)$.
\end{enumerate}
Then $f(m, h)=0$ for all $(m, h) \in \cS$.
\end{theorem}

As above, we first take $s=1+p^{t-1}k$ for some integer $k$ satisfying $(k, p)=(s, p)=1$. Then $f=E_s$ satisfies two assumptions with $a=t\geq 1$.
Thus, we have $E_s \equiv 0$. For each $1\leq h < p^{t-1}(p-1)$ with $(h, p)=1$, we can always find $k$ such that the remainder of $sh$ modulo $p^t$ is at least $p^{t-1}(p-1)$, and for such a $k$ we have $E(1, sh)=0$ by definition. Since $E_s(1, h)=0$ as well, we have $E(1, h)=0$ as desired. If $r=2t+1$, then by the same argument we have $E(p, h)=0$ as well. In other words, $E(m, h)=0$ for all $(m, h) \in \cA(t)=\cA^+(t-1)$. 

If $t=1$, we are done. Suppose that $t\geq 2$. Since $E(m, h)=0$ for all $(m, h) \in \cA^+(t-1)$, we easily have $E_s(m, h)=0$ for all $(m, h) \in \cA^+(t-1)$ as well. As above, we take $s=1+p^{t-2}k$ for some $k$ satisfying $(k, p)=(s, p)=1$.
Then $f=E_s$ satisfies two assumptions with $a=t-1\geq 1$. By the same argument as above, we can prove that
$E(m, h)=0$ for all $(m, h) \in \cA(t-1)$. Combining above, we have $E(m, h)=0$ for all $(m, h) \in \cA^+(t-2)$.
By doing this successively, we complete the proof. 

When $L$ is a product of two prime powers, then we were not able to prove such a vanishing theorem. Instead, we prove a weaker statement in Section \ref{section6} and deduce \eqref{eqn: final goal} from a more delicate inductive argument in Section \ref{section7}.

\svv   
\subsection*{Acknowledgements}
The authors would like to thank Kenneth Ribet and Yifan Yang for their inspiring discussions. 
H.Y. was supported by National Research Foundation of Korea(NRF) grant funded by the Korea government(MSIT) (No. RS-2023-00239918  and No. 2020R1A5A1016126). M.Y. was supported by Yonsei University Research Fund (2024-22-0146) and by Korea Institute for Advanced Study (KIAS) grant funded by the Korea government.

\svv    
\section{Modular units}\label{section2}
In this section, we elaborate the methods in the paper \cite{GYYY} and prove some new results.

\svv  
Before proceeding, we recall a description of the cusps of $X_0(N)$ (cf. \cite[Sec. 2]{Y1}). A cusp of $X_0(N)$ can be written as $\vect a c$  for some positive divisor $c$ of $N$ and an integer $a$ prime to $c$. In this notation, the cusps  $0$ and $\infty$ are written as $\vect 1 1$ and $\vect 1 N$, respectively.
Two cusps $\vect x d$ and $\vect y \d$ with two positive divisors $d$ and $\d$ of $N$ are equivalent if and only if $d=\d$ and $x \equiv y \pmod z$, where $z=(d, N/d)$. Moreover, a cusp $\vect x d$ is defined over $\Q(\z_z)$. Hence all the cusps of $X_0(N)$ are defined over $\Q(\z_L)$.
For unfamiliar notation, see Notation \ref{notation}.

\svv 
We recall some results of \cite{GYYY}, which are frequently used in the paper.
\begin{lemma}\label{lemma: Egh}
The functions $E_{g, h}$ satisfy
\begin{equation*}
E_{g+N, h}=E_{-g, -h}=-\z_N^{-h} E_{g, h} \qa E_{g, h+N}=E_{g, h}.
\end{equation*}
\end{lemma}
\begin{proof}
This is a part of \cite[Prop. 2.1]{GYYY}, which is \cite[Th. 1]{Ya04}.
\end{proof}

\begin{lemma}\label{lemma: order of Fmh}
Let $m\in \cD_N$. Then for any integer $h$, the order of $F_{m, h}$ at a cusp $\vect a c$ of $X_0(N)$ is
\begin{equation*}
\frac{\ell(m)(N', c)^2}{4(c^2, N)} \sum_{\a \in \mm {m''}} P_2 \left(\frac{\a a'}{m''}+\frac{\d h c'}{\ell(m)}\right),
\end{equation*}
where $P_2(x)=B_2(\{x\})$ is the second Bernoulli function with $\{x\}=x-
\gauss{x}$, $a'=\frac{N'a}{(N', c)}$ and $c'=\frac{c}{(N', c)}$.
\end{lemma}
\begin{proof}
This is \cite[Lem. 2.6]{GYYY}.
\end{proof}

\begin{lemma}\label{lemma: order of Fmh at special cusps}
The orders of $F_{m, h}$ at the cusps $\infty$ and $0$ are
\begin{equation*}
\frac{m}{24}\prod_{p \mid m''} (1-p) \qa \frac{m''}{24\ell(m)} \sum_{k \mid m''} \frac{\mu(k)}{k}(\ell(m), kh)^2,
\end{equation*}
respectively, where $\mu(k)$ is the M\"obius function. Also, if $N=2N_0$ for some odd integer $N_0$, then the order of $F_{m, h}$ at the cusp $\vect {1}{N_0}$ is
\begin{equation*}
\frac{m}{24(m, 2)}\prod_{p  \mid m'', ~p\neq 2} (1-p).
\end{equation*}
\end{lemma}
\begin{proof}
This is \cite[Cor. 2.7]{GYYY}.
\end{proof}

\begin{lemma}\label{lemma: Galois action}
For any integer $s$ prime to $L$, we have
\begin{equation*}
\s_{s}(\tn{div}(F_{m, h}))=\tn{div}(F_{m, ~sh}).
\end{equation*}
\end{lemma}
\begin{proof}
This is \cite[Lem. 2.9]{GYYY}.
\end{proof}

\begin{lemma}\label{lemma: Fm0 = eta product}
For any $m \in \cD_N$, we have
\begin{equation*}
F_{m, 0}(\tau) \approx \prod_{k \mid m''} \eta(km \tau)^{\mu(k)}=\prod_{k \mid m'} \eta(km \tau)^{\mu(k)},
\end{equation*}
where $\eta$ is the Dedekind eta function.
\end{lemma}
\begin{proof}
The first equality (up to constant) is \cite[Rmk. 2.12]{GYYY}. The second one easily follows by two facts: one is that $\mu(k)=0$ if $k$ is not squarefree; and the other is that
the set of all squarefree divisors of $m''$ coincides with that of $m'=N/m=m'' \ell(m)$.
\end{proof}

\begin{lemma}\label{lemma: root of unity from transformation}
Let $m \in \cD_N$ and $\g=\mat a b c d \in \SL_2(\Z)$ with $24 \dd b$, $24N \dd c$ and $c>0$.
If $N=2N_0$, then we have
\begin{equation*}
F_{N_0, 0}(\g \tau)=(-1)^{\frac{a^2-1}{8}} F_{N_0, 0}(\tau).
\end{equation*}
If $m''\neq 2$, then we have
\begin{equation*}
F_{m, h}(\g \tau)=(-1)^{\frac{(d-1)\p(m'')}{4}} \times \z_{\ell(m)}^{B_{m, h}} \prod_{\a \in S_{m''}} E_{\a a m \ell(m), \d d h N'}(N' \tau),
\end{equation*}
where
\begin{equation*}
B_{m, h}:=\frac{h(1-d)}{2} \sum\nolimits_{\a \in S_{m''}} \d.
\end{equation*}
\end{lemma}
\begin{proof}
The first assertion follows by \cite[Rmk. 2.5]{GYYY}, and the second is \cite[Lem. 2.14]{GYYY}.
\end{proof}

\begin{lemma}\label{lemma: sum with gauss}
Let $m \in \cD_N$ with $m'' \neq 2$ and $L$ odd. For any odd integer $a$ prime to $m''$, we have
\begin{equation*}
\prod_{\a \in S_{m''}} (-1)^{\gauss{\a a/{m''}}} = \begin{cases}
(-1)^{\frac{(a-1)\p(m'')}{4}} & \text{ if $m''$ is even},\\
\leg a p & \text{ if $m''=p^k$ for some odd prime $p$},\\
1 & \text{ otherwise}.
\end{cases}
\end{equation*}
\end{lemma}
\begin{proof}
This is \cite[Lem. 3.3]{GYYY}.
\end{proof}

\begin{lemma}\label{lemma: root of unity for G'}
Let $G$ be the subgroup of $\G_0(N)$ generated by $\mat 1 1 0 1$ and $\mat 1 0 N 1$.
\begin{enumerate}[(i)]
\item
Assume that 
\begin{equation*}
g(\tau)=\prod_{(m, h) \in \cS^*} F_{m, h}^{e(m, h)}~~\text{ for some } ~e(m, h) \in \Z
\end{equation*}
is a product such that its orders at the cusps $0$ and $\infty$ are integers. Then for any element $\g$ of $G_0(N)$, the value of the root of unity $\e$ in $g(\g \tau)=\e g(\tau)$ depends only on the right coset $G \g$ of $\g$ in $\G_0(N)$.
\item
Every right coset of $G$ in $\G_0(N)$ contains an element $\mat a b c d$ such that $24 \dd b$, $24N \dd c$ and $c>0$.
\end{enumerate}
\end{lemma}
\begin{proof}
This is \cite[Lem. 3.2]{GYYY}.
\end{proof}

\svv  
The following is a generalization of \cite[Th 1.10]{GYYY}, which is again a generalization of a famous criterion of Ligozat \cite[Prop. 3.2.1]{Li75}.
\begin{theorem}\label{theorem: criterion}
Suppose that $L$ is odd and let
\begin{equation*}
g=\prod_{(m, h) \in \cS^*} F_{m, h}^{a(m, h)} \quad\text{ for some }~~ a(m, h) \in \Z.
\end{equation*}
Then $g$ is a modular unit on $X_0(N)$ if and only if all the following hold.
\begin{enumerate}
\item
The order of $g$ at a cusp $\infty$ is an integer.
\item
The order of $g$ at a cusp $0$ is an integer.
\item
The order of $g$ at a cusp $\vect 1 {N_0}$ is an integer, where $N_0$ is the odd part of $N$.
\item
\tn{(mod $L$ condition)} For any integer $1\leq i \leq N$ prime to $N$, we have
\begin{equation*}
\sum_{m \in \cD_N^*} \psi_i(m) \sum_{h=1}^{\ell(m)-1} h a(m, h) \equiv 0 \pmod L,
\end{equation*}
where 
\begin{equation*}
\psi_{i}(m) := \sum_{1\leq \a \leq m'', ~ (\a, m'')=1} \delta \gauss{\frac{\a i }{m''}}  \frac{L}{\ell(m)}.
\end{equation*}

\item
\tn{(mod $2$ condition)} For every odd prime divisor $p$ of $N$, we have
\begin{equation*}
\sum_{m: m''=p^r} \sum_{h=0}^{\ell(m)-1} a(m, h) \equiv 0 \pmod 2,
\end{equation*}
where the first sum runs over all $m \in \cD_N$ such that $m''$ is a power of $p$.
\end{enumerate}
\end{theorem}
\begin{proof}
We closely follow the proofs of Theorem 1.10 and Proposition 3.4 of \cite{GYYY}.
We only prove the theorem assuming that $N= 2N_0$. It follows similarly when $N$ is odd, which is in fact much simpler. 
Note first that since $N_0$ is odd, 
\begin{equation*}
m''=2 \iff m=N_0.
\end{equation*}
For simplicity, for $k \in \{0, 1\}$ let
\begin{equation*}
\cD_k:=\{ m \in \cD_N : m''=2^k p^{r} ~~~\text{ for some }  p \equiv 3 \pmod 4 \text{ and an integer }~r\geq 1\}.
\end{equation*}
Since $N$ is not divisible by $4$, $m$ is even if and only if $m''$ is odd. 
Also, for any $m\in \cD_N$ different from $N_0$, we have $\p(m'') \equiv 0 \pmod 2$; and moreover
\begin{equation}\label{eqn: 2.1}
\prod_{p \mid m''}(1-p) \equiv 2 \pmod 4 \iff \p(m'') \equiv 2 \pmod 4 \iff m \in \cD_0 \cup \cD_1.
\end{equation}

First, suppose that $g$ is a modular unit on $X_0(N)$.
Then it is clear that conditions (1), (2) and (3) hold. Let
\begin{equation*}
\G':=\left\{ \mat a b c d \in \G_0(N) : 24 \dd b,  ~24N \dd c \qqa c>0\right\}.
\end{equation*}
Suppose that $\g=\mat a b c d \in \G'$, so in particular, $a$ and $d$ are both odd. 
By Lemma \ref{lemma: order of Fmh at special cusps} and \eqref{eqn: 2.1}, condition (1) implies that 
\begin{equation*}
a(N_0, 0) \equiv 0 \pmod 2 \qa \frac{a(N_0, 0)}{2}+\sum_{m \in \cD_1} \sum_{h=0}^{\ell(m)-1} a(m, h) \equiv 0 \pmod 2,
\end{equation*}
and condition (3) implies that
\begin{equation*}
a(N_0, 0) \equiv 0 \pmod 2  \qa \frac{a(N_0, 0)}{2}+\sum_{m \in \cD_0 \cup \cD_1} \sum_{h=0}^{\ell(m)-1} a(m, h) \equiv 0 \pmod 2.
\end{equation*}
Combining these, we get 
\begin{equation}\label{eqn: 2.2}
a(N_0, 0) \equiv 0 \pmod 2  \qa \sum_{m \in \cD_0} \sum_{h=0}^{\ell(m)-1} a(m, h) \equiv 0 \pmod 2.
\end{equation}
By Lemma \ref{lemma: Egh}, we have
\begin{equation*}
E_{\a a m \ell(m), \d d h N'}(N'\tau)=(-\z_N^{-\d dh N'})^{\gauss{\a a/{m''}}}E_{\a^* m\ell(m), \d^* h N'}(N'\tau),\\
\end{equation*}
where for $\a \in S_{m''}$, we let $\a^*$ be the element in $S_{m''}$ such that $\a a \equiv \a^* \pmod {m''}$ or $\a a \equiv -\a^* \pmod {m''}$, and $\d^*\in \Z$ satisfies $\a^* \d^* \equiv 1 \pmod {m''}$. Thus by Lemma \ref{lemma: root of unity from transformation}, for any $m \in \cD_N$ different from $N_0$, we have
\begin{equation*}
\begin{split}
F_{m, h}(\g \tau)&=(-1)^{\frac{(d-1)\p(m'')}{4}} \times \z_{\ell(m)}^{B_{m, h}} \prod_{\a \in S_{m''}} (-\z_N^{-\d dh N'})^{\gauss{\a a/{m''}}}E_{\a^* m\ell(m), \d^* h N'}(N'\tau)\\
&=(-1)^{\frac{(d-1)\p(m'')}{4}} \times \z_{\ell(m)}^{B_{m, h}}  \left(\prod_{\a \in S_{m''}} (-\z_{\ell(m)}^{-\d dh})^{\gauss{\a a/{m''}}} \right) F_{m, h}(\tau).\\
\end{split}
\end{equation*}

Let 
\begin{equation*}
\e_{1, \g}:=\prod_{m \in \cD_N, ~m\neq N_0} \prod_{h=0}^{\ell(m)-1} \left((-1)^{\frac{(a-1)\p(m'')}{4}}\right)^{a(m, h)},
\end{equation*}
\begin{equation*}
\e_{2, \g}:=\prod_{m \in \cD_N, ~m\neq N_0} \prod_{h=0}^{\ell(m)-1} \left(\prod_{\a \in S_{m''}} (-1)^{\gauss{\a a/{m''}}} \right)^{a(m, h)},
\end{equation*}
and 
\begin{equation*}
\mu_\g:=\sum_{m \in \cD_N^*} \left(\frac{1}{\ell(m)}\sum_{\alpha \in S_{m''}} \xi_m(\a) \sum_{h=1}^{\ell(m)-1} ha(m, h)\right),
\end{equation*}
where
\begin{equation*}
\xi_m(\a):=\d d \gauss{\frac{\a a}{m''}}+\frac{\d(d-1)}{2}.
\end{equation*}
Since $a\equiv d \equiv \pm 1 \pmod 4$ and $a(N_0, 0) \equiv 0 \pmod 2$, by Lemma \ref{lemma: root of unity from transformation} we have
\begin{equation*}
g(\g \tau)=\e_{1, \gamma} \e_{2, \gamma} \cdot e^{-2\pi i \mu_{\gamma}} g(\tau).
\end{equation*}
(If $\ell(m)=1$, then it is not necessary to consider the term involving $\z_{\ell(m)}$.) 
By \eqref{eqn: 2.1} and Lemma \ref{lemma: sum with gauss}, we easily have
\begin{equation*}
\e_{1, \g}=\prod_{m \in \cD_0 \cup \cD_1} \prod_{h=0}^{\ell(m)-1} (-1)^{\frac{(a-1)a(m, h)}{2}}
\end{equation*}
and
\begin{equation*}
\e_{2, \g}=\prod_{m \in \cD_1} \prod_{h=0}^{\ell(m)-1} (-1)^{\frac{(a-1)a(m, h)}{2}} \times \prod_{m:m''=p^{r}, p\neq 2} \prod_{h=0}^{\ell(m)-1} \leg a p^{a(m, h)}.
\end{equation*}
Thus, we have $\e_{1, \g}\e_{2, \g}=\e'_{1, \g} \e'_{2, \g}$, where
\begin{equation*}
\e'_{1, \g}:=\prod_{m \in \cD_0} \prod_{h=0}^{\ell(m)-1}(-1)^{\frac{(a-1)a(m, h)}{2}} \qa \e'_{2, \g}:=\prod_{m:m''=p^{r}, p\neq 2} \prod_{h=0}^{\ell(m)-1} \leg a p^{a(m, h)}.
\end{equation*}

Note that $e^{-2\pi i \mu_\g}$ is an $L$-th root of unity. Since $g$ is a modular unit and $L$ is odd, we must have $\e'_{1, \gamma} \e'_{2, \gamma} =1$ and $\mu_\g$ is an integer.
Note that \eqref{eqn: 2.2} implies that $\e'_{1, \gamma} = 1$, and hence $\e'_{2, \gamma}=1$. 
Furthermore, for each prime divisor $p$ of $N$, there is always an integer $a_p$ such that $(a_p, 6N)=1$, $\leg {a_p} p=-1$ and $\leg {a_p} q=1$ for all odd prime divisors $q$ of $N$ different from $p$. Thus, $\e'_{2, \g}=1$ for all $\g \in \G'$ if and only if condition (5) is satisfied. 

Now, we claim that condition (4) holds from the fact that $\mu_{\gamma}\in \Z$ for all $\gamma \in \Gamma'$. First, we prove that 
\begin{equation*}
\xi_m(\a) \equiv \xi_m(m''-\a) \pmod {m''}.
\end{equation*}
Indeed, since $d$ is odd, we have
\begin{equation*}
\xi_m(\a) - \xi_m(m''-\a) \equiv \d d \left(\gauss{\frac{\a a}{m''}}+\gauss{\frac{(m''-\a)a}{m''}}\right)+\d (d-1) \pmod {m''}.
\end{equation*}
Since $\a a$ is not divisible by $m''$, we have
\begin{equation*}
\gauss{\frac{\a a}{m''}}+\gauss{\frac{(m''-\a)a}{m''}}=a-1
\end{equation*}
and so the assertion follows from the fact that $ad \equiv 1 \pmod N$. Thus, we have
\begin{equation*}
2\sum_{\a \in S_{m''}} \xi_m(\a) \equiv \sum_{1\leq \a \leq m'', ~(\a, m'')=1} \xi_m(\a) \pmod {\ell(m)}
\end{equation*}
as $\ell(m)$ divides $m''$. Also, we have
\begin{equation*}
\sum_{1\leq \a\leq m'', ~ (\a, m'')=1} \d=\sum_{\a \in S_{m''}} (\d+(m''-\d))=\frac{m''\p(m'')}{2} \equiv 0 \pmod {\ell(m)}.
\end{equation*}
Hence we have
\begin{equation*}
\begin{split}
2\mu_\g &\equiv \sum_{m \in \cD_N^*} \left(\frac{1}{\ell(m)}\sum_{1\leq \a\leq m'', ~(\a, m'')=1} \xi_m(\a)\right) \sum_{h=1}^{\ell(m)-1} ha(m, h) \\
&\equiv \frac{d}{L}\sum_{m \in \cD_N^*} \left(\sum_{1\leq \a\leq m'', ~(\a, m'')=1} \delta \gauss{\frac{\a a}{m''}} \frac{L}{\ell(m)}\right) \sum_{h=1}^{\ell(m)-1} ha(m, h)  ~~ \pmod \Z.
\end{split}
\end{equation*}
Since $L$ is odd and $(d, N)=1=(d, L)$, $\mu_\g$ is an integer if and only if 
\begin{equation*}
\sum_{m \in \cD_N^*} \psi_a(m) \sum_{h=1}^{\ell(m)-1} h a(m, h) \equiv 0 \pmod L.
\end{equation*}
Note that we have $\psi_a(m) \equiv \psi_{a+N}(m) \pmod {L}$ as $N=mm'' \ell(m)$. Also, for any integer $1\leq i \leq N$ prime to $N$, we can always find $\g=\mat a b c d \in \G'$ such that $i\equiv a \pmod N$. (We need to take $a$ such that $(a, 6)=1$.) Thus, condition (4) is satisfied.

Conversely, suppose all the conditions hold. As above, for any $\g \in \G'$ we have
\begin{equation*}
g(\g \tau)=\e'_{1, \gamma} \e'_{2, \gamma} \cdot e^{-2\pi i \mu_{\gamma}} g(\tau).
\end{equation*}
Since $N=mm'' \ell(m)$, we have $\psi_a(m) \equiv \psi_{a+N}(m) \pmod {L}$. Thus, condition (4) implies that $\mu_\g \in \Z$. Also, conditions (1) and (3) imply \eqref{eqn: 2.2}, and hence $\e'_{1, \g}=1$. Furthermore, condition (5) implies $\e'_{2, \g}=1$. 
Therefore $g(\g\tau)=g(\tau)$ for any $\g \in \G'$. 
Finally, by Lemma \ref{lemma: root of unity for G'} we have $g(\g \tau)=g(\tau)$ for any $\g \in \G_0(N)$, i.e., $g$ is a modular unit on $X_0(N)$. This completes the proof. 
\end{proof}

\svv    
As a result, we obtain the following.
\begin{theorem}\label{theorem: rational modular units}
Conjecture \ref{conj: Yoo} implies Conjecture \ref{conj: A}.
\end{theorem}
\begin{proof}
First, let $\eta_d(\tau):=\eta(d\tau)$, where $\eta$ is the Dedekind eta function. 
For simplicity, let
\begin{equation*}
\cD(d):=\{ m \in \cD_N : d \dd m\}=\{ dn  : 1\leq n \dd d'\} \sm \{N\},
\end{equation*}
where $d'=N/d$.
Then by Lemma \ref{lemma: Fm0 = eta product}, for any $d\in \cD_N$ we have
\begin{equation}\label{eqn: Fmh=eta}
\prod_{m \in \cD(d)} F_{m, 0} \approx \prod_{m\in \cD(d)} \prod_{k\mid m'} \eta_{km}^{\mu(k)}=\frac{\prod_{n \mid d'} \prod_{k \mid \frac{N}{dn}}\eta_{kdn}^{\mu(k)}}{\eta_N}=\frac{\prod_{d\mid s \mid N} \prod_{k \mid \frac{s}{d}} \eta_s^{\mu(k)}}{\eta_N}=\frac{\eta_d}{\eta_N}.
\end{equation}

Next, we claim that any modular unit $F'$ on $X_0(N)$ whose divisor is rational can be written as a product of (integral powers of) $F_{m, 0}$ up to constant. Since $\tn{div}(F')$ is rational, by \cite[Th. 3.6]{Y1} there are integers $b(d)$ such that $F' \approx \prod_{d\mid N} \eta_d^{b(d)}$.
Since $F'$ is a modular unit on $X_0(N)$, we have $\sum_{d\dd N} b(d)=0$ by condition (3) of Proposition 3.5 of \textit{op. cit.}
Thus, by \eqref{eqn: Fmh=eta} we have 
\begin{equation*}
\begin{split}
F' &\approx \prod_{d\mid N} \eta_d^{b(d)} \approx \prod_{d\mid N}  \eta_N^{b(d)} \prod_{d\mid N, d\neq N}(\prod_{m \in \cD(d)} F_{m, 0})^{b(d)}=\eta_N^{\sum_{d \mid N} b(d)}\prod_{d \in \cD_N} (\prod_{m \in \cD(d)} F_{m, 0})^{b(d)}\\
&= \prod_{d \in \cD_N} (\prod_{m \in \cD(d)} F_{m, 0})^{b(d)}=\prod_{m \in \cD_N} (\prod_{k \mid m} F_{m, 0}^{b(k)})=\prod_{m \in \cD_N} F_{m, 0}^{\sum_{k\mid m} b(k)}.
\end{split}
\end{equation*}
This proves the claim.

Finally, let $F$ be a modular unit as in Conjecture \ref{conj: A}.
By assumption (1), there is a cuspidal divisor $D$ on $X_0(N)$ such that $\tn{div}(F)=nD$. Also by assumption (2), $[\s_s(D)]=[D]$ for any integer $s$ prime to $L$. Namely, we have $[D] \in \scC_N(\Q)$. 
Thus, there is a modular unit $g$ such that $D'=D-\tn{div}(g)$ is rational as $[D] \in \scC(N)$ by our assumption that Conjecture \ref{conj: Yoo} holds. 
Since the order of $[D']=[D]$ divides $n$, there is a modular unit $G$ such that
$\tn{div}(G)=nD'$. Moreover, since $D'$ is rational, by the claim above there are integers $a(m)$ such that $G \approx  \prod_{m \in \cD_N} F_{m, 0}^{a(m)}$.
Note that by Theorem \ref{thm: basis}, there are integers $b(m, h)$ such that
$g \approx \prod_{m\in \cD_N} \prod_{h=0}^{\p(\ell(m))-1} F_{m, h}^{b(m, h)}$ as $g$ is also a modular unit on $X_0(N)$. 
Since $n \cdot \tn{div}(g)=nD-nD'=\tn{div}(F/G)$, we have
\begin{equation*}
 \prod_{m \in \cD_N} F_{m, 0}^{e(m, 0)-a(m)-nb(m, 0)}\prod_{h=1}^{\p(\ell(m))-1} F_{m, h}^{e(m, h)-nb(m, h)} \approx 1.
\end{equation*}
Thus, by Theorem \ref{thm: basis} we have $e(m, h)=nb(m, h)$ for all $m \in \cD_N^*$ and $h\neq 0$. This completes the proof. 
\end{proof}

\svv   
We believe that the converse of the above implication also holds, 
and we prove it under a mild assumption using a slight generalization of the proof of Theorem 1.4 in \cite{GYYY}.
\begin{theorem}\label{theorem: B}
If $L$ is odd, Conjecture \ref{conj: A} implies Conjecture \ref{conj: Yoo}.
\end{theorem}
\begin{proof}
Let $D$ be a cuspidal divisor of degree $0$ on $X_0(N)$ such that $[D] \in \scC_N(\Q)$. Let $n$ be the order of $[D]$. 
If $n=1$, then $[D]=0 \in \scC(N)$ and so there is nothing to prove. Thus, we assume that $n\geq 2$.
By definition, there is a modular unit $F$ such that $\tn{div}(F)=nD$. By Theorem \ref{thm: basis}, we can write
\begin{equation*}
F\approx \prod_{m \in \cD_N} \prod_{h=0}^{\p(\ell(m))-1} F_{m, h}^{e(m, h)} \quad \text{ for some }~ e(m, h) \in \Z.
\end{equation*}

First, let
\begin{equation*}
a=-3\sum_{m \in \cD_N^*} \sum_{h=1}^{\ell(m)-1} a(m, h) \cdot E(m, h),
\end{equation*}
where
\begin{equation*}
a(m, h):=\frac{m''}{24\ell(m)} \sum_{k \mid m''} \frac{\mu(k)}{k}\left((\ell(m), kh)^2-\ell(m)^2\right)
\end{equation*}
and
\begin{equation*}
E(m, h)=\begin{cases}
\frac{1}{n}e(m, h) & \text{ if } ~~ 1\leq h < \p(\ell(m)),\\
\quad 0 & \text{ otherwise}.
\end{cases}
\end{equation*}
We claim that $a$ is an integer. Since we assume Conjecture \ref{conj: A}, which implies that $E(m, h)$ is an integer for any $(m, h) \in \cS^*$, it suffices to show that $3 a(m, h) \in \Z$. For simplicity, let $d=(\ell(m), k)$ (which depends on $k$), $k=dk'$ and $\ell(m)=d\ell'$. Since $(k', \ell')=1$, we have
$(\ell(m), kh)=d(\ell', k' h)=d(\ell', h)$. Also, let $c=m''/{k'\ell'}$ which is an integer as $k'$ and $\ell'$ both divide $m''$ and $(k', \ell')=1$.
Thus, we have
\begin{equation*}
24\cdot a(m, h)=\sum_{k \mid m''} \frac{m'' d^2 \mu(k)}{k\ell(m)} \left((\ell', h)^2-\ell'^2 \right)=\sum_{k \mid m''}  \mu(k) c\left((\ell', h)^2-\ell'^2 \right).
\end{equation*}
Since $\ell'$ is odd (as we assume that $L$ is odd), so is $(\ell', h)$, and hence $(\ell', h)^2\equiv \ell'^2 \equiv 1 \pmod 8$. This completes the proof of the claim.

Now, let
\begin{equation*}
g=G\prod_{m \in \cD_N^*} \prod_{h=1}^{\ell(m)-1}\left(\frac{F_{m, h}}{F_{m, 0}}\right)^{E(m, h)},
\end{equation*}
where $G=F_{N/3,0}^{4a}$ if $L$ is divisible by $3$; and $G=1$ otherwise. 
Since $\tn{div}(F_{m, 0})$ is fixed under the action of $\Gal(\Q(\z_L)/\Q)$, it is obvious that $D-\tn{div}(g)$ is rational. 
Hence it suffices to show that $g$ is a modular unit on $X_0(N)$. Since $a$ and $E(m, h)$ are all integers,
it is enough to show that all the conditions in Theorem \ref{theorem: criterion} hold. 
\begin{enumerate}
\item 
By Lemma \ref{lemma: order of Fmh at special cusps}, the order of $\frac{F_{m, h}}{F_{m, 0}}$ at the cusp $\infty$ is zero. 
Suppose that $L$ is divisible by $3$, i.e., $N$ is divisible by $9$. 
Again by Lemma \ref{lemma: order of Fmh at special cusps}, the order of $F_{N/3, 0}$ at the cusp $\infty$ is $-\frac{N}{36}$ as $m''=3$ in this case. Since $N$ is divisible by $9$, the order of $F_{N/3, 0}^4$ at the cusp $\infty$ is an integer.
Thus, whether $L$ is divisible by $3$ or not, the order of $g$ at the cusp $\infty$ is an integer.

\item
By Lemma \ref{lemma: order of Fmh at special cusps}, the order of $\frac{F_{m, h}}{F_{m, 0}}$ at the cusp $0$ is $a(m, h)$.
Suppose that $L$ is divisible by $3$. Then the order of $F_{N/3, 0}$ at the cusp $0$ is $\frac{1}{12}$. Thus, the order of $g$ at the cusp $0$ is $0$. 
Suppose that $L$ is not divisible by $3$. Then $\ell(m)$ is not divisible by $3$, and so 
by the same argument above, we can easily prove that $a(m, h)$ is an integer. Hence the order of $g$ at the cusp $0$ is also an integer.

\item
Similarly as (1), the order of $g$ at the cusp $\vect 1 {N_0}$ is an integer.
\item
By direct computation, we have
\begin{equation*}
(\s_s(F)/F)^{1/n} \approx \prod_{m \in \cD_N^*} \prod_{h=1}^{\ell(m)-1} \left(\frac{F_{m, sh}}{F_{m, h}}\right)^{E(m, h)}.
\end{equation*}
Since $(\s_s(F)/F)^{1/n}$ is a modular unit on $X_0(N)$, for any integer $1\leq i \leq N$ prime to $N$, we have
\begin{equation*}
\sum_{m \in \cD_N^*} \psi_i(m)\sum_{h=1}^{\ell(m)-1}(sh-h)E(m, h) =(s-1)\sum_{m \in \cD_N^*} \psi_i(m)\sum_{h=1}^{\ell(m)-1}h E(m, h) \equiv 0 \pmod L.
\end{equation*}
Since $L$ is odd, we can choose an integer $s$ such that $(s(s-1), L)=1$. (Here, we crucially use the assumption that $L$ is odd.) Thus, $g$ satisfies the mod $L$ condition. 
\item
It is obvious that $g$ satisfies the mod $2$ condition.
\end{enumerate}
This completes the proof.
\end{proof}

Finally, we provide some relations among $F_{m, h}$.
\begin{theorem}\label{theorem: relation among Fmh}
Let $m\in \cD_N^*$ such that $\ell(m)$ is divisible by a prime $p$. Then we have
\begin{equation*}
\prod_{j=0}^{p-1} F_{m, h+xj} \approx F_{mp, ~p^{\ve}h},
\end{equation*}
where $x=\ell(m)/p$ and $\ve=1+v_p(\ell(mp))-v_p(\ell(m))$.
\end{theorem}
\begin{proof}
For simplicity, let $H_{m, h}=\prod_{j=0}^{p-1} F_{m, h+xj}$.
It suffices to prove that the order of $H_{m, h}$ at a cusp $\vect a c$ is equal to that of $F_{mp, p^{\ve}h}$ at the same cusp.
Before proceeding, we recall the distribution property of Bernoulli polynomials $B_2(x)=x^2-x+1/6$ and $P_2(x)=B_2(\{x\})$. Namely, for any integer $n$ we have
\begin{equation*}
\sum_{j=0}^{n-1} B_2\left(x+\frac{j}{n}\right)=\frac{1}{n}B_2(nx),
\end{equation*}
and hence
\begin{equation*}
\sum_{j=0}^{n-1} P_2\left(x+\frac{j}{n}\right)=\frac{1}{n}P_2(nx),
\end{equation*}
which will be frequently used below.

For simplicity, let $\ell=\ell(m)$. Also, the subscript $()_1$ denotes the corresponding element for $mp$ instead of $m$. Namely, let
\begin{equation*}
\ell_1=\ell(mp), ~~m''_1=\frac{N}{mp\ell_1}, ~~N'_1=\frac{N}{\ell_1}, ~~ a_1'=\frac{N'_1 a}{(N'_1, c)} ~~\text{ and }~~c'_1=\frac{c}{(N_1', c)}.
\end{equation*}

Let $m''=\ell z=xpz$ for some integer $z$. Note that $\ve=0$ if $(z, p)=1$ and $\ve=1$ otherwise. Let 
\begin{equation*}
\begin{split}
A&:=\sum_{j=0}^{p-1}  \frac{\ell(N', c)^2}{4 (c^2, N)}  \sum_{\a \in \mm {m''}} P_2\left( \frac{\a a'}{m''} + \frac{\d h c'}{\ell} + \frac{\d jc'}{p}\right)\\
&=\frac{xp(N', c)^2}{4 (c^2, N)} \sum_{\a \in \mm {m''}} \sum_{j=0}^{p-1}  P_2\left( \frac{\a a'}{xpz} + \frac{\d h c'}{xp} + \frac{\d jc'}{p}\right)
\end{split}
\end{equation*}
and
\begin{equation*}
B:=\frac{\ell_1(N_1', c)^2}{4 (c^2, N)} \sum_{\a \in \mm {m_1''}} P_2\left( \frac{\a a_1'}{m''_1} + \frac{\d (p^{\ve}h) c_1'}{\ell_1}\right).
\end{equation*}
be the orders of  $H_{m, h}$ and $F_{mp, ~p^{\ve}h}$ at a cusp $\vect a c$, respectively, which are computed using Lemma \ref{lemma: order of Fmh}. We divide into two cases. 
\begin{enumerate}
\item
Suppose that $(c', p)=1$. Then we have
\begin{equation*}
\sum_{j=0}^{p-1} P_2\left( \frac{\a a'}{xpz} + \frac{\d h c'}{xp} + \frac{\d jc'}{p}\right)=\frac{1}{p}P_2\left( \frac{\a a'}{xz} + \frac{\d h c'}{x}\right).
\end{equation*}
Thus, we have
\begin{equation*}
A=\frac{x(N', c)^2}{4(c^2, N)} \sum_{\a \in \mm {m''}} P_2\left( \frac{\a a'}{xz} + \frac{\d h c'}{x}\right).
\end{equation*}

\begin{itemize}[--]
\item
If $(z, p)=1$ (and so $\ve=0$), then $\ell_1=x$ and $m_1''=m''=xpz$. Also, $N_1'=N/{\ell_1}=N'p$. Since $(c', p)=1$, i.e., $v_p(c) \leq v_p(N')$, we have $(N_1', c)=(N', c)$, $a_1'=pa'$ and $c_1'=c'$.
Thus, we have
\begin{equation*}
B=\frac{x(N', c)^2}{4 (c^2, N)} \sum_{\a \in \mm {m''}}P_2\left( \frac{\a a'}{xz} + \frac{\d h c'}{x}\right).
\end{equation*}

\item
If $(z, p)=p$ (and so $\ve=1$), then $\ell_1=xp$ and $m''_1=xz$. Also, $N_1'=N'$ and so $a_1'=a'$ and $c_1'=c'$. 
Since $p$ divides $z$, any element of $\mm {xpz}$ can be written as $\a+xzk$ for some $\a \in \mm {xz}$ and $k\in \{0, 1, \dots, p-1\}$. Note also that
the multiplicative inverse of $\a+xzk$ (modulo $xpz$) is of the form $\d+xzk'$, where $\d$ is the multiplicative inverse of $\a$ (modulo $xz$). Since
\begin{equation*}
\frac{\a a'}{xz} \equiv \frac{(\a+xzk)a'}{xz} \pmod \Z \qa \frac{\d h c'}{x} \equiv \frac{(\d+xzk')hc'}{x} \pmod \Z,
\end{equation*}
we have
\begin{equation*}
\sum_{\a \in \mm {m''}} P_2\left(\frac{\a a'}{xz} + \frac{\d h c'}{x}\right)=p\sum_{\a \in \mm {m_1''}} P_2\left(\frac{\a a'}{xz} + \frac{\d h c'}{x}\right).
\end{equation*}
Thus, we have
\begin{equation*}
\begin{split}
 B&=\frac{xp(N', c)^2}{4 (c^2, N)} \sum_{\a \in \mm {m_1''}} P_2\left( \frac{\a a'}{xz} + \frac{\d h c'}{x}\right)\\
 &=\frac{x(N', c)^2}{4 (c^2, N)} \sum_{\a \in \mm {m''}} P_2\left( \frac{\a a'}{xz} + \frac{\d h c'}{x}\right).
\end{split}
\end{equation*}
\end{itemize}

\item
Suppose that $(c', p)=p$. Then we have
\begin{equation*}
A=\frac{xp^2(N', c)^2}{4 (c^2, N)} \sum_{\a \in \mm {m''}}   P_2\left( \frac{\a a'}{xpz} + \frac{\d h c'}{xp}\right).
\end{equation*}
\begin{itemize}[--]
\item
If $(z, p)=1$ (and so $\ve=0$), then $\ell_1=x$ and $m_1''=m''=xpz$. Also, $N_1'=N/{\ell_1}=N'p$. Since $(c', p)=p$, i.e., $v_p(c) \geq v_p(N')+1$, we have $(N_1', c)=p(N', c)$, $a_1'=a'$ and $c_1'=c'/p$.
Thus, we have
\begin{equation*}
B=\frac{xp^2(N', c)^2}{4 (c^2, N)} \sum_{\a \in \mm {m''}}P_2\left( \frac{\a a'}{xpz} + \frac{\d h c'}{xp}\right).
\end{equation*}

\item
If $(z, p)=p$ (and so $\ve=1$), then $\ell_1=xp$ and $m''_1=xz$. Also, $N_1'=N'$ and so $a_1'=a'$ and $c_1'=c'$. 
As above, we write $\a$ and $\d$ in $\mm {xpz}$ as $\a+xzk$ and $\d+xzk'$ for some $\a, \d \in \mm {xz}$. Thus, we have
\begin{equation*}
\begin{split}
\sum_{\a \in \mm {m''}} P_2\left(\frac{\a a'+z\d h c'}{xpz}\right)&=\sum_{\a \in \mm {m_1''}} \sum_{k=0}^{p-1} P_2\left(\frac{\a a'}{xpz} + \frac{\d h c'}{xp}+\frac{ka'}{p}\right)\\
&=\frac{1}{p} \sum_{\a \in \mm {m_1''}} P_2\left(\frac{\a a'}{xz} + \frac{\d h c'}{x}\right),
\end{split}
\end{equation*}
where the first (resp. second) equality follows from the fact that $p$ divides $z$ (resp. $p$ does not divide $a'$).
Therefore we have
\begin{equation*}
\begin{split}
 B&=\frac{xp(N', c)^2}{4 (c^2, N)} \sum_{\a \in \mm {m_1''}} P_2\left( \frac{\a a'}{xz} + \frac{\d h c'}{x}\right)\\
 &=\frac{xp^2(N', c)^2}{4 (c^2, N)} \sum_{\a \in \mm {m''}} P_2\left( \frac{\a a'}{xpz} + \frac{\d h c'}{xp}\right).
\end{split}
\end{equation*}
\end{itemize}
\end{enumerate}
This completes the proof. 
\end{proof}

\begin{remark}\label{remark: (m, h) and (mp, h')}
By definition, if $v_p(h)=v_p(\ell(m))-a$ for some $a\in \Z$, then we have
\begin{equation*}
v_p(p^{\ve}h)=v_p(\ell(mp))-(a-1).
\end{equation*}
\end{remark}

\svv   
\section{Motivational examples}\label{section: toy}
Since our definition of $\Psi$ in the next section is quite complicated, 
we try to give a motivation behind the definition.
To begin with, we consider the case where $N$ is a prime power. 

First, let $N=p^2$. Then we easily have
\begin{equation*}
\cS=\{(1, 1), (1, 2), \dots, (1, p), (p, 1)\} \qa \cS^{\rd}=\{(1, 1)\}.
\end{equation*}
Thus, it suffices to replace $F_{1, 1}$ by others using Theorem \ref{theorem relation among Fmh}. Indeed we have
\begin{equation*}
F_{1, 1} \approx F_{p, 1} \times \prod_{j=2}^{p} F_{1, ~j}^{-1}
\end{equation*}
and hence 
\begin{equation*}
\prod_{(m, h) \in \cS} F_{m, h}^{a(m, h)}  \approx \prod_{(m, h) \in \cS^\new} F_{m, h}^{a(m, h)+b(m, h)},
\end{equation*}
where 
\begin{equation*}
b(m, h)=\begin{cases}
-a(1, 1) & \text{ if } ~~ m=1,\\
a(1, 1) & \text{ if } ~~ m=p.
\end{cases}
\end{equation*}
(Here and below $b$ is defined only on $\cS^{\new}$.)

\svv  
Next, let $N=p^3$. Then we have
\begin{equation*}
\cS=\{(1, k), (p, k), (p^2, 1) : 1\leq k \leq p\} \qa \cS^{\rd}=\{(1, 1), (p, 1)\}.
\end{equation*}
 As above, by Theorem \ref{theorem relation among Fmh} we have
\begin{equation*}
F_{1, 1} \approx F_{p, p} \times \prod_{j=2}^{p} F_{1, ~j}^{-1} \qa F_{p, 1} \approx F_{p^2, 1} \times \prod_{j=2}^{p} F_{p, ~j}^{-1}.
\end{equation*}
Thus, we can replace $F_{1, 1}$ and $F_{p, 1}$ by others. In this case, the order of the replacements is not important and we get
\begin{equation*}
\prod_{(m, h) \in \cS} F_{m, h}^{a(m, h)}  \approx \prod_{(m, h) \in \cS^\new} F_{m, h}^{a(m, h)+b(m, h)},
\end{equation*}
where
\begin{equation*}
b(m, h)=\begin{cases}
-a(1, 1) & \text{ if } ~~ m=1,\\
-a(p, 1) & \text{ if } ~~ m=p \qqa h\neq p,\\
a(1, 1)-a(p, 1) & \text{ if }~~ (m, h)=(p, p),\\
a(p, 1) & \text{ if }~~m=p^2.
\end{cases}
\end{equation*}

\svv   
Then, let $N=p^4$. We have
\begin{equation*}
\cS^{\rd}=\{(1, k), (p, 1), (p^2, 1) : 1\leq k \leq p\}.
\end{equation*}
Again by Theorem \ref{theorem relation among Fmh}, for each $1\leq k \leq p$ we have
\begin{equation*}
F_{1, k} \approx F_{p, k} \times \prod_{j=1}^{p-1} F_{1, k+pj}^{-1}.
\end{equation*}
Moreover, we have
\begin{equation*}
F_{p, 1} \approx F_{p^2, p} \times \prod_{j=2}^{p} F_{p, j}^{-1} \qa F_{p^2, 1} \approx F_{p^3, 1} \times \prod_{j=2}^{p} F_{p^2, j}^{-1}.
\end{equation*}
We may replace $F_{p^2, 1}$ and $F_{p, 1}$ by others and then replace $F_{1, k}$ by others. If so, the term $F_{p, 1}$ reappears and hence we should work again.
In contrast with the previous case, the order of the replacements is significant.
We first replace $F_{1, k}$ by others. Next, we replace $F_{p, 1}$ by others. Finally, we replace $F_{p^2, 1}$ by others. 
As a result, we have
\begin{equation*}
\prod_{(m, h) \in \cS} F_{m, h}^{a(m, h)}  \approx \prod_{(m, h) \in \cS^\new} F_{m, h}^{a(m, h)+b(m, h)},
\end{equation*}
where  
\begin{equation*}
b(m, h)=\begin{cases}
-a(1, h') & \text{ if } ~~ m=1,\\
a(1, h)-a(1, 1)-a(p, 1) & \text{ if } ~~ m=p,\\
-a(p^2, 1) & \text{ if }~~ m=p^2 \qqa h \neq p, \\
a(1, 1)+a(p, 1) -a(p^2, 1)& \text{ if }~~(m, h)=(p^2, p),\\
a(p^2, 1) & \text{ if }~~m=p^3.
\end{cases}
\end{equation*}
(Here, $h'$ is the smallest positive integer congruent to $h$ modulo $p$.)

\svv   
Finally, let $N=p^r$ for some integer $r\geq 2$. A general recipe is as follows: For each divisor $\fm$ of $p^r$ with $\ell(\fm)>1$, we construct a map $\Psi_\fm$ based on the discussion above. 
The desired map $\Psi$ will then be the composition of $\Psi_\fm$ from $\fm=1$ to $\fm=p^{r-2}$. 
More specifically, let $\fm=p^k$ with $\ell(\fm)=p^n>1$. Then by definition, we have
\begin{equation*}
I_\fm=\{ h \in \Z : 1 \leq h \leq p^n\} \qa I_\fm^{\rd}=\{ h \in \Z : 1 \leq h \leq p^{n-1}\}.
\end{equation*}
For each $h \in I_\fm$, let $h'$ be the corresponding element in $I_\fm^{\rd}$, i.e., $h'$ is the smallest positive integer congruent to $h$ modulo $p^{n-1}$. Then by Theorem \ref{theorem relation among Fmh}, we have
\begin{equation*}
\prod_{h \in I_\fm} F_{\fm, h}^{a(\fm, h)} \approx \prod_{h \in I_\fm^{\rd}} F_{\fm p, p^{\ve} h}^{a(\fm, h)}\times \prod_{h \in I_\fm \sm I_\fm^{\rd}} F_{\fm, h}^{a(\fm, h)-a(\fm, h')}.
\end{equation*}
This motivates the following definition.

\begin{definition}\label{definition1-toy}
Let $N=p^r$ be a prime power with $r\geq 2$. If $\ell(\fm)$ is divisible by $p$, then we define a map $\Psi_\fm$ as follows.
\begin{equation*}
\Psi_\fm(f)(m, h):=\begin{cases}
f(m, h)-f(\fm, h') & \text{ if } ~~ m=\fm,\\
f(m, h)+f(\fm, p^{-\ve}h) & \text{ if } ~~ m=\fm p \qqa p^{\ve} \dd h,\\
f(m, h) & \text{ otherwise},
\end{cases}
\end{equation*}
where $h'$ is the smallest positive integer congruent to $h$ modulo $\ell(\fm)/p$, and 
\begin{equation*}
\ve:=1+v_p(\ell(\fm p))-v_p(\ell(\fm)).
\end{equation*}
Furthermore, for any $f \in \scF$ we define a map $\Psi$ as follows.
\begin{equation*}
\Psi(f):=\Psi_{p^{r-2}} \circ \Psi_{p^{r-3}} \circ \cdots \circ \Psi_p \circ \Psi_1(f).
\end{equation*}
\end{definition}

\svv   
By definition, $\Psi_\fm(f)(\fm, h)=0$ for all $1\leq h \leq \ell(\fm)/p$, which means that $h\in I_\fm^{\rd}$ if $\ell(\fm)$ is a power of $p$. Hence the above definition works to ``remove'' all terms $F_{\fm, h}$ with $h \in I_\fm^{\rd}$. On the other hand, if $\ell(\fm)$ is divisible by $pq$ then the size of $I_\fm^{\rd}$ is larger than both $\ell(\fm)/p$ and $\ell(\fm)/q$. Thus, we need a variant of Definition \ref{definition1-toy} as follows.

\begin{definition}\label{definition2-toy}
Suppose that $\ell(\fm)$ is divisible by a prime $\fp$. We fix an interval $I_\fm^\fp$ in $I_\fm$ of size $\ell(\fm)/\fp$, and replace $F_{\fm, h}$ for $h \in I_\fm^\fp$ by others. For each $h \in I_\fm$, let $\rho_\fm(h)$ denote the unique element of $I_\fm^\fp$ congruent to $h$ modulo $\ell(\fm)/\fp$. Then we define a map $\Phi_\fm$ as follows.
\begin{equation*}
\Phi_\fm(f)(m, h):=\begin{cases}
f(m, h)-f(\fm, \rho_\fm(h)) & \text{ if } ~~ m=\fm,\\
f(m, h)+f(\fm, \rho_\fm(\fp^{-\ve}h)) & \text{ if } ~~ m=\fm \fp \qqa \fp^{\ve} \dd h,\\
f(m, h) & \text{ otherwise}.
\end{cases}
\end{equation*}
\end{definition}

\svv    
Now, suppose that $\ell(\fm)=pq$ for two primes $p<q$. By definition, we have
\begin{equation*}
I_\fm=\{h\in \Z: 1\leq h \leq pq\} \qa I_\fm^{\rd}=\{h \in \Z : 1\leq h \leq p+q-1\}.
\end{equation*}
Let $\Phi_\fm^1$ be the map $\Phi_\fm$ in Definition \ref{definition2-toy} with $\fp=p$ and $I_\fm^\fp=\{h \in \Z : 1\leq h \leq q\}$. Similarly, we may define $\Phi_\fm^2$ be the map with $\fp=q$ and $I_\fm^\fp=\{ h \in \Z : q < h \leq p+q\}$. Since $p+q \not\in I_\fm^{\rd}$, it is not necessary to replace $F_{\fm, p+q}$ by others. Thus, we slightly modify our definition of $\Phi_\fm^2$ as follows.
\begin{equation*}
\Phi_\fm^2(f)(m, h):=\begin{cases}
f(m, h)-\d(h) \cdot f(\fm, \rho_\fm(h)) & \text{ if } ~~ m=\fm,\\
f(m, h)+\d(\fp^{-\ve}h) \cdot f(\fm, \rho_\fm(\fp^{-\ve}h)) & \text{ if } ~~ m=\fm \fp \qqa \fp^{\ve} \dd h,\\
f(m, h) & \text{ otherwise},
\end{cases}
\end{equation*}
where $\d(h)=0$ if $\rho_\fm(h) \not\in I_\fm^{\rd}$ and $1$ otherwise. 

\svv
Note that $\d(q)=0$ as $\rho_\fm(q)=p+q \not\in I_\fm^{\rd}$.
Note also that for any $f \in \scF$,
\begin{equation*}
\Phi_\fm^1(f)(\fm, h)=0  \quad \text{ for all } h \in I_\fm^p.
\end{equation*}
Since $q \in I_\fm^p$ (and $\d(q)=0$), we have
\begin{equation}\label{eqn: mot1}
\Phi_\fm^2 \circ \Phi_\fm^1(f)(\fm, h)=0 \quad \text{for all } h \in I_\fm^p \text{ congruent to $q$ modulo $p$}.
\end{equation}
Thus, we have
\begin{equation*}
(\Phi_\fm^1 \circ \Phi_\fm^2) \circ \Phi_\fm^1(f)(\fm, q+r)=0,
\end{equation*}
where $r$ is the remainder of $q$ modulo $p$. As in \eqref{eqn: mot1}, we then have
\begin{equation*}
\Phi_\fm^2 \circ \Phi_\fm^1\circ \Phi_\fm^2 \circ \Phi_\fm^1(f)(\fm, h)=0 \quad \text{for all } h \in I_\fm^p \text{ congruent to $2q$ modulo $p$}.
\end{equation*}
Similarly, we have
\begin{equation*}
(\Phi_\fm^1 \circ \Phi_\fm^2)^2 \circ \Phi_\fm^1(f)(\fm, q+r)=(\Phi_\fm^1 \circ \Phi_\fm^2)^2 \circ \Phi_\fm^1(f)(\fm, q+r')=0,
\end{equation*}
where $r'$ is the remainder of $2q$ modulo $p$. Doing this successively, we can easily prove that
\begin{equation*}
(\Phi_\fm^1 \circ \Phi_\fm^2)^{p-1} \circ \Phi_\fm^1(f)(\fm, h)=0 \quad \text{ for all } h \in I_\fm^{\rd}.
\end{equation*}
Hence if $\ell(\fm)=pq$, we can define our map $\Psi_\fm$ as follows.
\begin{equation*}
\Psi_\fm :=(\Phi_\fm^1 \circ \Phi_\fm^2)^{p-1} \circ \Phi_\fm^1.
\end{equation*}

\svv   
More generally, suppose that $\ell(\fm)=p^a q^b$ for some integers $a, b \geq 1$. 
Let $\fd(\fm)=p^{a-1}q^{b-1}$ be an integer. Then we divide the interval $I_\fm$ into $pq$ pieces, where $\nu$-th piece is
\begin{equation*}
I_\fm(\nu):=\{h \in \Z : (\nu-1)\fd(\fm)< h \leq \nu\fd(\fm)\}.
\end{equation*}
Then the situation is analogous to the previous case. For instance, we have
\begin{equation*}
I_\fm^p=\mcoprod_{1\leq \nu\leq q} I_\fm(\nu) \qa I_\fm^q=\mcoprod_{q< \nu\leq p+q} I_\fm(\nu). 
\end{equation*}
Also, we have
\begin{equation*}
(\Phi_\fm^1 \circ \Phi_\fm^2) \circ \Phi_\fm^1(f)(\fm, h)=0 \quad \text{ for all } ~h \in I_\fm(q+r).
\end{equation*}
Hence the above definition of $\Psi_\fm$ works  in general.

\begin{remark}
If we regard an interval $I_\fm(\nu)$ as an integer $\nu$, then all the formula for $\Psi_\fm$ are the same as the case of $\ell(\fm)=pq$. This is the reason why we use ``vector notation'' in Section \ref{section: formula}.
\end{remark}
 
So far we have defined our map $\Psi_\fm$ for all divisors $\fm$ of $N$ with $\ell(\fm)>1$. 
To define our map $\Psi$, we compose the maps $\Psi_\fm$ based on the divisibility relation as in the prime power case. Since $\Psi_\fm(f)(m, h)=f(m, h)$ unless $m \in \{ \fm, \fm p, \fm q\}$, we split our domain into several subsets and treat them separately. For more details, see the next section.

\svv   
\section{The map $\Psi$}\label{section: the map Psi}
From now on, let $N=Mp_1^{r_1}p_2^{r_2}$ for some squarefree integer $M$ and two distinct primes $p_1, p_2$ such that $\gcd(M, p_1p_2)=1$ and $p_1<p_2$. 
We henceforth assume that $r_1\geq 2$ so that $N$ is not squarefree. If $r_2=1$, then we may replace $M$ by $Mp_2$, and so we assume that either $r_2=0$ in which case $L$ is a prime power, or $r_2\geq 2$ in which case $L$ is divisible by $p_1p_2$.
We exclusively use the letter $\iota \in \{1, 2\}$ for the index of the prime divisors of $N/M$. Also, for simplicity, we frequently denote $p_1$ and $p_2$ by $p$ and $q$, respectively. (Hence we always assume that $p<q$.)

\svv   
To begin with, we define various subsets of $\cD_N$ and $\cS$ as follows.
Since some divisors of $N$ are not necessary in our discussion, we define
\begin{equation*}
\cD':=\cD_N \sm \{ dp_1^{r_1}p_2^{r_2} : d \dd M\} \qa \cS':=\{(m, h) \in \cS : m \in \cD'\}.
\end{equation*}
(For instance, we already have $E(m, h)=0$ for any $(m, h) \in \cS\sm \cS'$.)
For a divisor $d$ of $M$ we define the following.
\begin{itemize}[--]
\item
$\cD(d):=\{ m \in \cD' : \gcd(m, M)=d\}=\{dp_1^{k_1}p_2^{k_2} : 0\leq k_i \leq r_i \qqa (k_1, k_2) \neq (r_1, r_2)\}$.

\item
$\cD(d)_{\iota}:=\{ m \in \cD(d): v_{p_\iota}(m)=r_\iota\}=\{dp_\tau^k p_\iota^{r_\iota} : 0\leq k < r_\tau\}$, where $\tau=3-\iota$.

\item
$\cS(d) :=\{(m, h) \in \cS' : m \in \cD(d) \} ~~ \text{ and } ~~  \cS(d)_\iota:=\{(m, h) \in \cS(d) : m \in \cD(d)_\iota \}$.
\item
$\cD(d)_0:=\cD(d) \sm (\cD(d)_1 \cup \cD(d)_2)  ~~ \text{ and } ~~   \cS(d)_0:=\cS(d) \sm (\cS(d)_1 \cup \cS(d)_2)$.
\end{itemize}
By definition, we have
\begin{equation*}
\cD'=\mcoprod_{d \mid M} \left(\cD(d)_0 \smallcoprod \cD(d)_1 \smallcoprod \cD(d)_2 \right) \qa \cS'=\mcoprod_{d \mid M} \left(\cS(d)_0 \smallcoprod \cS(d)_1 \smallcoprod \cS(d)_2 \right).
\end{equation*}

Recall that $\scF$ denotes the set of all set-theoretic maps from $\cS$ to $\Q$. 
Similarly, for each positive divisor $d$ of $M$ and  $i \in \{0, 1, 2\}$, let $\scF(d)_i$ denote the set of all set-theoretic maps from $\cS(d)_i$ to $\Q$. (As mentioned in the introduction, we often regard them as maps to $\Q/{\Z}$ after composing the natural projection $\Q \to \Q/{\Z}$.) We will define the map $\Psi$ on each $\cS(d)_i$ in Section \ref{section: 4.2}, and then for each $f \in \scF$, let
\begin{equation*}
\Psi(f)(m, h):=\begin{cases}
\Psi(f|_{\cS(d)_i})(m, h) & \quad \text{ if } ~~ (m, h) \in \cS(d)_i,\\
f(m, h) & \quad  \text{ if } ~~ (m, h) \not\in \cS'.
\end{cases}
\end{equation*}
As we have seen in the definition, it is natural to deal with our discussion on $\scF(d)_i$ separately. Thus, we henceforth fix a positive divisor $d$ of $M$.  
It will be clear that our discussion does not depend on this fixed divisor.

\svv   
\subsection{The definition of $\Psi_\fm$}\label{subsection: 4.1}
To begin with, we fix some notation.
\begin{notation}\label{notation: notion associated with m}
For any $\fm \in \cD'$, we define various notions associated with $\fm$ as follows.
\begin{itemize}[--]
\item
$\ell(\fm)$ : the largest integer whose square divides $N/\fm$.
\item
$\fl_\iota(\fm):=\ell(\fm)/{p_\iota}$.
\item
$\fd(\fm):=\ell(\fm)/{pq}$ if $\ell(\fm)$ is divisible by $pq$. As in the previous section, for each integer $1\leq \nu \leq pq$ let
\begin{equation*}
I_\fm(\nu):=\{h \in \Z : (\nu-1)\fd(\fm)< h \leq \nu\fd(\fm)\}.
\end{equation*}
\item
$\ve_\iota(\fm):=1+v_{p_\iota}(\ell(\fm p_\iota))-v_{p_\iota}(\ell(\fm)) \in \{0, 1\} ~~~\text{ if } ~~ ~\fm p_\iota \in \cD'$. 
\item
$I_\fm:=\{ h \in \Z : 1\leq h \leq \ell(\fm)\}$.
\item
$I_\fm^{\rd}:=\{ h \in \Z : 1\leq h \leq \ell(\fm)-\p(\ell(\fm))\}$.
\item
As in the discussion below Definition \ref{definition2-toy}, we define subsets $I_\fm^\iota$ of $I_\fm$ as follows: If $\ell(\fm)$ is a power of $p_\iota$, then let 
\begin{equation*}
I_\fm^\iota:=\{ h \in I_\fm : 1\leq h \leq \fl_\iota(\fm)\}.
\end{equation*}
If $\ell(\fm)$ is divisible by $pq$, let
\begin{equation*}
\begin{split}
I_\fm^1&:=\{ h \in I_\fm : 1\leq h \leq \fl_1(\fm)\}=\mcoprod\nolimits_{1\leq \nu\leq q} I_\fm(\nu) \qa \\
I_\fm^2&:=\{ h \in I_\fm : \fl_1(\fm)< h \leq \fl_1(\fm)+\fl_2(\fm)\}=\mcoprod\nolimits_{q< \nu\leq p+q} I_\fm(\nu).
\end{split}
\end{equation*}

\item
For any $h \in I_\fm$ with $p_\iota \dd \ell(\fm)$, let $\rho_\fm^\iota(h)$ be the unique element of $I_\fm^\iota$ such that
\begin{equation*}
\rho_\fm^\iota(h)  \equiv h \pmod {\fl_\iota(\fm)}.
\end{equation*}

\item
Suppose that $\ell(\fm)$ is divisible by $p_\iota$. Then for any $h \in I_\fm$, let
\begin{equation*}
\chi^\iota_\fm(h):=\begin{cases}
1 & \text{ if } ~~ \rho_\fm^\iota(h) \in I_\fm^{\rd},\\
0 & \text{ otherwise}.
\end{cases}
\end{equation*}
\end{itemize}
\end{notation}

As discussed in the previous section, we define the following.
\begin{definition}\label{definition: the map Phi}
For each $i \in \{0, 1, 2\}$, let $\fm \in \cD(d)_i$.
If $\ell(\fm)$ is divisible by $p_\iota$, then we define a map $\Phi_\fm^\iota : \scF(d)_i \to \scF(d)_i$ as follows: For any $f \in \scF(d)_i$, let
\begin{equation*}
\Phi_\fm^\iota(f)(m, h):=\begin{cases}
f(m, h)-\chi_\fm^\iota(h)\cdot f(\fm, \rho_\fm^\iota(h)) & \text{ if } ~~m=\fm,\\
f(m, h)+\chi_\fm^\iota(p_\iota^{-\ve_\iota(\fm)}h)\cdot f(\fm, \rho_\fm^\iota(p_\iota^{-\ve_\iota(\fm)}h)) & \text{ if } ~~m=\fm p_\iota \qqa p_\iota^{\ve_\iota(\fm)} \dd h,\\
f(m, h) & \text{ otherwise}.
\end{cases}
\end{equation*}
If $\ell(\fm)$ is not divisible by $p_\iota$, then let $\Phi_\fm^\iota(f):=f$.  Using $\Phi_\fm^\iota$, let
\begin{equation*}
\Psi_\fm:=(\Phi_{\fm}^1 \circ \Phi_\fm^2)^{p-1} \circ \Phi_\fm^1.
\end{equation*}
\end{definition}

\svv
By Corollary \ref{corollary: vanishing on Im1} below, we have $\Psi_\fm(f)(\fm, h)=0$ for all $h \in I_\fm^{\rd}$ as desired.

\begin{remark}\label{remark: Hm1 Hm2 Im1}
If $\ell(\fm)$ is a power of $p_\iota$, then we always have $\chi_\fm^\iota(h)=1$ as $I_\fm^{\rd}=I_\fm^\iota$. Also, if $\ell(\fm)$ is divisible by $pq$, then we always have $\chi_\fm^1(h)=1$ as $I_\fm^1 \subset I_\fm^{\rd}$. Thus, for any $h \in I_\fm$ with $p_\iota \dd \ell(\fm)$, we have $\chi_\fm^\iota(h)=0$ if and only if $\ell(\fm)$ is divisible by $pq$, $\iota=2$ and $\rho_\fm^2(h) \in I_\fm(p+q)=I_\fm^2 \sm I_\fm^{\rd}$. 
\end{remark}

\begin{remark}\label{remark: Hm1 Hm2 Im2}
Suppose $p_\iota \dd \ell(\fm)$. If $h \in I_\fm^{\iota}$, then we have $\rho_\fm^{\iota}(h)=h$ and hence we have
\begin{equation*}
\Phi_\fm^\iota(\fm, h) =0 \quad \text{ for all } ~ h \in I_\fm^{\iota} \cap I_\fm^{\rd}.
\end{equation*}
\end{remark}

\begin{remark}\label{remark: map Phi 1}
Let $p_\iota \dd \ell(\fm)$. If $f(\fm, h)=0$ for all $h \in I_\fm^\iota \cap I_\fm^{\rd}$, then for any $h\in I_\fm$ we have
\begin{equation*}
\chi_\fm^\iota(h)\cdot f(\fm, \rho_\fm^\iota(h))=0
\end{equation*}
and hence $\Phi_\fm^\iota(f)=f$. 
In particular, we have $\Phi_\fm^\iota \circ \Phi_\fm^\iota(f)=\Phi_\fm^\iota(f)$, i.e., $\Phi_\fm^\iota$ is a projection.
\end{remark}

\begin{remark}\label{remark: map Phi 2}
By definition, if $\ell(\fm)=1$, then $\Phi_\fm^1(f)=\Phi_\fm^2(f)=f$ and so $\Psi_\fm(f)=f$. 
Suppose that $\ell(\fm)$ is a power of $p_\iota$. Then by definition, $\Phi_\fm^\tau(f)=f$, where $\tau=3-\iota$, and so 
\begin{equation*}
\Psi_\fm(f)=(\Phi_\fm^\iota)^{p-1} \circ \Phi_\fm^1(f)=\Phi_\fm^\iota(f),
\end{equation*}
where the last equality follows by the previous remark. 
\end{remark}

\begin{remark}\label{remark: identity if vanish on Im1}
Suppose further that $f(\fm, h)=0$ for all $h \in I_\fm^{\rd}$. Then by the same argument as above, we easily prove that
$\Phi_\fm^1(f)=\Phi_\fm^2(f)=f$ and so
\begin{equation*}
\Psi_\fm(f)=f.
\end{equation*}
Thus, for any $f \in \scF$ we have $\Psi_\fm \circ \Psi_\fm(f)=\Psi_\fm(f)$, i.e., $\Psi_\fm$ is also a projection.
\end{remark}

\svv   
\subsection{The definition of $\Psi$}\label{section: 4.2}
Now we are ready to define our map $\Psi$ on each $\scF(d)_i$ for $i \in \{0, 1, 2\}$. 
For $\iota \in \{1, 2\}$, we define a map $\Psi$ on $\scF(d)_\iota$ as follows.
\begin{equation*}
\Psi:=\Psi_{dp_\iota^{r_\iota}p_\tau^{r_\tau-1}} \circ \Psi_{dp_\iota^{r_\iota}p_\tau^{r_\tau-2}}\circ \cdots \circ \Psi_{dp_\iota^{r_\iota}p_\tau^{1}} \circ \Psi_{dp_\iota^{r_\iota}}=\Phi^\tau_{dp_\iota^{r_\iota}p_\tau^{r_\tau-1}} \circ \cdots \circ \Phi^\tau_{dp_\iota^{r_\iota}},
\end{equation*}
where $\tau=3-\iota$. Finally, we define it on $\scF(d)_0$. To do so, we need an ordering on $\cD(d)_0$. Let $\prec$ be an ordering $\cD(d)_0$ preserving the divisibility relation, i.e., $m_i \prec m_j$ whenever $m_i \dd m_j$ and $m_i\neq m_j$. Write
\begin{equation*}
\cD(d)_0=\{m_1, m_2, \dots, m_{r_1r_2}\}
\end{equation*}
so that $m_i \prec m_j$ if and only if $i<j$. Then by our assumption on divisibility, we always have
\begin{equation*}
m_1=d \qa m_{r_1r_2}=dp_1^{r_1-1}p_2^{r_2-1}.
\end{equation*}
Let
\begin{equation*}
\begin{split}
\U(m_k)&:=\Psi_{m_{k-1}} \circ \Psi_{m_{k-2}} \circ \cdots \circ \Psi_{m_1} \qa \\ \Psi(m_k)&:=\Psi_{m_k} \circ \U(m_k)=\Psi_{m_k} \circ \Psi_{m_{k-1}} \circ \cdots \circ \Psi_{m_1}.
\end{split}
\end{equation*}
Finally, let $\Psi:=\Psi(m_{r_1r_2})$.

\svv   
\begin{remark}\label{remark: Psi(f) independent of ordering}
Let $m$ and $\fm$ be two elements of $\cD(d)_0$ such that $\fm \prec m$. 
If $\fm \in \{m, mp, mq\}$ then $m \preceq \fm$ by our definition of the ordering, which is a contradiction. 
By Lemma \ref{lemma: simple formula for Psi}(1) below, we have 
\begin{equation*}
\Psi_{m}(f)(\fm, h)=f(\fm, h) \quad \text{ for all } ~~ h \in I_{\fm}
\end{equation*}
and hence if $\fm=m_k$ for some $k$, then 
\begin{equation*}
\begin{split}
\Psi(f)(\fm, h)&=\Psi_{m_{r_1r_2}} \circ \cdots \circ \Psi_{m_{k+1}} \circ \Psi(\fm)(f)(\fm, h) \\
&=\Psi(\fm)(f)(\fm, h)=\Psi_\fm(\U(\fm)(f))(\fm, h)  \quad \text{ for all } ~~ h \in I_{\fm}.
\end{split}
\end{equation*}
\end{remark}

\begin{remark}\label{remark: Psi(Psi)=Psi}
By Remarks \ref{remark: identity if vanish on Im1} and \ref{remark: Psi(f) independent of ordering} (and Corollary \ref{corollary: vanishing on Im1}), we have 
\begin{equation*}
\Psi(\Psi(\fm)(f))=\Psi(f) \quad \text{ for any }~ \fm \in \cD(d)_0.
\end{equation*}
\end{remark}

The definition of $\Psi$ is independent of a choice of an ordering on $\cD(d)_0$ as long as it preserves the divisibility relation.
Indeed, if $m, \fm \in \cD(d)_0$ such that $m \nmid \fm$ and $\fm \nmid m$, then by Lemma \ref{lemma: simple formula for Psi}(1) below, we easily prove that
\begin{equation*}
\Psi_\fm \circ \Psi_m(f)=\Psi_m \circ \Psi_\fm(f).
\end{equation*}
Thus, we will choose the (co-)lexicographic order on $\cD(d)_0$ whenever it is convenient.

\svv  
The following is the reason for our definition of the map $\Psi$.
\begin{theorem}\label{theorem: Psi(f)=0}
Suppose that 
\begin{equation*}
\prod_{(m, h) \in \cS} F_{m, h}^{f(m, h)}
\end{equation*}
is a modular unit on $X_0(N)$. Then we have
\begin{equation*}
\Psi(f)(m, h) \in \Z \quad \text{ for all }~~ (m, h) \in \cS.
\end{equation*}
\end{theorem}
\begin{proof}
First, let $\fm \in \cD(d)_\iota \cap \cD_N^*$ and $h \in I_\fm^{\rd}$.   Using Theorem \ref{theorem relation among Fmh}, we can replace $F_{\fm, h}^{f(\fm, h)}$ by others. This process is exactly what the map $\Phi_\fm^\iota$ does.
Since $\ell(\fm)$ is always a prime power, we have $\Phi_\fm^\iota=\Psi_\fm$ and so we obtain
\begin{equation*}
\prod_{(m, h) \in \cS(d)_\iota} F_{m, h}^{f(m, h)} \approx \prod_{(m, h) \in \cS(d)_\iota} F_{m, h}^{\Psi_\fm(f)(m, h)}.
\end{equation*}
Doing this from $dp_\iota^{r_\iota}$ to $dp_\iota^{r_\iota}p_\tau^{r_\tau-1}$ inductively (where $\tau=3-\iota$), we get
\begin{equation*}
\prod_{(m, h) \in \cS(d)_\iota}  F_{m, h}^{f(m, h)} \approx \prod_{(m, h) \in \cS(d)_\iota} F_{m, h}^{\Psi(f)(m, h)}.
\end{equation*}

Next, let $\fm\in \cD(d)_0 \cap \cD_N^*$ and $h \in I_\fm^{\rd}$. As above, we replace $F_{\fm, h}^{f(\fm, h)}$ by others.
\begin{equation*}
\prod_{(m, h) \in \cS(d)_0} F_{m, h}^{f(m, h)} \approx \prod_{(m, h) \in \cS(d)_0} F_{m, h}^{\Psi_\fm(f)(m, h)}.
\end{equation*}
Doing this from $m_1$ to $m_{r_1r_2}$ inductively, we get
\begin{equation*}
\prod_{(m, h) \in \cS(d)_0}  F_{m, h}^{f(m, h)} \approx \prod_{(m, h) \in \cS(d)_0} F_{m, h}^{\Psi(f)(m, h)}.
\end{equation*}

We can do the same thing for other divisors of $M$, and finally we get
\begin{equation*}
\prod_{(m, h) \in \cS} F_{m, h}^{f(m, h)}\approx \prod_{(m, h) \in \cS} F_{m, h}^{\Psi(f)(m, h)}.
\end{equation*}
Since $\Psi(f)(m, h)=\Psi_m(g)(m, h)$ for some $g \in \scF$, by Corollary \ref{corollary: vanishing on Im1} below we have 
\begin{equation*}
\Psi(f)(m, h)=\Psi_m(g)(m, h)=0 \quad \text{ for any }~h \in I_m^{\rd}.
\end{equation*}
(For instance, if $m\in \cD(d)_0$, then we have $g=\U(m)(f)$. See Remark \ref{remark: Psi(f) independent of ordering}.)
Thus, by Theorem \ref{thm: basis variant} (and Remark \ref{remark: rational becomes integral}) we have $\Psi(f)(m, h) \in \Z$ for all $(m, h) \in \cS$. 
\end{proof}

By our assumption on $E_s$, we easily have the following.
\begin{corollary}\label{corollary: Psi(gs)=0}
For any integer $s$ prime to $L$, we have
\begin{equation*}
\Psi(E_s)(m, h) \in \Z \quad \text{ for all } ~(m, h) \in \cS.
\end{equation*}
\end{corollary}

\svv   
\section{Formulas for $\Psi_\fm$}\label{section: formula}
If $\ell(\fm)$ is a prime power, then the definition of $\Psi_\fm$ is simple. (For instance, see Section \ref{section: toy}.) So throughout this section, we assume that 
$\fm \in \cD(d)_0$ and $\ell(\fm)$ is divisible by $p_1p_2$ unless otherwise mentioned. Also, let $f \in \scF(d)_0$. Before proceeding, we introduce various notations.

\begin{notation}\label{notation: ell=pq}
For any $j\in \Z$, let $\a_j$ be the smallest positive integer such that $\a_j \equiv jq \pmod p$ and $\b_j:=\a_j+q$. Also, for any $i \in \Z$ let $\a^i_j:=\a_j+pi$. From now on, we will write
\begin{alignat*}{2}
T_0&:=\{1, 2, \dots, p\}   &&=\{ \a_1, \a_2, \dots, \a_p\},\\
T_1 &:=\{ h \in \Z : 1 \leq h \leq q \} &&=\{ \a_1, \a_2, \dots, \a_p, \a^1_j, \dots \},\\
T_2&:=\{ h \in \Z : q < h \leq p+q\} &&=\{ \b_0, \b_1, \dots, \b_{p-2}, \b_{p-1}\}.
\end{alignat*}
(Here, we set $\b_0=\b_p=\a_p+q=p+q$.)
Namely, if we use the notation $\a_j$ (resp.  $\a^i_j$), then it \textbf{always} means that $1\leq j \leq p$ and $\a_j \in T_0$ (resp. $1\leq j \leq p$ and $\a^i_j \in T_1$). Similarly, if we use the notation $\b_k$, then we always assume that $0\leq k \leq p-1$ and $\b_k \in T_2$.
\end{notation}
\begin{notation}\label{notation: nu}
From now on, we exclusively use the letter $\nu$ to denote an integer such that $1\leq \nu \leq pq$. Let $[\nu]_\iota$ be the element of $T_\iota$ such that 
\begin{equation*}
\nu \equiv [\nu]_\iota \pmod {p_{\tau}},
\end{equation*}
where $\tau=3-\iota$. So we will write $[\nu]_1=\a^i_j$ and $[\nu]_2=\b_k$ for suitable $i, j$ and $k$ as above. 
\end{notation}

\begin{notation}
Suppose that $\ell(\fm)$ is divisible by $p_\iota$. For any $1\leq a \leq p_\iota$, let
$U_a^\iota(f, \fm)$ be the vector in $\Q^{\fl_\iota(\fm)}$ whose $z$-th component is 
\begin{equation*}
U_a^\iota(f, \fm)_z := f(\fm, \fl_\iota(\fm)(a-1)+z).
\end{equation*}
Suppose that $\ell(\fm)$ is divisible by $pq$. For any $1\leq \nu \leq pq$ and $1\leq a \leq p_\tau$, let $V_{\nu}(f, \fm)$ and $X^\iota_a(f, \fm)$ be the vectors in $\Q^{\fd(\fm)}$ whose $z$-th components are
\begin{equation*}
\begin{split}
V_{\nu}(f, \fm)_z&:= f(\fm, \fd(\fm)(\nu-1)+z)   \qa  \\
X^\iota_a(f, \fm)_z&:=f(\fm p_\iota, p_\iota^{\ve_\iota(\fm)}(\fd(\fm)(a-1)+z)),
\end{split}
\end{equation*}
respectively. 
\end{notation}

Using this notation, we state the formulas for $\Psi_\fm$.
\begin{lemma}\label{lemma: f01}
Let $[\nu]_1=\a^i_j$ and $[\nu]_2=\b_k$. Then we have
\begin{equation*}
\begin{split}
V_\nu(\Psi_\fm(f), \fm)&=V_\nu(\Phi_\fm^1(f), \fm)+\sum_{n=1}^{j-1} V_{\b_n}(\Phi_\fm^1(f), \fm)-\sum_{n=1}^k V_{\b_n}(\Phi_\fm^1(f), \fm)\\
&=V_\nu(\Phi_\fm^2(f), \fm)-V_{\a^i_j}(\Phi_\fm^2(f), \fm)-\sum_{n=1}^{j-1} V_{\a_n}(\Phi_\fm^2(f), \fm)+\sum_{n=1}^k V_{\a_n}(\Phi_\fm^2(f), \fm).
\end{split}
\end{equation*}
\end{lemma}

\begin{lemma}\label{lemma: f02}
We have
\begin{equation*}
X^1_{\a^i_j}(\Psi_\fm(f), \fm)=X^1_{\a^i_j}(f, \fm)+V_{\a^i_j}(\Phi_\fm^2(f), \fm)+\sum_{n=1}^{j-1} V_{\a_n}(\Phi_\fm^2(f), \fm).
\end{equation*}
\end{lemma}

\begin{lemma}\label{lemma: f03}
We have
\begin{equation*}
X^2_{\a_j}(\Psi_\fm(f), \fm)=X^2_{\a_j}(f, \fm)+\sum_{n=1}^{j-1} V_{\b_n}(\Phi_\fm^1(f), \fm).
\end{equation*}
\end{lemma}

\svv
We prove all the lemmas together. Before proceeding, we introduce the following notations which are only used in this section. Let
\begin{equation*}
A^n:=(\Phi_\fm^1 \circ \Phi_\fm^2)^n \circ \Phi_\fm^1(f) \qa B^n:=(\Phi_\fm^2 \circ \Phi_\fm^1)^n(f).
\end{equation*}
Also, let $1\leq z \leq \fd(\fm)$  be an integer and let
\begin{equation*}
K_1^{a, b}=\fd(\fm)(\a^a_b-1)+z \qa K_2^c=\fd(\fm)(\b_c-1)+z,
\end{equation*}
where $a, b$ and $c$ are integers following Notation \ref{notation: ell=pq}.
In this notation, we easily have the following.
\begin{enumerate}
\item
$A^0=\Phi_\fm^1(f)~$ and $~A^{p-1}=\Psi_\fm(f)$.
\item
$A^n=\Phi_\fm^1(B^n)~$ and $~B^n=\Phi_\fm^2(A^{n-1})$.
\item
$K_1^{a, b} \in I_\fm^1$ and $K_2^c \in I_\fm^2$.
\item
If $c\neq 0$, then $K_2^c \in I_\fm^{\rd}$ and $K_2^0 \not\in I_\fm^{\rd}$.
\item
If $c\neq 0$, then $\rho_\fm^1(K_2^c)=K_1^{0, c}$ as $\a_c=\b_c-q$ and $\fl_1(\fm)=q \fd(\fm)$. Similarly, $\rho_\fm^1(K_2^0)=K_1^{0, p}$.
\item
$\rho_\fm^2(K_1^{a, b})=K_2^{b-1}$ as $\a_b \equiv \b_{b-1} \pmod p$ and $\fl_2(\fm)=p \fd(\fm)$.
\end{enumerate}
For simplicity, let $\ve_1=\ve_1(\fm)$ and $\ve_2=\ve_2(\fm)$.

\svv
As in the discussion in Section \ref{subsection: 4.1}, by (2) and (3) we have
\begin{equation}\label{eqn: f3}
A^n(\fm, K_1^{a, b})=0 \qa \d(c)\cdot B^n(\fm, K_2^c)=0,
\end{equation}
where $\d(c)=0$ if $c=0$; and $\d(c)=1$ otherwise. Also, since $\Phi_\fm^\iota(g)(\fm p_\tau, h)=g(\fm p_\tau, h)$ by definition (where $\tau=3-\iota$), by (2) we have
\begin{equation}\label{eqn: f4}
\begin{split}
B^n(\fm p, p^{\ve_1}K_1^{a, b})&=A^{n-1}(\fm p, p^{\ve_1}K_1^{a, b}) \qa \\
A^n(\fm q, q^{\ve_2}K_1^{0, b})&=B^n(\fm q, q^{\ve_2}K_1^{0, b}).
\end{split}
\end{equation}

Finally, let
\begin{equation*}
h=\fd(\fm)(\nu-1)+z,
\end{equation*}
where $1\leq \nu \leq pq$ is an integer such that $[\nu]_1=\a^i_j$ and $[\nu]_2=\b_k$. Then by our notation, we have
\begin{equation}\label{eqn: f1}
\rho_\fm^1(h)=K_1^{i, j} \qa \rho_\fm^2(h)=K_2^k.
\end{equation}
Again as in the discussion in Section \ref{subsection: 4.1}, for any $g \in \scF(d)_0$ we easily have
\begin{equation}\label{eqn: f2}
\begin{split}
\Phi_\fm^1(g)(\fm, h)&=g(\fm, h)-g(\fm, K_1^{i, j}),\\
\Phi_\fm^2(g)(\fm, h)&=g(\fm, h)-\d(k)\cdot g(\fm, K_2^k),
\end{split}
\end{equation}
which are frequently used below. (Note that $h$ can be any element in $I_\fm$. On the other hand, $K_1^{a, b}$ and $K_2^c$ are elements in $T_1$ and $T_2$, respectively.)

\svv    
Now, we are ready to prove the lemmas. By \eqref{eqn: f2}, we have
\begin{equation*}
A^n(\fm, h)=\Phi_\fm^1(B^n)(\fm, h)=B^n(\fm, h)-B^n(\fm, K_1^{i, j}).
\end{equation*}
Also, we have $B^n(\fm, h)=A^{n-1}(\fm, h)-\d(k)\cdot A^{n-1}(\fm, K_2^k)$ and
\begin{equation}\label{eqn: f5}
B^n(\fm, K_1^{i, j})=-\d(j-1)\cdot A^{n-1}(\fm, K_2^{j-1}).
\end{equation}
Thus, we have
\begin{equation*}
A^n(\fm, h)-A^{n-1}(\fm, h)=\d(j-1)\cdot A^{n-1}(\fm, K_2^{j-1})-\d(k)\cdot A^{n-1}(\fm, K_2^k).
\end{equation*}
Similarly, by \eqref{eqn: f3} and \eqref{eqn: f2} we have
\[
\d(k) \cdot A^n(\fm, K_2^k)=\d(k-1)\cdot A^{n-1}(\fm, K_2^{k-1})
\]
if $k\geq 1$; and $0$ otherwise. Doing this successively, we have
\begin{equation}\label{eqn: f6}
\d(k) \cdot A^n(\fm, K_2^k)=\begin{cases}
A^0(\fm, K_2^{k-n}) & \text{ if } ~~ k>n,\\
~~0 & \text{ otherwise}.
\end{cases}
\end{equation}
Therefore we have
\begin{equation*}
\begin{split}
A^{p-1}(\fm, h)-A^0(\fm, h)&=\sum_{n=1}^{p-1} \left(\d(j-1)\cdot A^{n-1}(\fm, K_2^{j-1})-\d(k)\cdot A^{n-1}(\fm, K_2^{k})\right)\\
&=\sum_{n=1}^{j-1} A^0(\fm, K_2^n)-\sum_{n=1}^k A^0(\fm, K_2^n),
\end{split}
\end{equation*}
which proves the first equality of Lemma \ref{lemma: f01}. 

Similarly as above, we easily have 
\begin{equation*}
\begin{split}
\Phi_\fm^2(f)(\fm, h)&=f(\fm, h)-\d(k)\cdot f(\fm, K_2^{k}),\\
\Phi_\fm^2(f)(\fm, K_1^{i, j})&=f(\fm, K_1^{i, j})- \d(j-1)\cdot f(\fm, K_2^{j-1}).
\end{split}
\end{equation*}
Also, for any $1\leq n \leq p-1$ we have
\begin{equation*}
\Phi_\fm^1(f)(\fm, K_2^n)=f(\fm, K_2^n)-f(\fm, K_1^{0, n}).
\end{equation*}
Thus, for any $1\leq n \leq p-1$ we have
\begin{equation*}
\Phi_\fm^2(f)(\fm, K_1^{0, n})=-\Phi_\fm^1(f)(\fm, K_2^n)+ f(\fm, K_2^{n})- \d(n-1)\cdot f(\fm, K_2^{n-1})
\end{equation*}
and so
\begin{equation*}
\sum_{n=1}^k \Phi_\fm^2(f)(\fm, K_1^{0, n})=-\sum_{n=1}^k \Phi_\fm^1(f)(\fm, K_2^n)+\d(k)\cdot f(\fm, K_2^k).
\end{equation*}
Combining the results, we have
\begin{equation}\label{eqn: f7}
f(\fm, h)-\sum_{n=1}^{k} \Phi_\fm^1(f)(\fm, K_2^n)=\Phi_\fm^2(f)(\fm, h)+\sum_{n=1}^{k} \Phi_\fm^2(f)(\fm, K_1^{0, n})
\end{equation}
and
\begin{equation}\label{eqn: f8}
f(\fm, K_1^{i, j})-\sum_{n=1}^{j-1} \Phi_\fm^1(f)(\fm, K_2^n)=\Phi_\fm^2(f)(\fm, K_1^{i, j})+\sum_{n=1}^{j-1}  \Phi_\fm^2(f)(\fm, K_1^{0, n}).
\end{equation}
From this, we easily deduce the second equality of Lemma \ref{lemma: f01}. This completes the proof of Lemma \ref{lemma: f01}. \qed

\svv  
Next, suppose that $\nu \in T_1$. By our notation, $h=K_1^{i, j}$ for some $i$ and $j$, i.e., $\nu=\a^i_j$. Similarly as above, we have 
\begin{equation*}
A^n(\fm p, p^{\ve_1}h)=\Phi_\fm^1(B^n)(\fm p, p^{\ve_1}h)=B^{n}(\fm p, p^{\ve_1}h)+B^n(\fm, K_1^{i, j}).
\end{equation*}
By \eqref{eqn: f4}, we have $B^n(\fm p, p^{\ve_1}h)=\Phi_\fm^2(A^{n-1})(\fm p, p^{\ve_1}h)=A^{n-1}(\fm p, p^{\ve_1}h)$. So by \eqref{eqn: f5} and \eqref{eqn: f6}, we have
\begin{equation*}
A^n(\fm p, p^{\ve_1}h)-A^{n-1}(\fm p, p^{\ve_1}h)=B^n(\fm, K_1^{i, j})=\begin{cases}
-A^0(\fm, K_2^{j-n}) & \text{ if } ~~ j>n,\\
~~0 & \text{ otherwise}.
\end{cases}
\end{equation*}
As above, we have 
\begin{equation*}
\Psi_\fm(f)(\fm p, p^{\ve_1}h)=f(\fm p, p^{\ve_1}h)+f(\fm, K_1^{i, j})-\sum_{n=1}^{j-1} \Phi_\fm^1(f)(\fm, K_2^n).
\end{equation*}
Also, by \eqref{eqn: f8} we easily have
\begin{equation*}
\Psi_\fm(f)(\fm p, p^{\ve_1}h)=f(\fm p, p^{\ve_1}h)+\Phi_\fm^2(f)(\fm, K_1^{i, j})+\sum_{n=1}^{j-1} \Phi_\fm^2(f)(\fm, K_1^{0,n}).
\end{equation*}
This completes the proof of Lemma \ref{lemma: f02}. \qed

\svv    
Lastly, suppose that $\nu \in T_0$. By our notation, $h=K_1^{0, j}$ for some $j$. Similarly as above, we have
\begin{equation*}
B^n(\fm q, q^{\ve_2}h)=\Phi_\fm^2(A^{n-1})(\fm q, q^{\ve_2}h)=A^{n-1}(\fm q, q^{\ve_2}h)+\d(j-1)\cdot A^{n-1}(\fm, K_2^{j-1}).
\end{equation*}
By \eqref{eqn: f4}, we have
$A^n(\fm q, q^{\ve_2}h)=\Phi_\fm^1(B^n)(\fm q, q^{\ve_2}h)=B^n(\fm q, q^{\ve_2}h)$ and so by \eqref{eqn: f6}, we have
\begin{equation*}
A^n(\fm q, q^{\ve_2}h)-A^{n-1}(\fm q, q^{\ve_2}h)=\d(j-1)\cdot A^{n-1}(\fm, K_2^{j-1})=\begin{cases}
A^0(\fm, K_2^{j-n}) & \text{ if }~~ j>n, \\
~~0 & \text{ otherwise}.
\end{cases}
\end{equation*}
As above, this proves Lemma \ref{lemma: f03}. \qed

\vv   
The following is obvious from the discussion above, but we state it for later use.
\begin{lemma}\label{lemma: simple formula for Psi}
For an integer $\fh$, we denote by $\fh_1:=p^{-\ve_1(\fm)}\fh$ and $\fh_2:=q^{-\ve_2(\fm)}\fh$.

\begin{enumerate}
\item
If $m \not\in \{\fm, \fm p, \fm q\}$, then we have 
\begin{equation*}
\Psi_\fm(f)(m, h)=f(m, h).
\end{equation*}
Also, if $\fh_1 \not\in \Z$ (resp. $\fh_2 \not\in \Z$), then we have
\begin{equation*}
\Psi_\fm(f)(\fm p, \fh)=f(\fm p, \fh) \quad (\text{resp.} ~~~\Psi_\fm(f)(\fm q, \fh)=f(\fm q, \fh)).
\end{equation*}

\item
Suppose that $\fh \in I_{\fm p}$ and $\fh_1 \in \Z$. Then there is a set $S_1 \subset I_\fm^1$ such that
\begin{equation*}
\Psi_\fm(f)(\fm p, \fh)-f(\fm p, \fh)=\sum_{h \in S_1} \Phi_\fm^2(f)(\fm, h).
\end{equation*}
Also, for any $h \in S_1$ we have $h\equiv \fh_1\pmod {\fd(\fm)}$. Furthermore, we have
\begin{equation}\label{eqn: f001}
v_p(\ell(\fm p))-v_p(\fh)=v_p(\ell(\fm))-v_p(\fh_1)-1
\end{equation}
and
\begin{equation}\label{eqn: f002}
v_q(\ell(\fm p))=v_q(\ell(\fm)) \qa v_q(\fh)=v_q(\fh_1).
\end{equation}

\item
Suppose that $\fh \in I_{\fm q}$ and $\fh_2 \in \Z$. Then there is a set $S_2 \subset I_\fm^{\rd} \cap I_\fm^2$ such that
\begin{equation*}
\Psi_\fm(f)(\fm q, \fh)-f(\fm q, \fh)=\sum_{h \in S_{2}} \Phi_\fm^1(f)(\fm, h).
\end{equation*}
Also, for any $h \in S_2$ we have $h\equiv \fh_2\pmod {\fd(\fm)}$. Furthermore, we have
\begin{equation}\label{eqn: f003}
v_q(\ell(\fm q))-v_q(\fh)=v_q(\ell(\fm))-v_q(\fh_2)-1
\end{equation}
and
\begin{equation}\label{eqn: f004}
v_p(\ell(\fm q))=v_p(\ell(\fm)) \qa v_p(\fh)=v_p(\fh_2).
\end{equation}
\end{enumerate}
\end{lemma}
\begin{proof}
The first assertion is obvious from the definition. For (2) and (3), we use the same notation as in the proof of Lemma \ref{lemma: f01}. If we write $h=\fd(\fm)(\nu-1)+z$ as above, then for any $a, b$ and $c$ we have
\begin{equation*}
h \equiv z \equiv K_1^{a, b} \equiv K_2^c  \pmod {\fd(\fm)}.
\end{equation*}
Moreover, \eqref{eqn: f001}--\eqref{eqn: f004} follow by definition (cf. Remark \ref{remark: (m, h) and (mp, h')}).
Thus, the second (resp. third) assertion directly follows by Lemma \ref{lemma: f02} (resp. Lemma \ref{lemma: f03}).
\end{proof}

\begin{corollary}\label{corollary: vanishing on Im1}
For any $h \in I_\fm^{\rd}$, we have 
\begin{equation*}
\Psi_\fm(f)(\fm, h)=0.
\end{equation*}
\end{corollary}
\begin{proof}
If $\ell(\fm)$ is a prime power, then this is obvious. So we may assume that $\ell(\fm)$ is divisible by $pq$. 
Write $h=\fd(\fm)(\nu-1)+z$ as above.  
If $h \in I_\fm^1$, then $\nu=\a^i_j$ and $[\nu]_2=\b_{j-1}$. 
If $h \in I_\fm^2 \cap I_\fm^{\rd}$, then $\nu=\b_k$ for some $1\leq k\leq p-1$ and $[\nu]_1=\a_k$. Thus, in both cases the assertion easily follows by Lemma \ref{lemma: f01} and Remark \ref{remark: Hm1 Hm2 Im2}.
\end{proof}

\svv  
\section{Vanishing result I}\label{section5}
In this section, we prove a vanishing result for $f \in \scF(d)_\iota$ satisfying certain assumptions. As above, we have $\iota \in \{1, 2\}$. We include the case where $r_2=0$, in which case, $\cS(d)=\cS(d)_2$.

\svv  
Based on Remark \ref{remark: (m, h) and (mp, h')}, we introduce the following: For a positive integer $a$, let
\begin{equation*}
\cA_\iota(a):=\{(m, h) \in \cS(d)_\iota : v_{p_\tau}(h)=v_{p_\tau}(\ell(m))-a\} \qa \cA^+_\iota(a):=\cup_{n>a} \cA_\iota(n),
\end{equation*}
where $\tau=3-\iota$. Throughout the section, $f$ always denotes an element of $\scF(d)_\iota$.

\begin{theorem}\label{theorem: vanishing 1}
Let $a\in \Z_{\geq 1}$. Suppose that all the following hold.
\begin{enumerate}
\item
$\Psi(f)(m, h)=0$ for all $(m, h) \in \cS(d)_\iota$.
\item
$f(m, h)=0$ for all $(m, h) \in \cS(d)_\iota \sm  \cA_\iota(a)$.
\end{enumerate}
Then $f(m, h)=0$ for all $(m, h) \in \cS(d)_\iota$.
\end{theorem}

We only prove the case $\iota=2$ as the other case is proved exactly by the same method. 
For simplicity, let $\fm=dp^xq^{r_2}$ for some $0\leq x \leq r_1-1$, and let $\ve:=1+v_p(\ell(\fm p))-v_p(\ell(\fm))$. Furthermore, let 
\begin{equation*}
\begin{split}
\U(\fm)&:=\Psi_{dp^{x-1}q^{r_2}} \circ \Psi_{dp^{x-2}q^{r_2}} \circ \cdots \circ \Psi_{dq^{r_2}} \qa \\
\Psi(\fm)&:=\Psi_{\fm} \circ \U(\fm)=\Phi_{\fm}^1 \circ \U(\fm).
\end{split}
\end{equation*}

To begin with, we prove some lemmas.
\begin{lemma}\label{lemma: v0_1}
Let $a\in \Z_{\geq 1}$. Suppose that $f(m, h)=0$ for all $(m, h) \in \cA_2^+(a)$.
Then we have
\begin{equation*}
\U(\fm)(f)(\fm, \fh)=f(\fm, \fh) \quad \text{ for all }~ \fh \in I_{\fm} ~~\text{ with }~~ v_p(\fh) \leq v_p(\ell(\fm))-a.
\end{equation*}
\end{lemma}
\begin{proof}
We prove the assertion by induction on $x=v_p(\fm)$. If $x=0$, then $\U(\fm)(f)=f$ and so it is obvious. Next, suppose that the assertion holds for $x\geq 0$ and we are going to prove that the assertion holds for $\fm p$. For simplicity, let $g=\U(\fm)(f)$ and $\fh \in I_{\fm p}$. By definition, we have $\U(\fm p)(f) = \Phi_\fm^1 (g)$.  Also, by applying Lemma \ref{lemma: simple formula for Psi}(1) successively we have $g(\fm p, h)=f(\fm p, h)$. Thus, if $p^{\ve}$ does not divide $\fh$, then we have
\begin{equation*}
\U(\fm p)(f)(\fm p, \fh)=\Phi^1_\fm (g)(\fm p, \fh)=g(\fm p, \fh)=f(\fm p, \fh).
\end{equation*}
So we may assume that $h=p^{-\ve}\fh$ is an integer. Then we have
\begin{equation*}
\U(\fm p)(f)(\fm p, \fh)=\Phi^1_\fm (g)(\fm p, \fh)=g(\fm p, \fh)+g(\fm, h)=f(\fm p, \fh)+\U(\fm)(f)(\fm, h).
\end{equation*}
For simplicity, let $c=v_p(\ell(\fm p))-v_p(\fh)$, and suppose that $v_p(\fh) \leq v_p(\ell(\fm p))-a$, i.e., $c\geq a$. 
Then we have $v_p(h)=v_p(\ell(\fm))-c-1<v_p(\ell(\fm))-a$ (cf. Remark \ref{remark: (m, h) and (mp, h')}). Hence by induction hypothesis, we have 
$\U(\fm)(f)(\fm, h)=f(\fm, h)$. Since $(\fm, h) \in \cA_2^+(a)$, $f(\fm, h)=0$ by our assumption. This completes the proof by induction.
\end{proof}

\begin{lemma}\label{lemma: v0_4}
Suppose that all the following hold.
\begin{enumerate}
\item
$\Psi(f)(m, h)=0$ for all $(m, h) \in \cS(d)_2$.
\item
$f(m, h)=0$ for all $(m, h) \in \cS(d)_2$ with $m \neq \fm$.
\end{enumerate}
Then we have $f(m, h)=0$ for all $(m, h) \in \cS(d)_2$.
\end{lemma}
\begin{proof}
Note that by definition $\Psi_m(f)=f$ for any $m \in \cD(d)_2$ with $m\neq \fm$ as we assume that $f(m, h)=0$ for all $h \in I_m^\rd$.  (cf. Remark \ref{remark: identity if vanish on Im1}).
Thus, we have $\U(\fm)(f)=f$ and $\Psi(\fm)(f)=\Psi_\fm(f)$.
We prove the assertion by (reverse) induction on $x=v_p(\fm)$. If $x=r_1-1$, then $\Psi_\fm$ is the identity map and so
$\Psi(f)=\Psi(\fm)(f)=\Psi_\fm(f)=f$. Thus, the assertion obviously holds by assumption (1). 
Suppose that the assertion holds for $x+1\leq r_1-1$ and let $g=\Psi_\fm(f)$. If $v_p(m)\leq x$, then for any $h \in I_m$ we have
\begin{equation*}
g(m, h)=\Psi(\fm)(f)(m, h)=\Psi(m)(f)(m, h)=\Psi(f)(m, h)=0
\end{equation*}
(cf. Remark \ref{remark: Psi(f) independent of ordering}). Also, if $v_p(m)>x+1$, then $g(m, h)=f(m, h)=0$. 
Since $\Psi(g)=\Psi(\Psi(\fm)(f))=\Psi(f)$ (cf. Remark \ref{remark: Psi(Psi)=Psi}), $g$ also satisfies the two assumptions with $\fm$ replaced by $\fm p$ in assumption (2). Hence by induction hypothesis we have $g(m, h)=\Psi_\fm(f)(m, h)=0$ for all $(m, h) \in \cS(d)_2$. By definition, we have
\begin{equation*}
0=\Psi_\fm(f)(\fm p, p^{\ve}h)=f(\fm p, p^{\ve}h)+f(\fm, h)=f(\fm, h).
\end{equation*}
This proves that $f(\fm, h)=0$ for all $h \in I_\fm^1=I_\fm^{\rd}$. Also, we have
\begin{equation*}
0=\Psi_\fm(f)(\fm, h)=f(\fm, h)-f(\fm, \rho_\fm^1(h)).
\end{equation*}
Thus, for any $h \in I_\fm$ we have $f(\fm, h)=f(\fm, \rho_\fm^1(h))=0$. By induction, this completes the proof.
\end{proof}

Now, we are ready to prove Theorem \ref{theorem: vanishing 1}.
\begin{proof}[Proof of Theorem \ref{theorem: vanishing 1}]
By definition, it suffices to show that we have
\begin{equation*}
f(\fm, \fh)=0 \quad \text{ for all } ~~ \fh\in I_\fm ~~\text{ with } ~~v_p(\fh)=v_p(\ell(\fm))-a.
\end{equation*}
We prove the assertion under the assumption that 
\begin{equation}\tag{$\star$}
f(m, h)=0 \quad \text{ for all } ~(m, h) \in \cS(d)_2 ~~\text{ with } ~~v_p(m)>x=v_p(\fm).
\end{equation} 
We first claim that $\U(\fm)(f)$ satisfies two assumptions in Lemma \ref{lemma: v0_4}. Indeed, if $v_p(m)>x$, then $\U(\fm)(f)(m, h)=f(m, h)=0$. If $v_p(m)<x$, then 
by Remark \ref{remark: Psi(f) independent of ordering}
we have $\U(\fm)(f)(m, h)=\Psi(f)(m, h)$, which is zero by assumption (1). Thus by Lemma \ref{lemma: v0_4} we have $\U(\fm)(f)(m, h)=0$ for all $(m, h) \in \cS(d)_2$.
By assumption (2), $f$ satisfies the assumption in Lemma \ref{lemma: v0_1}.
Hence for any $\fh \in I_\fm$ with $v_p(\fh)=v_p(\ell(\fm))-a$, we have
$f(\fm, \fh)=\U(\fm)(f)(\fm, \fh)=0$.

Note that if $v_p(m)> r_1-2a$, then for any $h\in I_m$ we have $v_p(\ell(m))-v_p(h) \leq v_p(\ell(m))<a$, and so $(m, h) \not\in \cA_2(a)$. Thus, by assumption (2) we have $f(m, h)=0$ and hence we can take $x=r_1-2a$. 
By the argument above, we obtain that $f(m, h)=0$ for all $(m, h)\in \cS(d)_2$ whenever $v_p(m)\geq r_1-2a$. Now, we can also take $x=r_1-2a-1$ above and run the same argument again to deduce $f(m, h)=0$ for all $(m, h) \in \cS(d)_2$ whenever $v_p(m)\geq r_1-2a-1$. Doing this successively, the assertion follows.
\end{proof}

\svv  
\section{Vanishing result II}\label{section6}
As in the previous section, we try to prove a vanishing result for $f \in \scF(d)_0$. 
However, since the definition of $\Psi$ is not so simple any more, the corresponding statement may not be true. Instead, we prove a weaker one. 

\svv  
As above, let
\begin{equation*}
\scA_\iota(a):=\{(m, h) \in \cS(d)_0 : v_{p_\iota}(\ell(m))-v_{p_\iota}(h)=a\} \qa \scA^+_\iota(a):=\cup_{n> a} \scA_\iota(n).
\end{equation*}
Also, let $\cD_\iota(x):=\{ m \in \cD(d)_0 : v_{p_\iota}(m)=x\}$. Furthermore, let
\begin{equation*}
\scB_\iota(x):=\{(m, h) \in \cS(d)_0 : x \in \cD_\iota(x)\} \qa \scB^-_\iota(x):= \cup_{n<x} \scB_\iota(n).
\end{equation*}
Throughout the section, $f$ always denotes an element of $\scF(d)_0$.

\svv   
The main results in this section are the following.
\begin{theorem}\label{theorem: vanishing i=1}
Let $x \in \Z_{\geq 0}$ and $b \in \Z_{\geq 1}$. Suppose that all the following hold.
\begin{enumerate}
\item
$\Psi(f)(m, h)=0$ for all $(m, h) \in \cS(d)_0$.
\item
$f(m, h)=0$ for all $(m, h) \in \scB^-_1(x)$.

\item
$f(m, h)=0$ for all $(m, h) \in \scB_1(x) \sm \scA_2(b)$.
\end{enumerate}
Then for any $\fm \in \cD_1(x)$, we have
\begin{equation*}
\Phi_\fm^1 \circ \U(\fm)(f)(\fm, \fh)=\Phi_\fm^1(f)(\fm, \fh)=0 \quad \text{ for all }~ \fh \in I_\fm.
\end{equation*}
\end{theorem}

\svv    
In the case of $\iota=2$, to make our argument parallel to the above one, let
\begin{equation*}
\Phi_\fm^*(f)(m, h):=\begin{cases}
f(m, h)-f(\fm, \rho_\fm^2(h)) & \text{ if } ~~m=\fm,\\
f(m, h)+f(\fm, \rho_\fm^2(q^{-\ve_2(\fm)}h)) & \text{ if } ~~m=\fm q \qqa q^{\ve_2(\fm)} \dd h,\\
f(m, h) & \text{ otherwise}.
\end{cases}
\end{equation*}

\begin{theorem}\label{theorem: vanishing i=2}
Let $y \in \Z_{\geq 0}$ and $a\in \Z_{\geq 1}$. Suppose that all the following hold.
\begin{enumerate}
\item
$\Psi(f)(m, h)=0$ for all $(m, h) \in \cS(d)_0$.
\item
$f(m, h)=0$ for all $(m, h) \in \scB_2^-(y)$.
\item
$f(m, h)=0$ for all $(m, h) \in \scB_2(y) \sm \scA_1(a)$.
\end{enumerate}
Then for any $\fm \in \cD_2(y)$, we have
\begin{equation*}
\Phi_\fm^* \circ \U(\fm)(f)(\fm, \fh)=\Phi_\fm^*(f)(\fm, \fh)=0 \quad \text{ for all }~ \fh \in I_\fm.
\end{equation*}
\end{theorem}

\svv   
\subsection{Proof of Theorem \ref{theorem: vanishing i=1}}\label{section: vanishing i=1}
In this subsection, we use the lexicographic ordering on $\cD(d)_0$ and write
\begin{equation*}
\cD(d)_0=\{ m_1, m_2, \dots, m_{r_1r_2}\}.
\end{equation*}
Namely, $m_1=d$, $m_2=dq$, $m_3=dq^2$ and so on. (In particular, $m_i=dp^{\gauss{(i-1)/{r_2}}}q^{i-1-\gauss{(i-1)/{r_2}}r_2}$.) For simplicity, let $\fm=dp^x q^y$ for some $0\leq y \leq r_2-1$ and let $\bbO$ denote a zero vector of appropriate size. During the subsection, we assume the following.
\begin{enumerate}
\item
$f(m, h)=0$ for all $(m, h) \in \scB_1^-(x)$.
\item
$0\leq x\leq r_1-2$ unless otherwise mentioned. (When $x=r_1-1$, the arguments in Section \ref{section5} work, and indeed we get a stronger result that $f(\fm, h)=0$ for all $h \in I_\fm$.)
\end{enumerate}
Note that $\Psi_m(f)=f$ for any $m \in \cD(d)_0$ with $v_p(m)<x$ as $f(m, h)=0$ for all $h \in I_m^{\rd}$ for such an $m$ (cf.  Remark \ref{remark: identity if vanish on Im1}). Thus, we have
\begin{equation}\label{eqn: v1_1}
\U(dp^x)(f)=f,
\end{equation}
which will be frequently used below. To begin with, we prove some lemmas.

\begin{lemma}\label{lemma: v1_1}
Let $b\in \Z_{\geq 1}$ and suppose that all the following hold.
\begin{enumerate}
\item
$\Psi(f)(m, h)=0$ for all $(m, h) \in \cS(d)_0$.
\item
$f(m, h)=0$ for all $(m, h) \in \scB_1(x) \cap \scA_2^+(b)$.
\end{enumerate}
Then we have
\begin{equation*}
\U(\fm)(f)(\fm, \fh)=f(\fm, \fh) \quad \text{ for all } ~ \fh \in I_\fm ~~\text{ with } ~~v_q(\fh) \leq v_q(\ell(\fm))-b.
\end{equation*}
\end{lemma}
\begin{proof}
The proof is very similar to that of Lemma \ref{lemma: v0_1}. For convenience of the readers, we provide a complete proof.
We prove the assertion by induction on $y$. If $y=0$, then the assertion follows by \eqref{eqn: v1_1}. Suppose that the assertion holds for $y\geq 0$ and let $g=\U(\fm)(f)$. Then by induction hypothesis, we have
\begin{equation}\label{eqn: v1_2}
g(\fm, h)=f(\fm, h) \quad \text{ for all } ~ h \in I_\fm ~\text{ with }~ v_q(h) \leq v_q(\ell(\fm))-b.
\end{equation}
Also, by definition we have $g(\fm q, \fh)=f(\fm q, \fh)$ for all $\fh \in I_{\fm q}$. 
So by Lemma \ref{lemma: simple formula for Psi}(3), we have
\begin{equation*}
\U(\fm q)(f)(\fm q, \fh)=\Psi_\fm(g)(\fm q, \fh)=f(\fm q, \fh)+\sum_{h \in S_2} \Phi_\fm^1(g)(\fm, h)
\end{equation*}
for some $S_2 \subset I_\fm^{\rd} \cap I_\fm^2$. For simplicity, let $\fh \in I_{\fm q}$ and $c=v_q(\ell(\fm q))-v_q(\fh)$. 
Suppose that $v_q(\fh)\leq v_q(\ell(\fm q))-b$, i.e., $c\geq b$. To prove the assertion, it suffices to show that 
$\Phi_\fm^1(g)(\fm, h)=0$ for all $h \in S_2$. If $\fh_2=q^{-\ve_2(\fm)}\fh \not\in \Z$, then $S_2=\emptyset$ (cf. Lemma \ref{lemma: simple formula for Psi}(1)).
So we may assume that $\fh_2 \in \Z$. Then by \eqref{eqn: f003}, we have
\begin{equation*}
v_q(\ell(\fm))-v_q(\fh_2)=v_q(\ell(\fm q))-v_q(\fh)+1=c+1 \geq 2,
\end{equation*}
i.e., $v_q(\fh_2) \leq v_q(\ell(\fm))-2 < v_q(\fd(\fm)) < v_q(\fl_1(\fm))$. Since $h \equiv \fh_2 \pmod {\fd(\fm)}$, we have
\begin{equation*}
v_q(\fh_2)=v_q(h)=v_q(\rho_\fm^1(h))=v_q(\ell(\fm))-(c+1) <v_q(\ell(\fm))-b.
\end{equation*}
In other words, $(\fm, h), (\fm, \rho_\fm^1(h)) \in \scB_1(x) \cap \scA_2^+(b)$. Thus, by \eqref{eqn: v1_2} and assumption (2) we have 
\begin{equation*}
\Phi_\fm^1(g)(\fm, h)=g(\fm, h)-g(\fm, \rho_\fm^1(h))=f(\fm, h)-f(\fm, \rho_\fm^1(h))=0-0=0.
\end{equation*}
By induction, this completes the proof.
\end{proof}

\begin{lemma}\label{lemma: v1_0}
We have $\Phi_\fm^1(f)(\fm, h)=0$ for all $h \in I_\fm$ if and only if 
\begin{equation*}
U^1_1(f, \fm)=U^1_2(f, \fm)=\cdots=U^1_p(f, \fm).
\end{equation*}
\end{lemma}
\begin{proof}
This is obvious from its definition.
\end{proof}

\begin{lemma}\label{lemma: v1_2}
Suppose that $\ell(\fm)$ is divisible by $pq$. If $\Psi_\fm(f)(\fm, h)=0$ for all $h \in I_\fm$ and $V_{\b_n}(\Phi_\fm^1(f), \fm)=\bbO$ for all $1\leq n \leq p-1$, then 
we have $\Phi_\fm^1(f)(\fm, h)=0$.
\end{lemma}
\begin{proof}
Assume that $V_{\b_n}(\Phi_\fm^1(f), \fm)=\bbO$ for all $1\leq n \leq p-1$. Then by Lemma \ref{lemma: f01}, for any $1\leq \nu\leq pq$ we have $V_{\nu}(\Phi^1_\fm(f), \fm)=V_{\nu}(\Psi_\fm(f), \fm)$, which is zero by our assumption. This completes the proof.
\end{proof}

\begin{lemma}\label{lemma: v1_3}
Suppose that $\ell(\fm)$ is divisible by $pq$. Also, suppose that all the following hold.
\begin{enumerate}
\item
$f(\fm q, h)=0$ for all $h \in I_{\fm q}$.
\item
$\Phi_{\fm q}^1(\Psi_\fm(f))(\fm q, h)=0$ for all $h \in I_{\fm q}$.
\item
$\Psi_\fm(f)(\fm, h)=0$ for all $h \in I_\fm$.
\end{enumerate}
Then we have $\Phi_\fm^1(f)(\fm, h)=0$ for all $h \in I_\fm$.
\end{lemma}
\begin{proof}
By assumption (1), we have $X^2_j(f, \fm)=\bbO$ for all $1\leq j \leq p$. Also, by assumption (2) and Lemma \ref{lemma: v1_0}, we have
\begin{equation*}
U^1_1(\Psi_\fm(f), \fm q)=U^1_2(\Psi_\fm(f), \fm q)=\cdots=U^1_p(\Psi_\fm(f), \fm q).
\end{equation*}

Let $1\leq j \leq p$. If $\ve_2(\fm)=0$, i.e., $\ell(\fm)=\ell(\fm q)q$ and $\fd(\fm)=\fl_1(\fm q)$, then by definition we have 
\begin{equation*}
X_j^2(\Psi_\fm(f), \fm)=U^1_j(\Psi_\fm(f), \fm q).
\end{equation*}
If $\ve_2(\fm)=1$, i.e., $\fl_1(\fm q)=\fl_1(\fm)=\fd(\fm)q$, then we have
\begin{equation*}
X_j^2(\Psi_\fm(f), \fm)_z=U^1_j(\Psi_\fm(f), \fm q)_{qz} \quad \text{ for any } ~ 1\leq z\leq \fd(\fm).
\end{equation*}
Thus, in both cases we have
\begin{equation}\label{eqn: v1_3}
X^2_1(\Psi_\fm(f), \fm)=X^2_2(\Psi_\fm(f), \fm)= \cdots = X^2_p(\Psi_\fm(f), \fm).
\end{equation}
So by Lemma \ref{lemma: f03}, we easily have $V_{\b_k}(\Phi_\fm^1(f), \fm)=\bbO$ for all $1\leq k \leq p-1$. By assumption (3) and Lemma \ref{lemma: v1_2}, the result follows.
\end{proof}

\begin{lemma}\label{lemma: v1_4}
Suppose that all the following hold.
\begin{enumerate}
\item
$\Psi(f)(m, h)=0$ for all $(m, h) \in \cS(d)_0$.
\item
$f(m, h)=0$ for all $(m, h) \in \scB_1(x)$ with $m\neq \fm$.
\end{enumerate}
Then we have $\Phi_\fm^1(f)(\fm, h)=0$ for all $h \in I_\fm$.
\end{lemma}
\begin{proof}
The proof is very similar to that of Lemma \ref{lemma: v0_4}. For convenience of the readers, we provide a complete proof.
Note first that $\U(\fm)(f)=f$ and $\Psi(\fm)(f)=\Psi_\fm(f)$ by Remark \ref{remark: identity if vanish on Im1} (recall assumption (1) above \eqref{eqn: v1_1}).
We prove the assertion by (reverse) induction on $y$. First, let $y=r_2-1$. Then $\ell(\fm)$ is a power of $p$. Thus, $\Phi_\fm^1=\Psi_\fm$ and so
\begin{equation*}
\Phi_\fm^1(f)(\fm, h)=\Psi_\fm(f)(\fm, h)=\Psi(\fm)(f)(\fm, h)=\Psi(f)(\fm, h)=0
\end{equation*}
(cf. Remark \ref{remark: Psi(f) independent of ordering}). 

Next, let $y \leq r_2-2$ and assume the assertion holds for $y+1$. Let $g=\Psi_\fm(f)$. 
We then have $g(\fm, h)=\Psi(\fm)(f)(\fm, h)=\Psi(f)(\fm, h)=0$ for all $h \in I_\fm$. Furthermore, since $g(m, h)=\Psi_\fm(f)(m, h)=f(m, h)$ for all $(m, h) \in \scB_1(x)$ with $m \not\in \{\fm, \fm q\}$, we have $g(m, h)=0$ for all $(m, h) \in \scB_1(x)$ with $m\neq \fm q$. Thus, by induction hypothesis we have $\Phi_{\fm q}^1(g)(\fm q, h)=0$ for all $h \in I_{\fm q}$. In other words, $\Phi_{\fm q}^1(\Psi_\fm(f))(\fm q, h)=0$ for all $h \in I_{\fm q}$. 
Since $f(\fm q, h)=0$ for all $h \in I_{\fm q}$ by our assumption, the assertion for $y$ follows by Lemma \ref{lemma: v1_3}. This completes the proof.
\end{proof}

Finally, we prove a generalization of Lemma \ref{lemma: v1_3}.
\begin{lemma}\label{lemma: v1_5}
Let $b \in \Z_{\geq 1}$ and suppose that $\ell(\fm)$ is divisible by $pq$. Suppose that all the following hold.
\begin{enumerate}
\item
$f(\fm q, h)=0$ for all $h \in I_{\fm q}$ unless $v_q(h)= v_q(\ell(\fm q))-b$.
\item
$\Phi_{\fm q}^1(\Psi_\fm(f))(\fm q, h)=0$ for all $h \in I_{\fm q}$.
\item
$f(\fm, h)=0$ for all $h \in I_\fm$ with $v_q(h) < v_q(\ell(\fm))-b$.
\item
$\Psi_\fm(f)(\fm, h)=0$ for all $h \in I_\fm$.
\end{enumerate}
Then we have $\Phi_\fm^1(f)(\fm, h)=0$ for all $h \in I_\fm$.
\end{lemma}

\begin{proof}
As above, assumption (2) implies that
\begin{equation*}
X^2_1(\Psi_\fm(f), \fm)=X^2_2(\Psi_\fm(f), \fm)= \cdots = X^2_p(\Psi_\fm(f), \fm)
\end{equation*}
(cf. \eqref{eqn: v1_3}).
So by Lemma \ref{lemma: f03}, for any $1\leq k \leq p-1$ we have
\begin{equation}\label{eqn: v1_4}
V_{\b_k}(\Phi_\fm^1(f), \fm)=X^2_{\a_k}(f, \fm)-X^2_{\a_{k+1}}(f, \fm).
\end{equation}

Let $1\leq z \leq \fd(\fm)$ be an integer.
\begin{enumerate}
\item
Suppose that $v_q(z)<v_q(\ell(\fm))-b$. Let $h=\fd(\fm)(\b_k-1)+z$. Since $b\geq 1$, we have 
\begin{equation*}
v_q(z)=v_q(h)=v_q(\rho_\fm^1(h)) < v_q(\ell(\fm))-b.
\end{equation*}
By assumption (3), we have
$f(\fm, h)=f(\fm, \rho_\fm^1(h))=0$ and so 
\[
V_{\b_k}(\Phi_\fm^1(f), \fm)_z=\Phi_\fm^1(f)(\fm, h)=0.
\]

\item
Suppose that $v_q(z)\geq v_q(\ell(\fm))-b$. Let $h_j=\fd(\fm)(\a_j-1)+z$ for any $1\leq j \leq p$. Again, we have
$v_q(h_j)\geq v_q(\ell(\fm))-b$. As \eqref{eqn: f003}, we have
$v_q(\ell(\fm q))-v_q(q^{\ve_2(\fm)}h_j) \leq b-1$ and so $X^2_{\a_j}(f, \fm)_z=f(\fm q, q^{\ve_2(\fm)}h_j)$, which is zero by assumption (1). By \eqref{eqn: v1_4}, we have $V_{\b_k}(\Phi_\fm^1(f), \fm)_z=0$.
\end{enumerate}
In summary, we have $V_{\b_k}(\Phi_\fm^1(f), \fm)=\bbO$ for any $1\leq k \leq p-1$. Since we have assumption (4), the result follows by Lemma \ref{lemma: v1_2}.
\end{proof}

Now, we are ready to prove Theorem \ref{theorem: vanishing i=1}.
\begin{proof}[Proof of Theorem \ref{theorem: vanishing i=1}]
Let $b\in \Z_{\geq 1}$ and suppose that $f(m, h)=0$ for all $(m, h) \in \scB_1(x) \sm \scA_2(b)$.
For simplicity, let
\begin{equation*}
f_y:=\U(\fm)(f) \qa f_{y+1}=\Psi(\fm)(f)=\Psi_\fm(f_y).
\end{equation*}

We first prove that $\Phi_\fm^1(f_y)(\fm, h)=0$ for all $h \in I_\fm$. We prove this by (reverse) induction on $y$. If $y\geq r_2-2b$, then $f_y$ satisfies all the assumptions in Lemma \ref{lemma: v1_4}, and so the assertion follows. Next, suppose that the assertion is true for $y+1 \leq r_2-2b$, i.e.,
\begin{equation*}
\Phi_{\fm q}^1(f_{y+1})(\fm q, h)=\Phi_{\fm q}^1(\Psi_\fm(f_y))(\fm q, h)=0 \quad \text{ for any } ~ h \in I_{\fm q}.
\end{equation*}
Since $f(m, h)=0$ for all $(m, h) \in \scB_1(x) \cap \scA^+_2(b)$ and $\Psi(f)(m, h)=0$ for all $(m, h) \in \cS(d)_0$, by Lemma \ref{lemma: v1_1} we have
\begin{equation}\label{eqn: v1_6} 
f_y(\fm, h)=f(\fm, h) \quad \text{ for all $h \in I_{\fm}$ with }~ v_q(h) \leq v_q(\ell(\fm))-b.
\end{equation}
In particular, we have $f_y(\fm, h)=f(\fm, h)=0$ for all $h \in I_\fm$ with $v_q(h)<v_q(\ell(\fm))-b$.  Also, since $f_y(\fm q, h)=f(\fm q, h)$ for all $h \in I_{\fm q}$, we have $f_y(\fm q, h)=0$ unless $v_q(h)= v_q(\ell(\fm q))-b$. Since $\Psi(f_y)=\Psi(f)$,
the assertion for $y$ follows by Lemma \ref{lemma: v1_5}.

Next, we prove that $\Phi_\fm^1(f)(\fm, h)=0$ for all $h \in I_\fm$. Let $h \in I_\fm$ with $v_q(h)=v_q(\ell(\fm))-c$ for some $c$. Since $h \equiv \rho_\fm^1(h) \pmod {\fl_1(\fm)}$ and $v_q(\fl_1(\fm))=v_q(\ell(\fm))$, we have  $v_q(h)=v_q(\rho_\fm^1(h))$ if $c\geq 1$. Otherwise we have $v_q(\rho_\fm^1(h)) \geq v_q(\ell(\fm))$ and so $(\fm, \rho_\fm^1(h)) \not\in \scA_2(b)$. Therefore if $c\neq b$, then we have
\begin{equation*}
\Phi_\fm^1(f)(\fm, h)=f(\fm, h)-f(\fm, \rho_\fm^1(h))=0-0=0.
\end{equation*}
If $c=b$, then we have $(\fm, \rho_\fm^1(h)) \in \scA_2(b)$ and so by \eqref{eqn: v1_6} we have
\begin{equation*}
\Phi_\fm^1(f)(\fm, h)=f(\fm, h)-f(\fm, \rho_\fm^1(h))=f_y(\fm, h)-f_y(\fm, \rho_\fm^1(h))=\Phi_\fm^1(f_y)(\fm, h)=0.
\end{equation*}
Thus, the assertion follows.
\end{proof}

\svv   
\subsection{Proof of Theorem \ref{theorem: vanishing i=2}}\label{section: vanishing i=2}
In this subsection, we use the colexicographic ordering on $\cD(d)_0$ and write
\begin{equation*}
\cD(d)_0=\{ m_1, m_2, \dots, m_{r_1r_2}\}.
\end{equation*}
Namely, $m_1=d$, $m_2=dp$, $m_3=dp^2$ and so on. (In particular, $m_i=dp^{i-1-\gauss{(i-1)/{r_1}}r_1}q^{\gauss{(i-1)/{r_1}}}$.) As in the previous subsection, let $\fm=dp^x q^y$ for some $0\leq x \leq r_1-1$. 
Also, we assume the following.
\begin{enumerate}
\item
$f(m, h)=0$ for all $(m, h) \in \scB_2^-(y)$.
\item
$0\leq y\leq r_2-2$ unless otherwise mentioned. (As remarked in the previous subsection, the case for $y=r_1-1$ follows by the reasoning in Section \ref{section5}.)
\end{enumerate}

As above (cf. \eqref{eqn: v1_1}), we have
\begin{equation}\label{eqn: v2_1}
\U(dq^y)(f)=f.
\end{equation}
To begin with, we prove some lemmas. Although the statements and their proofs are parallel to the ones in the previous subsection, we provide detailed proofs for convenience of the readers.

\begin{lemma}\label{lemma: v2_1}
Let $a\in \Z_{\geq 1}$ and suppose all the following hold.
\begin{enumerate}
\item
$\Psi(f)(m, h)=0$ for all $(m, h) \in \cS(d)_0$.
\item
$f(m, h)=0$ for all $(m, h) \in \scB_2(y) \cap \scA^+_1(a)$.
\end{enumerate}
Then we have
\begin{equation*}
\U(\fm)(f)(\fm, \fh)=f(\fm, \fh) \quad \text{ for all } ~ \fh \in I_\fm ~~\text{ with } ~~v_p(\fh) \leq v_p(\ell(\fm))-a.
\end{equation*}
\end{lemma}
\begin{proof}
We prove the assertion by induction on $x$. If $x=0$, then it follows by \eqref{eqn: v2_1}. Suppose that the assertion holds for $x\geq 0$
and let $g=\U(\fm)(f)$. Then by induction hypothesis, we have
\begin{equation}\label{eqn: v2_2}
g(\fm, h)=f(\fm, h) \quad \text{ for all } ~ h \in I_\fm ~\text{ with }~ v_p(h) \leq v_p(\ell(\fm))-a.
\end{equation}
Also, by definition we have $g(\fm p, \fh)=f(\fm p, \fh)$ for all $\fh \in I_{\fm p}$. 
So by Lemma \ref{lemma: simple formula for Psi}(2), we have
\begin{equation*}
\U(\fm p)(f)(\fm p, \fh)=\Psi_\fm(g)(\fm p, \fh)=f(\fm p, \fh)+\sum\nolimits_{h \in S_1} \Phi_\fm^2(g)(\fm, h)
\end{equation*}
for some $S_1 \subset I_\fm^{\rd}$. For simplicity, let $\fh \in I_{\fm p}$ and $c=v_p(\ell(\fm p))-v_p(\fh)$. 
Suppose that $v_p(\fh)\leq v_p(\ell(\fm p))-a$, i.e., $c\geq a$. To prove the assertion, it suffices to show that 
$\Phi_\fm^2(g)(\fm, h)=0$ for all $h \in S_1$. As in the proof of Lemma \ref{lemma: v1_1}, we may assume that $\fh_1=p^{-\ve_1(\fm)}\fh \in \Z$. Then by \eqref{eqn: f001}, we have
\begin{equation*}
v_p(\ell(\fm))-v_p(\fh_1)=v_p(\ell(\fm p))-v_p(\fh)+1=c+1 \geq 2,
\end{equation*}
i.e., $v_p(\fh_1) \leq v_p(\ell(\fm))-2 < v_p(\fd(\fm)) < v_p(\fl_2(\fm))$. Since $h \equiv \fh_1 \pmod {\fd(\fm)}$, we have
\begin{equation*}
v_p(\fh_1)=v_p(h)=v_p(\rho_\fm^2(h)) = v_p(\ell(\fm))-(c+1) <v_p(\ell(\fm))-a.
\end{equation*}
In other words, $(\fm, h), (\fm, \rho_\fm^2(h)) \in \scB_2(y) \cap \scA_1^+(a)$. Thus, by \eqref{eqn: v2_2} and assumption (2) we have
\begin{equation*}
\Phi_\fm^2(g)(\fm, h)=g(\fm, h)-\chi_\fm^2(h)\cdot g(\fm, \rho_\fm^2(h))=f(\fm, h)-\chi_\fm^2(h)\cdot f(\fm, \rho_\fm^2(h))=0-0=0.
\end{equation*}
By induction, this completes the proof.
\end{proof}

\begin{lemma}\label{lemma: v2_0}
We have $\Phi_\fm^*(f)(\fm, h)=0$ for all $h \in I_\fm$ if and only if 
\begin{equation*}
U^2_1(f, \fm)=U^2_2(f, \fm)=\cdots=U^2_q(f, \fm).
\end{equation*}
\end{lemma}
\begin{proof}
This is obvious from its definition.
\end{proof}

\hide{\begin{lemma}\label{lemma: v2_22}
Suppose that $\ell(\fm)$ is divisible by $pq$. 
If $\Psi_\fm(f)(\fm, h)=0$ for all $h \in I_\fm$ and $V_{n}(\Phi_\fm^*(f), \fm)=\bbO$ for all $1\leq n \leq q$, then 
 $\Phi_\fm^*(f)(\fm, h)=0$ for all $h \in I_\fm$ .
\end{lemma}
\begin{proof}
Suppose that $\Psi_\fm(f)(\fm, h)=0$ for all $h \in I_\fm$ and $V_{n}(\Phi_\fm^*(f), \fm)=\bbO$ for all $1\leq n \leq q$.
For simplicity, let $V=V_{\b_0}(f, \fm)$. Also, let $[\nu]_1=\a^i_j$ and $[\nu]_2=\b_k$.
Furthermore, let $\e(k)=\max(0, 1-k)$. Then by definition, we have  $V_{\nu}(\Phi_\fm^*(f), \fm)=V_\nu(\Phi_\fm^2(f), \fm)-\e(k)\cdot V$, and so
\begin{equation*}
V_{\a^i_n}(\Phi_\fm^2(f), \fm)=V_{\a^i_n}(\Phi_\fm^*(f), \fm)+\e(n-1)\cdot V=\e(n-1)\cdot V.
\end{equation*}
Thus, for any $j$ we have 
\begin{equation*}
V_{\a^i_j}(\Phi_\fm^2(f), \fm)+\sum_{n=1}^{j-1} V_{\a_n}(\Phi_\fm^2(f), \fm)=V.
\end{equation*}

Suppose first that $k=0$. Then by Lemma \ref{lemma: f01}, we have
\begin{equation*}
V_\nu(\Phi_\fm^*(f), \fm)=V_{\nu}(\Phi_\fm^2(f), \fm)-V=V_{\nu}(\Psi_\fm(f), \fm)=\bbO.
\end{equation*}
Similarly, if $k\geq 1$, then we have
\begin{equation*}
V_\nu(\Phi_\fm^*(f), \fm)=V_\nu(\Phi_\fm^2(f), \fm)=V-\sum_{n=1}^k V_{\a_n}(\Phi_\fm^2(f), \fm)=\bbO.
\end{equation*}
This completes the proof.
\end{proof}}

\begin{lemma}\label{lemma: v2_2}
Suppose that $\ell(\fm)$ is divisible by $pq$. Suppose further that all the following hold.
\begin{enumerate}
\item
$\Psi_\fm(f)(\fm, h)=0$ for all $h \in I_\fm$.
\item
$V_{\a^i_j}(\Phi_\fm^*(f), \fm)=\bbO$ for all $\a^i_j$ unless $j=1$.
\item
$V_{\a^i_1}(f, \fm)=V_{\a_1}(f, \fm)$ for all $\a^i_1$.
\end{enumerate}
Then we have $\Phi_\fm^*(f)(\fm, h)=0$ for all $h \in I_\fm$ .
\end{lemma}
\begin{proof}
As before, let  $1\leq \nu \leq pq$ be an integer so that $[\nu]_1=\a^i_j$ and $[\nu]_2=\b_k$. For simplicity, let $V=V_{\b_0}(f, \fm)$ and $\e(k)=\max(0, 1-k)$. 
Then by definition, we have  $V_{\nu}(\Phi_\fm^*(f), \fm)=V_\nu(\Phi_\fm^2(f), \fm)-\e(k)\cdot V$. Therefore we have
$V_{\a^i_n}(\Phi_\fm^2(f), \fm)=V_{\a^i_n}(\Phi_\fm^*(f), \fm)=\bbO$ for any $2\leq n\leq p$ as $[\a^i_n]_2=\b_{n-1}$.
Also, by assumption (3) we have 
\begin{equation*}
V_{\a^i_1}(\Phi_\fm^2(f), \fm)=V_{\a^i_1}(f, \fm)\overset{(3)}{=}V_{\a_1}(f, \fm)=V_{\a_1}(\Phi_\fm^2(f), \fm)=V_{\a_1}(\Phi_\fm^*(f), \fm)+V.
\end{equation*}
Thus, for any $1\leq j \leq p$ we have 
\begin{equation*}
V_{\a^i_j}(\Phi_\fm^2(f), \fm)+\sum_{n=1}^{j-1} V_{\a_n}(\Phi_\fm^2(f), \fm)=V_{\a_1}(\Phi_\fm^2(f), \fm)=V_{\a_1}(\Phi_\fm^*(f), \fm)+V.
\end{equation*}

If $k\geq 1$, then by Lemma \ref{lemma: f01}, we have
\begin{equation*}
V_\nu(\Phi_\fm^*(f), \fm)=V_\nu(\Phi_\fm^2(f), \fm)=V_{\a_1}(\Phi_\fm^2(f), \fm)-\sum_{n=1}^k V_{\a_n}(\Phi_\fm^2(f), \fm)=\bbO.
\end{equation*}

Next, suppose that $k=0$. Similarly as above, we have
\begin{equation*}
V_\nu(\Phi_\fm^*(f), \fm)=V_{\nu}(\Phi_\fm^2(f), \fm)-V=V_{\a_1}(\Phi_\fm^2(f), \fm)-V=V_{\a_1}(\Phi_\fm^*(f), \fm).
\end{equation*}
So if we take $\nu=p+q$, which satisfies $[p+q]_2=\b_0$, then we have $V_{\nu}(\Phi_\fm^*(f), \fm)=\bbO$ by definition. This completes the proof.
\end{proof}

\begin{lemma}\label{lemma: v2_3}
Suppose that $\ell(\fm)$ is divisible by $pq$. Also, suppose that all the following hold.
\begin{enumerate}
\item
$f(\fm p, h)=0$ for all $h \in I_{\fm p}$.
\item
$\Phi_{\fm p}^*(\Psi_\fm(f))(\fm p, h)=0$ for all $h \in I_{\fm p}$.
\item
$\Psi_\fm(f)(\fm, h)=0$ for all $h \in I_\fm$.
\end{enumerate}
Then we have $\Phi_\fm^*(f)(\fm, h)=0$ for all $h \in I_\fm$.
\end{lemma}
\begin{proof}
As above, let $V=V_{\b_0}(f, \fm)$.
By assumption (1), we have $X^1_n(f, \fm)=\bbO$ for all $1\leq n \leq q$. 
Also, as before (cf. \eqref{eqn: v1_3}), by assumption (2) and Lemma \ref{lemma: v2_0} we have
\begin{equation}\label{eqn: v2_3}
X^1_1(\Psi_\fm(f), \fm)=X^1_2(\Psi_\fm(f), \fm)= \cdots = X^1_q(\Psi_\fm(f), \fm).
\end{equation}
So by Lemma \ref{lemma: f02}, we have 
\begin{equation*}
V_{\a^i_j}(\Phi_\fm^2(f), \fm)=\begin{cases}
X^1_{\a_1}(\Psi_\fm(f), \fm) & \text{ if }~~ j=1,\\
~~\bbO & \text{ otherwise}.
\end{cases}
\end{equation*}
Thus, if $j\neq 1$, then $V_{\a^i_j}(\Phi_\fm^*(f), \fm)=V_{\a^i_j}(\Phi_\fm^2(f), \fm)=\bbO$.
Otherwise, we have 
\begin{equation*}
V_{\a^i_1}(\Phi_\fm^*(f), \fm)=V_{\a^i_1}(\Phi_\fm^2(f), \fm)-V=X^1_{\a_1}(\Psi_\fm(f), \fm)-V.
\end{equation*}
Thus, for any $i$ we have 
\begin{equation*}
V_{\a^i_1}(f, \fm)=V_{\a^i_1}(\Phi_\fm^2(f), \fm)=X^1_{\a_1}(\Psi_\fm(f), \fm).
\end{equation*}
By Lemma \ref{lemma: v2_2}, the result follows.
\end{proof}

\begin{lemma}\label{lemma: v2_4}
Suppose that all the following hold.
\begin{enumerate}
\item
$\Psi(f)(m, h)=0$ for all $(m, h) \in \cS(d)_0$.
\item
$f(m, h)=0$ for all $(m, h) \in \scB_2(y)$ with $m\neq \fm$.
\end{enumerate}
Then we have $\Phi_\fm^*(f)(\fm, h)=0$ for all $h \in I_\fm$.
\end{lemma}
\begin{proof}
We prove the assertion by (reverse) induction on $x$. First, let $x=r_1-1$. Then $\ell(\fm)$ is a power of $q$. 
In this case, $I_\fm^2=I_\fm^{\rd}$ (cf. Remark \ref{remark: Hm1 Hm2 Im1}) and so $\Phi_\fm^*=\Phi_\fm^2=\Psi_\fm$ (cf. Remark \ref{remark: map Phi 2}). Therefore we have
\begin{equation*}
\Phi_\fm^*(f)(\fm, h)=\Psi_\fm(f)(\fm, h)=\Psi(\fm)(f)(\fm, h)=\Psi(f)(\fm, h)=0.
\end{equation*}
Next, let $x \leq r_1-2$ and assume the assertion holds for $x+1$. 
Let $g=\Psi_\fm(f)$. As in the proof of Lemma \ref{lemma: v1_4}, $g(m, h)=0$ for all $(m, h) \in \scB_2(y)$ with $m \neq \fm p$. Thus, by induction hypothesis we have $\Phi_{\fm p}^*(g)(\fm p, h)=0$ for all $h \in I_{\fm p}$. In other words, $\Phi_{\fm p}^*(\Psi_\fm(f))(\fm p, h)=0$ for all $h \in I_{\fm p}$. Since $f(\fm p, h)=0$ for all $h \in I_{\fm p}$ by our assumption, the assertion for $x$ follows by Lemma \ref{lemma: v2_3}. This completes the proof.
\end{proof}

Finally, we prove a generalization of Lemma \ref{lemma: v2_3}.
\begin{lemma}\label{lemma: v2_5}
Let $a\in \Z_{\geq 1}$ and suppose that $\ell(\fm)$ is divisible by $pq$. Also, suppose that all the following hold.
\begin{enumerate}
\item
$f(\fm p, h)=0$ for all $h \in I_{\fm p}$ unless $v_p(h)= v_p(\ell(\fm p))-a$.
\item
$\Phi_{\fm p}^*(\Psi_\fm(f))(\fm p, h)=0$ for all $h \in I_{\fm p}$.
\item
$f(\fm, h)=0$ for all $h \in I_\fm$ with $v_p(h) < v_p(\ell(\fm))-a$.
\item
$\Psi_\fm(f)(\fm, h)=0$ for all $h \in I_\fm$.
\end{enumerate}
Then we have $\Phi_\fm^*(f)(\fm, h)=0$ for all $h \in I_\fm$.
\end{lemma}
\begin{proof}
As \eqref{eqn: v2_3}, assumption (2) implies that
\begin{equation*}
X^1_1(\Psi_\fm(f), \fm)=X^1_2(\Psi_\fm(f), \fm)= \cdots = X^1_q(\Psi_\fm(f), \fm).
\end{equation*}
So by Lemma \ref{lemma: f02}, for any $2\leq n \leq p$ we have
\begin{equation}\label{eqn: v2_4}
V_{\a_n}(\Phi_\fm^*(f), \fm)=V_{\a_n}(\Phi_\fm^2(f), \fm)=X^1_{\a_{n-1}}(f, \fm)-X^1_{\a_{n}}(f, \fm)
\end{equation}
and for any $1\leq j \leq p$, we have
\begin{equation}\label{eqn: v2_5}
\begin{split}
V_{\a^i_j}(\Phi_\fm^*(f), \fm)-V_{\a_j}(\Phi_\fm^*(f), \fm)&=V_{\a^i_j}(\Phi_\fm^2(f), \fm)-V_{\a_j}(\Phi_\fm^2(f), \fm)\\
&=X^1_{\a_j}(f, \fm)-X^1_{\a^i_j}(f, \fm).
\end{split}
\end{equation}

Let $1\leq z \leq \fd(\fm)$ be an integer.
\begin{enumerate}
\item
Suppose that $v_p(z)<v_p(\ell(\fm))-a$. Let $h=\fd(\fm)(\a_j^i-1)+z$ for any $1\leq j \leq p$. Since $a\geq 1$, we have 
\begin{equation*}
v_p(z)=v_p(h)=v_p(\rho_\fm^2(h)) < v_p(\ell(\fm))-a.
\end{equation*}
By assumption (3), we have
$f(\fm, h)=f(\fm, \rho_\fm^2(h))=0$. Therefore $V_{\a^i_j}(f, \fm)_z=f(\fm, h)=0$ and $V_{\a_j^i}(\Phi_\fm^*(f), \fm)_z=\Phi_\fm^*(f)(\fm, h)=0$.

\item
Suppose that $v_p(z)\geq v_p(\ell(\fm))-a$. Let $h=\fd(\fm)(\a_j^i-1)+z$ for any $1\leq j \leq p$. Again, we have
$v_p(h)\geq v_p(\ell(\fm))-a$. As \eqref{eqn: f001}, we have
$v_p(\ell(\fm p))-v_p(p^{\ve_1(\fm)}h) \leq a-1$ and so $X^1_{\a_j^i}(f, \fm)_z=f(\fm p, p^{\ve_1(\fm)}h)$ = 0 by assumption (1). By \eqref{eqn: v2_4}, we have $V_{\a_n}(\Phi_\fm^*(f), \fm)_z=0$ for any $2 \leq n \leq p$. Also, by \eqref{eqn: v2_5}, we have $V_{\a_j^i}(f, \fm)_z - V_{\a_j}(f, \fm)_z = V_{\a_j^i}(\Phi_\fm^*(f), \fm)_z - V_{\a_j}(\Phi_\fm^*(f), \fm)_z = 0$ for any $1 \leq j \leq p$. 
\end{enumerate}
In summary, we have $V_{\a^i_n}(\Phi_\fm^*(f), \fm)=\bbO$ for any $2\leq n \leq p$ and $V_{\a^i_1}(f, \fm) = V_{\a_1}(f, \fm)$. Since we have assumption (4), the result follows by Lemma \ref{lemma: v2_2}.   

\hide{
\begin{enumerate}
\item
Suppose that $v_p(z)<v_p(\ell(\fm))-a$. Let $h=\fd(\fm)(\a_n-1)+z$. Since $a\geq 1$, we have 
\begin{equation*}
v_p(z)=v_p(h)=v_p(\rho_\fm^2(h)) < v_p(\ell(\fm))-a.
\end{equation*}
By assumption (3), we have
$f(\fm, h)=f(\fm, \rho_\fm^2(h))=0$ and so $V_{\a_n}(\Phi_\fm^*(f), \fm)_z=\Phi_\fm^*(f)(\fm, h)=0$.

\item
Suppose that $v_p(z)\geq v_p(\ell(\fm))-a$. Let $h_j=\fd(\fm)(\a_j-1)+z$ for any $1\leq j \leq p$. Again, we have
$v_p(h_j)\geq v_p(\ell(\fm))-a$. As \eqref{eqn: f001}, we have
$v_p(\ell(\fm p))-v_p(p^{\ve_1(\fm)}h_j) \leq a-1$ and so $X^1_{\a_j}(f, \fm)_z=f(\fm p, p^{\ve_1(\fm)}h_j)$, which is zero by assumption (1). By \eqref{eqn: v2_4}, we have $V_{\a_n}(\Phi_\fm^*(f), \fm)_z=0$.
\end{enumerate}}
\end{proof}

Now, we are ready to prove Theorem \ref{theorem: vanishing i=2}.
\begin{proof}[Proof of Theorem \ref{theorem: vanishing i=2}]
Let $a\in \Z_{\geq 1}$ and suppose that $f(m, h)=0$ for all $(m, h) \in \scB_2(y) \sm \scA_1(a)$.
For simplicity, let
\begin{equation*}
f_x:=\U(\fm)(f) \qa f_{x+1}=\Psi(\fm)(f)=\Psi_\fm(f_x).
\end{equation*}

We first prove that $\Phi_\fm^*(f_x)(\fm, h)=0$ for all $h \in I_\fm$. We prove this by (reverse) induction on $x$. If $x\geq r_1-2a$, then $f_x$ satisfies all the assumptions in Lemma \ref{lemma: v2_4}, and so the assertion follows. Next, suppose that the assertion holds for $x+1 \leq r_1-2a$, i.e.,
\begin{equation*}
\Phi_{\fm p}^*(f_{x+1})(\fm p, h)=\Phi_{\fm p}^*(\Psi_\fm(f_x))(\fm p, h)=0 \quad \text{ for any } ~ h \in I_{\fm p}.
\end{equation*}
Since $f(m, h)=0$ for all $(m, h) \in \scB_2(y) \cap \scA^+_1(a)$ and $\Psi(f)(m, h)=0$ for all $(m, h) \in \cS(d)_0$, by Lemma \ref{lemma: v2_1} we have
\begin{equation}\label{eqn: v2_6} 
f_x(\fm, h)=f(\fm, h) \quad \text{ for all $h \in I_{\fm}$ with }~ v_p(h) \leq v_p(\ell(\fm))-a.
\end{equation}
In particular, we have $f_x(\fm, h)=f(\fm, h)=0$ for all $h \in I_\fm$ with $v_p(h)<v_p(\ell(\fm))-a$.  Also, since $f_x(\fm p, h)=f(\fm p, h)$ for all $h \in I_{\fm p}$, we have $f_x(\fm p, h)=0$ unless $v_p(h)= v_p(\ell(\fm p))-a$. Since $\Psi(f_x)=\Psi(f)$,
the assertion for $x$ follows by Lemma \ref{lemma: v2_5}.

Next, we prove that $\Phi_\fm^*(f)(\fm, h)=0$ for all $h \in I_\fm$. Let $h \in I_\fm$ with $v_p(h)=v_p(\ell(\fm))-c$ for some $c$. Since $h \equiv \rho_\fm^2(h) \pmod {\fl_2(\fm)}$ and $v_p(\fl_2(\fm))=v_p(\ell(\fm))$, we have  $v_p(h)=v_p(\rho_\fm^2(h))$ if $c\geq 1$. Otherwise we have $v_p(\rho_\fm^2(h)) \geq v_p(\ell(\fm))$ and so $(\fm, \rho_\fm^2(h)) \not\in \scA_1(a)$. Therefore if $c\neq a$, then we have
\begin{equation*}
\Phi_\fm^*(f)(\fm, h)=f(\fm, h)-f(\fm, \rho_\fm^2(h))=0-0=0.
\end{equation*}
If $c=a$, then we have $(\fm, \rho_\fm^2(h)) \in \scA_1(a)$ and so by \eqref{eqn: v2_6} we have
\begin{equation*}
\Phi_\fm^*(f)(\fm, h)=f(\fm, h)-f(\fm, \rho_\fm^2(h))=f_x(\fm, h)-f_x(\fm, \rho_\fm^2(h))=\Phi_\fm^*(f_x)(\fm, h)=0.
\end{equation*}
This completes the proof.
\end{proof}

\svv   
\section{Proof of Theorem \ref{theorem A}}\label{section7}
Finally, we prove Theorem \ref{theorem A}. In other words, we show that $E(m, h) \in \Z$ for all $(m, h) \in \cS(d)$ for a fixed divisor $d$ of $M$. 
Since all the arguments below are independent with our choice of $d$, the result will follow. As already mentioned in Remark \ref{remark: Q to Q/Z}, we 
will consider $E|_{\cS(d)_i}$ and $E_s|_{\cS(d)_i}$ for each $i \in \{0, 1, 2\}$ as functions from $\cS(d)_i$ to $\Q/\Z$, and show that $E(m, h)=0$ for all $(m, h) \in \cS(d)_i$ using the fact that $\Psi(E_s)(m, h)=0$ for all $(m, h) \in \cS(d)_i$. For simplicity, for any $\fm \in \cS(d)_i$ let
\begin{equation*}
I_\fm^0:=\{ h \in I_\fm : \p(\ell(\fm)) \leq h \leq \ell(\fm)\}.
\end{equation*}
Then by definition, for any $m \in \cD(d)_i$ we have $E(m, h)=0$ for all $h \in I_m^0$. 

\svv   
\subsection{Case 1: $i=2$}\label{section: case1}
In this subsection, we prove that $E(m, h)=0$ for all $(m, h) \in \cS(d)_2$.
Before proceeding, we introduce a simple lemma from elementary number theory.
\begin{lemma}\label{lemma: elementary 1}
Let $\fm \in \cD(d)_2$ with $\ell(\fm)=p^u$. Also, let $\fh \in I_\fm \sm I_\fm^0$ with $v_p(\fh)=u-a$.
If $a\geq 2$, then there exists an integer $1\leq k \leq p-1$ such that
\begin{equation*}
(1+p^{a-1}k)\fh \equiv \rho_\fm^1(\fh)+p^{u-1}(p-1) \pmod {\ell(\fm)}.
\end{equation*}
\end{lemma}
In particular, $E(\fm, (1+p^ak)\fh)=0$. 
\begin{proof}
For simplicity, for any $n \in \Z$ prime to $p$ let $n^*$ be its multiplicative inverse modulo $p$. In other words, $1\leq n^*\leq p-1$ and $nn^* \equiv 1 \pmod p$. Let $j=p^{1-u}(\fh - \rho_\fm^1(\fh))$ and let $\fh=p^{u-a}n$ for some integer $n$ prime to $p$. 
By direct computation, we can take $k$ as the remainder of $n^*(p-1-j)$ modulo $p$.
Since $\fh \in I_\fm \sm I_\fm^0$, we have $0\leq j \leq p-2$ and so $k\neq 0$. This completes the proof.
\end{proof}
 
For simplicity, let $dq^{r_2}=D$, $r_1=r$ and $t=\gauss{r/2}$. Then by definition, $\ell(D)=p^t$ and 
\begin{equation*}
\cS(d)_2=\mcoprod\nolimits_{0\leq a\leq t} \cA_2(a).
\end{equation*}
Since $\cA_2(0)=\{(m, \ell(m)) : m \in \cD(d)_2\}$, by definition we have $E(m, h)=0$ for all $(m, h)\in \cA_2(0)$. Thus, it suffices to show that
$E(m, h)=0$ for all $(m, h) \in \cA_2(a)$ for any $1\leq a \leq t$.
We proceed as follows.
\begin{enumerate}
\item
First, we prove that $E(m, h)=0$ for all $(m, h) \in \cA_2(t)$ if $t\geq 2$.
\item
Next, let $a\in \Z_{\geq 2}$ and assume that $E(m, h)=0$ for all $(m, h) \in \cA_2^+(a)$.
Then we prove that $E(m, h)=0$ for all $(m, h) \in \cA_2(a)$.
\item
Finally, we prove that $E(m, h)=0$ for all $(m, h) \in \cA_2(1)$.
\end{enumerate}

First, we take $s=1+p^{t-1}k$ for some $1\leq k\leq p-1$. Then $sh \equiv h \pmod {\ell(m)}$ unless $v_p(\ell(m))=t$ and $p \nmid h$, in which case $(m, h) \in \cA_2(t)$. Thus, $E_s(m, h)=0$ for all $(m, h) \in \cS(d)_2 \sm \cA_2(t)$.
By Theorem \ref{theorem: vanishing 1}, we have $E_s(m, h)=0$ for all $(m, h) \in \cS(d)_2$. 
Let $(\fm, \fh) \in \cA_2(t)$. 
If $\fh\in I_\fm^0$, then $E(\fm, \fh)=0$ by definition. So we may assume that $\fh \in I_\fm \sm I_\fm^0$.
Then by Lemma \ref{lemma: elementary 1}, we can find $k$ such that $E(\fm, s\fh)=0$. Thus, we get $E(\fm, \fh)=E(\fm, s\fh)-E_s(\fm, \fh)=0$. This proves the first claim.

Next, let $a\in \Z_{\geq 2}$ and assume that $E(m, h)=0$ for all $(m, h) \in \cA_2^+(a)$. Similarly as above, we take $s=1+p^{a-1}k$ for some $1\leq k \leq p-1$. Then $sh\equiv h \pmod {\ell(m)}$ unless $v_p(h) \leq v_p(\ell(m))-a$.
In other words, $E_s(m, h)=0$ for all $(m, h) \in \cS(d)_2 \sm \cA_2^+(a-1)$. 
Also, for any $(m, h) \in \cA_2^+(a)$ we have $(m, sh) \in \cA_2^+(a)$ as $(s, p)=1$, and hence $E_s(m, h)=0$ for all $(m, h) \in \cA_2^+(a)$.
Thus, by Theorem \ref{theorem: vanishing 1} we have
$E_s(m, h)=0$ for all $(m, h) \in \cS(d)_2$. As above, we can deduce $E(\fm, \fh)=0$ for all $(\fm, \fh) \in \cA_2(a)$ by Lemma \ref{lemma: elementary 1}.

Finally, by induction we conclude that $E(m, h)=0$ for all $(m, h)\in \cA_2^+(1)$. If $(s, p)=1$, we have $E_s(m, h)=0$ for all $(m, h) \in \cS(d)_2 \sm \cA_2(1)$. Again by Theorem \ref{theorem: vanishing 1}, for any $s \in \mm p$ we have $E_s(m, h)=0$ for all $(m, h) \in \cS(d)_2$. Let $(\fm, \fh) \in \cA_2(1)$ and let $v_p(\ell(\fm))=u\geq 1$. By definition, $\fh=p^{u-1}j$ for some $1\leq j \leq p-1$. By taking $s=(-j)^*$, we have
\begin{equation*}
E(\fm, \fh)=E(\fm, s\fh)-E_s(\fm, \fh)=E(\fm, p^{u-1}(p-1))-E_s(\fm, \fh)=0-0=0.
\end{equation*}
This completes the proof.
\qed

\svv   
\subsection{Case 2: $i=1$}\label{section: case2}
Exactly as in the previous subsection, we have $E(m, h)=0$ for all $(m, h) \in \cS(d)_1$. We leave the details to the readers.

\svv  
\subsection{Case 3: $i=0$}\label{section: case3}
In this subsection, we prove that $E(m, h)=0$ for all $(m, h) \in \cS(d)_0$. 
For simplicity, let
\begin{equation*}
\begin{split}
D_\iota&:=\{ m \in \cD(d)_0: v_{p_\iota}(m)=r_\iota-1\} \qa \\
D_0&:=\{ m \in \cD(d)_0 : p_1p_2 \dd \ell(\fm)\}=\cD(d)_0 \sm (D_1 \cup D_2).
\end{split}
\end{equation*}
Also, let 
\[
A_i:=\{(m, h) \in \cS(d)_0 : m \in D_i\}.
\]
In the rest of the paper, we prove that $E(m, h)=0$ for all $(m, h) \in A_0$. Then by the same argument as in Section \ref{section: case1} (resp. \ref{section: case2}), we obtain that $E(m, h)=0$ for all $(m, h) \in A_2$ (resp. $A_1$). Thus, this will finish the proof of Theorem \ref{theorem A}.

\svv   
By definition, if $m \in D_0$, then $\ell(m)$ is divisible by $pq$. So we henceforth assume that $\ell(m)$ is divisible by $pq$ unless otherwise mentioned. To prove our assertion, we proceed as follows.

\begin{theorem}\label{thm: final 1}
Let $a \in \Z_{\geq 2}$. Suppose that $E(m, h)=0$ for all $(m, h) \in A_0\cap \scA_1^+(a)$. Then we have
\begin{equation*}
E(m, h)=0 \quad \text{ for all }~~(m, h) \in A_0\cap \scA_1(a).
\end{equation*}
\end{theorem}

\begin{theorem}\label{thm: final 2}
Let $b \in \Z_{\geq 2}$. Suppose that $E(m, h)=0$ for all $(m, h) \in A_0 \cap \scA_2^+(b)$. Then we have
\begin{equation*}
E(m, h)=0 \quad \text{ for all }~~(m, h) \in A_0 \cap \scA_2(b).
\end{equation*}
\end{theorem}

\begin{theorem}\label{thm: final 3}
Suppose that $E(m, h)=0$ for all $(m, h) \in A_0 \cap (\scA_1^+(1) \cup \scA_2^+(1))$. Then we have
\begin{equation*}
E(m, h)=0 \quad \text{ for all }~~(m, h) \in A_0.
\end{equation*}
\end{theorem}

\svv  
Before proceeding, we introduce some further notation: For any $\fm \in D_0$, let 
\begin{equation*}
J_\fm^\iota:=\{ h \in I_\fm :  z(\fm) < h \leq z(\fm)+\fl_{\iota}(\fm)\},
\end{equation*}
where
\begin{equation*}
z(\fm):=\ell(\fm)-\fl_1(\fm)-\fl_2(\fm)=\fd(\fm)(pq-p-q).
\end{equation*}

Also, for any $h \in I_\fm$ let $\br{h}_\iota$ be an element of $J_\fm^\iota$ such that $h \equiv \br{h}_\iota \pmod {\fl_\iota(\fm)}$, and let
\begin{equation*}
K_\fm^\iota(h):=\{ h' \in I_\fm : h' \equiv h \pmod {\fl_\iota(\fm)} \}=\{ R_\fm(h+k \cdot \fl_\iota(\fm)) : 0\leq k < p_\iota\},
\end{equation*}
where $R_\fm(a)$ denotes the smallest positive integer congruent to $a$ modulo $\ell(\fm)$, i.e., for any $a\in \Z$ we have $R_\fm(a) \in I_\fm$
and $R_\fm(a) \equiv a \pmod {\ell(\fm)}$. For $\iota \in \{1, 2\}$, let $t_\iota = \gauss {r_\iota/2}$.

As before, we introduce a simple lemma. 
\begin{lemma}\label{lemma: elementary 2}
Let $\fm \in D_0$ with $v_{p_\iota}(\ell(\fm))=u_\iota$ and let $\fh \in I_\fm$ with $v_{p_\iota}(\fh)=u_\iota-a_\iota$. 
\begin{enumerate}
\item
If $a_1\geq 2$, then there is a unique integer $h_1 \in K_\fm^1(\fh)$ such that $\br{h_1}_2 \not\in I_\fm^0$ and $\br{h}_2 \in I_\fm^0$ for all $h \in K_\fm^1(\fh)$ different from $h_1$. Let
$\fa_k=R_\fm(h_1+k \cdot \fl_1(\fm)) \in K_\fm^1(\fh)$. Then for any $1\leq k \leq p-1$, there is an integer $1\leq n\leq p-1$ such that
\begin{equation*}
(1+p^{a_1-1}q^{t_2}n)\fa_k \equiv \fa_{k+1} \pmod {\ell(\fm)}.
\end{equation*}
Moreover, for such an $n$ we have
\begin{equation*}
(1+p^{a_1-1}q^{t_2}n)\br{\fa_k}_2 \equiv \br{\fa_k}_2+\fl_1(\fm) \pmod {\ell(\fm)}.
\end{equation*}

\item
If $a_2\geq 2$, then there is a unique integer $h_2 \in K_\fm^2(\fh)$ such that $\br{h_2}_1 \not\in I_\fm^0$ and $\br{h}_1 \in I_\fm^0$ for all $h \in K_\fm^2(\fh)$ different from $h_2$. Let
$\fb_k=R_\fm(h_2+k \cdot \fl_2(\fm)) \in K_\fm^2(\fh)$. Then for any $1\leq k \leq q-1$, there is an integer $1\leq n\leq q-1$ such that
\begin{equation*}
(1+p^{t_1}q^{a_2-1}n)\fb_k \equiv \fb_{k+1} \pmod {\ell(\fm)}.
\end{equation*}
Moreover, for such an $n$ we have
\begin{equation*}
(1+p^{t_1}q^{a_2-1}n)\br{\fb_k}_1 \equiv \br{\fb_k}_1+\fl_2(\fm)\pmod {\ell(\fm)}.
\end{equation*}

\item
Suppose that $a_1=1$ and $a_2\leq 1$. Then there is a unique integer $h_1 \in K_\fm^1(\fh)$ such that $v_p(h_1)\geq u_1$ and $v_p(h)=u_1-1$ for all $h \in K_\fm^1(\fh)$ different from $h_1$. Let
$\fa_k=R_\fm(h_1+k \cdot \fl_1(\fm)) \in K_\fm^1(\fh)$. Then for any $1\leq k \leq p-2$, the following hold.
\begin{itemize}[--]
\item
We have $\br{\fa_k}_2, \br{\fa_k}_2+\fl_1(\fm) \in I_\fm^0$.
\item
There is an integer $1\leq n\leq p-1$ such that 
\begin{equation*}
(1+q^{t_2}n)\fa_k \equiv \fa_{k+1} \pmod {\ell(\fm)}
\end{equation*}
and $(1+q^{t_2}n)$ is divisible by neither $p$ nor $q$. Moreover, for such an $n$ we have
\begin{equation*}
(1+q^{t_2}n)\br{\fa_k}_2 \equiv \br{\fa_{k}}_2+\fl_1(\fm) \pmod {\ell(\fm)}.
\end{equation*}
\end{itemize}
\item
Suppose that $a_1\leq 1$ and $a_2=1$. Then there is a unique integer $h_2 \in K_\fm^2(\fh)$ such that $v_q(h_2)\geq u_2$ and $v_q(h)=u_2-1$ for all $h \in K_\fm^2(\fh)$ different from $h_2$. Let
$\fb_k=R_\fm(h_2+k \cdot \fl_2(\fm)) \in K_\fm^2(\fh)$. Then for any $1\leq k \leq q-2$, the following hold.
\begin{itemize}[--]
\item
We have $\br{\fb_k}_1, \br{\fb_k}_1+\fl_2(\fm) \in I_\fm^0$.
\item
There is an integer $1\leq n\leq q-1$ such that 
\begin{equation*}
(1+p^{t_1}n)\fb_k \equiv \fb_{k+1} \pmod {\ell(\fm)}
\end{equation*}
and $(1+p^{t_1}n)$ is divisible by neither $p$ nor $q$.  Moreover, for such an $n$ we have
\begin{equation*}
(1+p^{t_1}n)\br{\fb_k}_1 \equiv \br{\fb_{k}}_1+\fl_2(\fm) \pmod {\ell(\fm)}.
\end{equation*}
\end{itemize}
\end{enumerate}
\end{lemma}

\begin{proof}
As before, for any $n \in \Z$ prime to $p$ let $n^*$ be its multiplicative inverse modulo $p$. Note that 
\begin{equation*}
J_\fm^1 \sm I_\fm^0=J_\fm^2 \sm I_\fm^0=\{ h \in I_\fm : z(\fm) < h < z(\fm)+\fd(\fm)\}.
\end{equation*}
\begin{enumerate}
\item
Suppose that $a_1\geq 2$. Then $\fh$ is indivisible by $\fd(\fm)$ and so is $\br{h}_2$ for any $h \in K_\fm^1(\fh)$. Let
\begin{equation*}
J_k=\{ h \in I_\fm : z(\fm)+(k-1)\fd(\fm) < h < z(\fm)+k\cdot \fd(\fm)\}.
\end{equation*}\
Then any element of $J_\fm^2 \sm (\cup_{k\geq 0} J_k)$ is divisible by $\fd(\fm)$. So for any $n \in J_\fm^2$ not divisible by $\fd(\fm)$, we can find $1\leq k \leq p$ such that $n \in J_k$.
In such a case, let $\wp(n):=k$. Since $p$ and $q$ are relatively prime, for two distinct elements $h$ and $h'$ of $K_\fm^1(\fh)$
we easily have $\wp(\br{h}_2) \neq \wp(\br{h'}_2)$. Thus, there is a unique element $h_1 \in K_\fm^1(\fh)$ such that $\wp(\br{h_1}_2) \in J_1$,
which is equal to $J_\fm^2 \sm I_\fm^0$. This completes the first assertion.

Next, since $v_p(\fh)=u_1-a_1$ and $a_1\geq 2$, we have $v_p(h_1)=u_1-a_1$ as well, and so we can write $h_1=p^{u_1-a_1}c$ for some $c \in \Z$ prime to $p$.
Let $n=(q^{t_2-u_2}c)^*$. Then we have
\begin{equation*}
(1+p^{a_1-1}q^{t_2}n)(h_1+k\cdot \fl_1(\fm)) \equiv \fa_k+\fl_1(\fm) \cdot q^{t_2-u_2}cn \equiv \fa_k+\fl_1(\fm) \pmod {\ell(\fm)}
\end{equation*}
as $p^{a-1}\cdot \fl_1(\fm)$ is divisible by $\ell(\fm)$. 

Lastly, let $s=1+p^{a_1-1}q^{t_2}n$. Since $\fa_k \equiv \br{\fa_k}_2 \pmod {\fl_2(\fm)}$ and $s-1$ is divisible by $q$, we have
\begin{equation}\label{eqn: final}
(s-1)(\fa_k-\br{\fa_k}_2) \equiv \fl_1(\fm)+\br{\fa_k}_2-s\br{\fa_k}_2 \equiv 0 \pmod {\ell(\fm)}.
\end{equation}
This completes the proof.

\item
By the same argument as (1), the assertion follows.

\item
Suppose that $a_1= 1$ and $a_2 \leq 1$. 
By our assumption, $\fh$ is divisible by $\fd(\fm)$ and so we write $\fh=\fd(\fm)\nu$ for some integer $1\leq \nu \leq pq$, which is
prime to $p$. Also, let $c$ be the remainder of $-q^*\nu$ modulo $p$.
Since $\nu$ is prime to $p$, we have $1\leq c \leq p-1$. By direct computation we have
\begin{equation*}
\fh+c\cdot \fl_1(\fm)=\fd(\fm)(\nu+qc) = \fd(\fm)(pe) \quad \text{ for some } e \in \Z.
\end{equation*}
Thus, we can take $h_1=\fh+c \cdot \fl_1(\fm)$. It is straightforward that $v_p(h)=u_1-1$ for all other $h \in K_\fm^1(\fh)$. 

Next, note that 
\begin{equation*}
\fh\equiv \fa_k \equiv \br{\fa_k}_2 \equiv 0 \pmod {\fd(\fm)}.
\end{equation*}
Since there is no element of $J_\fm^2 \sm I_\fm^0$ divisible by $\fd(\fm)$, we have $\br{\fa_k}_2 \in I_\fm^0$. 
Moreover, since $\br{\a_k}_2 \leq \ell(\fm)-\fl_1(\fm)$, we have $\br{\a_k}_2+\fl_1(\fm) \in I_\fm^0$.

Lastly, let $n=(q^{t_2}k)^*$. As above, we have $1\leq n \leq p-1$ and 
\begin{equation*}
(1+q^{t_2}n)\fa_k \equiv (1+q^{t_2}n)(h_1+k\cdot \fl_1(\fm)) \equiv \fa_k+q^{t_2} kn \cdot \fl_1(\fm) \equiv \fa_k+\fl_1(\fm) \pmod {\ell(\fm)}
\end{equation*}
as $q^{t_2}h_1$ is divisible by $\ell(\fm)$. It is obvious that $1+q^{t_2}n$ is prime to $q$. Also, since $v_p(\fa_k)=v_p(\fa_{k+1})=u_1-1$, $1+q^{t_2}n$ must be indivisible by $p$. The last assertion easily follows by the same argument as \eqref{eqn: final}.
This completes the proof.
\item
By the same argument as (3), the assertion follows. \qedhere
\end{enumerate}
\end{proof}

\svv   
Now, we prove the theorems above.
\begin{proof}[Proof of Theorem \ref{thm: final 1}]
Let $a\in \Z_{\geq 2}$ and suppose that $E(m, h)=0$ for all $(m, h) \in A_0 \cap \scA_1^+(a)$. 
Before proceeding, we note the following. As usual, let $s$ be an integer prime to $L$, where $L=p^{t_1}q^{t_2}$.
\begin{enumerate}
\item
Since $\gcd(s, pq)=1$, we have $v_{p_\iota}(h)=v_{p_\iota}(sh)$ for any $h \in I_m$. Thus by our assumption, for any $(m, h) \in A_0 \cap \scA_1^+(a)$ we have
\begin{equation*}
E_s(m, h)=E(m, sh)-E(m, h)=0-0=0.
\end{equation*}
\item
If $q^{t_2} \dd (s-1)$, then $sh \equiv h \pmod {\ell(m)}$ for all $(m, h) \in A_1$.
\end{enumerate}

Let $s=1+p^{a-1}q^{t_2}j$ for some $1\leq j \leq p-1$. Then by our choice of $s$, we have $sh\equiv h \pmod {\ell(m)}$ for all $(m, h) \in A_0 \sm \scA_1^+(a-1)$. 
Thus, by (1) and (2) above we have $E_s(m, h)=0$ for all $(m, h) \in (A_0 \sm \scA_1(a)) \cup A_1$. Hence $E_s$ satisfies all the assumptions in Theorem \ref{theorem: vanishing i=2} with $y=0$. Thus, for any $\fm \in \cD_2(0)$ we have 
$\Phi_\fm^*(E_s)(\fm, h)=0$ for all $h\in I_\fm$. Take $\fh \in I_\fm$ such that $v_p(\fh)=v_p(\ell(\fm))-a$. Then by Lemma \ref{lemma: elementary 2}(1) we can write $K_\fm^1(\fh)=\{\fa_1, \fa_2, \dots, \fa_p\}$. Also, by taking $j=n$ as in Lemma \ref{lemma: elementary 2}(1), we have 
\begin{equation*}
E_s(\fm, \fa_k)=E_s(\fm, \br{\fa_k}_2)=E(\fm, \br{\fa_k}_2+\fl_1(\fm))-E(\fm, \br{\fa_k}_2),
\end{equation*}
where the first equality follows from $\Phi_\fm^*(E_s)(m, h)=0$ (cf. Lemma \ref{lemma: v2_0}). Since $\br{\fa_k}_2, \br{\fa_k}_2+\fl_1(\fm) \in I_\fm^0$ if $1\leq k \leq p-1$, we have $E_s(\fm, \fa_k)=E(\fm, \fa_{k+1})-E(\fm, \fa_k)=0$. Thus, we have
\begin{equation*}
E(\fm, \fh)=E(\fm, \fa_1)=E(\fm, \fa_2)=\cdots = E(\fm, \fa_p).
\end{equation*}
Since $K_\fm^1(h) \cap I_\fm^0 \neq \emptyset$, we have $E(\fm, \fh)=0$. This shows that $E(m, h)=0$ for all $(m, h) \in \scB_2(0) \cap \scA_1(a)$. Thus, similarly as (1) above, for any $s$ prime to $L$ we have $E_s(m, h)=0$ for all $(m, h) \in \scB_2(0) \cap \scA_1(a)$. 

Next, we take $s$ as above, i.e., $s=1+p^{a-1}q^{t_2}j$ for some $1\leq j \leq p-1$.
Since $E_s(m, h)=0$ for all $(m, h) \in \scB_2(0) \cap \scA_1(a)$, $E_s$ satisfies all 
the assumptions in Theorem \ref{theorem: vanishing i=2} with $y=1$. By the same argument as above, we can easily prove that $E(m, h)=0$ for all $(m, h) \in \scB_2(1) \cap \scA_1(a)$.

Doing this successively, the result follows.
\end{proof}

By the same argument as above, we can easily prove Theorem \ref{thm: final 2}. We leave the details to the readers. Finally, we prove Theorem \ref{thm: final 3}.
\begin{proof}[Proof of Theorem \ref{thm: final 3}]
We use the same argument as in the proof of Theorem \ref{thm: final 1}.

By our assumption, for any $m \in D_0$ it suffices to show that $E(m, h)=0$ for all $h \in I_m$ divisible by $\fd(m)$.
First, we take $s=1+p^{t_1}j$ for some $1\leq j\leq q-1$ so that $(s, q)=1$. As above, by Theorem \ref{theorem: vanishing i=1} with $x=0$ and $b=1$, we have
$\Phi_m^1(E_s)(m, h)=0$ for any $(m, h) \in \scB_1(0)$. 
Let $\fm \in \cD_1(0)$ and let $\fh \in I_\fm$ with $v_p(\fh)\geq v_p(\ell(\fm))-1$ and $v_q(\fh)=v_q(\ell(\fm))-1$.
Then by Lemma \ref{lemma: elementary 2}(4), we have (as in the proof of Theorem \ref{thm: final 1})
\begin{equation*}
E(\fm, \fh)=E(\fm, \fb_1)=E(\fm, \fb_2)=\cdots=E(\fm, \fb_{q-1}). 
\end{equation*}
Since the length of $I_\fm^0$ is larger than $2\cdot \fl_2(\fm)$, there is an integer $1\leq k \leq q-1$ such that $\fb_k \in I_\fm^0$. 
Thus, we have $E(\fm, \fh)=E(\fm, \fb_k)=0$. This proves that $E(m, h)=0$ for all $(m, h) \in \scB_1(0) \cap \scA_2(1)$. 
Next, as above we have $E_s(m, h)=0$ for all $(m, h) \in \scB_1(0)$. So it satisfies all the assumptions in Theorem \ref{theorem: vanishing i=1} with $x=1$ and $b=1$. 
As above, we can easily prove that $E(m, h)=0$ for all $(m, h) \in \scB_1(1) \cap \scA_2(1)$.
Doing this successively, we prove that $E(m, h)=0$ for all $(m, h) \in A_0 \cap \scA_2(1)$. 

Finally, let $\scA:=A_0 \cap \scA_1(1) \cap \scA_2(0)$. Then by definition, we have
\begin{equation*}
\scA=\{(m, k \cdot \fl_1(m)) \in A_0 : 1\leq k \leq p-1\}
\end{equation*}
and
\begin{equation*}
A_0 \sm (\scA_2^+(0) \cup \scA_1^+(1))= \scA \cup \{(m, \ell(m)) : m \in D_0\}.
\end{equation*}
Since $E(m, \ell(m))=0$ for any $m \in D_0$ by definition, it suffices to show that $E(m, h)=0$ for all $(m, h) \in \scA$.

As above, we take $s=1+q^{t_2}j$ for some $1\leq j \leq p-1$ so that $(s, p)=1$. By Theorem \ref{theorem: vanishing i=2} with $y=0$ and $a=1$, we then have $\Phi_m^*(E_s)(m, h)=0$ for any $(m, h) \in \scB_2(0)$. 
Let $\fm \in \cD_2(0)$ and let $\fh \in I_\fm$ with $v_p(\fh)=v_p(\ell(\fm))-1$ and $v_q(\fh)=v_q(\ell(\fm))$.
Then by direct computation, we have $K_\fm^1(\fh)=\{ k\cdot \fl_1(\fm) : 1\leq k \leq p\}$ and so $h_1=\ell(\fm)$ and $\fa_{p-1}=\ell(\fm)-\fl_1(\fm) \in I_\fm^0$ as in the notation of Lemma \ref{lemma: elementary 2}(3).
Thus, by Lemma \ref{lemma: elementary 2}(3) we have (as in the proof of Theorem \ref{thm: final 1})
\begin{equation*}
E(\fm, \fh)=E(\fm, \fa_1)=E(\fm, \fa_2)=\cdots=E(\fm, \fa_{p-1})=0.
\end{equation*}
This shows that $E(m, h)=0$ for all $(m, h) \in \scB_2(0) \cap \scA$.
As above, we easily prove that $E(m, h)=0$ for all $(m, h) \in \scA$. 
This completes the proof.
\end{proof}

\vv

\end{document}